\documentclass[11pt]{article}

\usepackage{amsmath}
\usepackage{amsthm}

\usepackage{amssymb}
\usepackage{latexsym}
\usepackage{mathtools}

\usepackage{caption, subcaption}%
\usepackage[font=scriptsize]{caption}

\usepackage{xparse} 
\usepackage{stmaryrd} 

\usepackage{mathrsfs}

\usepackage{multirow}

\usepackage{geometry}
\usepackage{bm}

\usepackage{enumitem}

\usepackage{color}
\usepackage{graphicx}
\usepackage{amsfonts}
\usepackage{epstopdf}

\makeatletter
\DeclareFontFamily{OMX}{MnSymbolE}{}
\DeclareSymbolFont{MnLargeSymbols}{OMX}{MnSymbolE}{m}{n}
\SetSymbolFont{MnLargeSymbols}{bold}{OMX}{MnSymbolE}{b}{n}
\DeclareFontShape{OMX}{MnSymbolE}{m}{n}{
    <-6>  MnSymbolE5
   <6-7>  MnSymbolE6
   <7-8>  MnSymbolE7
   <8-9>  MnSymbolE8
   <9-10> MnSymbolE9
  <10-12> MnSymbolE10
  <12->   MnSymbolE12
}{}
\DeclareFontShape{OMX}{MnSymbolE}{b}{n}{
    <-6>  MnSymbolE-Bold5
   <6-7>  MnSymbolE-Bold6
   <7-8>  MnSymbolE-Bold7
   <8-9>  MnSymbolE-Bold8
   <9-10> MnSymbolE-Bold9
  <10-12> MnSymbolE-Bold10
  <12->   MnSymbolE-Bold12
}{}

\let\llangle\@undefined
\let\rrangle\@undefined
\DeclareMathDelimiter{\llangle}{\mathopen}%
                     {MnLargeSymbols}{'164}{MnLargeSymbols}{'164}
\DeclareMathDelimiter{\rrangle}{\mathclose}%
                     {MnLargeSymbols}{'171}{MnLargeSymbols}{'171}

\makeatletter
\def\widebreve{\mathpalette\wide@breve}
\def\wide@breve#1#2{\sbox\z@{$#1#2$}%
	\mathop{\vbox{\m@th\ialign{##\crcr
				\kern0.08em\brevefill#1{0.8\wd\z@}\crcr\noalign{\nointerlineskip}%
				$\hss#1#2\hss$\crcr}}}\limits}
\def\brevefill#1#2{$\m@th\sbox\tw@{$#1($}%
	\hss\resizebox{#2}{\wd\tw@}{\rotatebox[origin=c]{90}{\upshape(}}\hss$}
\makeatletter

\usepackage[colorlinks=true]{hyperref}

\allowdisplaybreaks[3]
\newtheorem{theorem}{Theorem}

\newtheorem{definition}{Definition}
\newtheorem{remark}{Remark}

\theoremstyle{remark}
\newtheorem{remark1}{\bf Remark}[]
\newtheorem{example}{\bf Example}[]

\allowdisplaybreaks

\hypersetup{
colorlinks=true,
linkcolor=blue
}

\definecolor{newcolor}{rgb}{.8,.349,.1}

\usepackage[mathlines]{lineno}

\providecommand{\keywords}[1]
{
	\small	
	\textbf{\textit{Keywords:}} #1
}

\title{{\large \bf  {High-order Accurate Entropy Stable Schemes for Relativistic Hydrodynamics with General Synge-type Equation of State}}\footnote{This work is partially supported by Shenzhen Science and Technology Program (Grant No.~RCJC20221008092757098) and National Natural Science Foundation of China (Grant No.~12171227).}} 

\author{
    Linfeng Xu \thanks{Department of Mathematics, Southern University of Science and Technology, Shenzhen 518055, P.R.~China.  ({\tt xulf2022@mail.sustech.edu.cn}). },
    Shengrong Ding \thanks{Shenzhen International Center for Mathematics and Department of Mathematics, Southern University of Science and Technology, Shenzhen 518055, P.R.~China.  ({\tt dingsr@sustech.edu.cn}). },
	Kailiang Wu\thanks{Department of Mathematics and Shenzhen International Center for Mathematics, Southern University of Science and Technology, and National Center for Applied Mathematics Shenzhen (NCAMS), Shenzhen 518055, China. ({\tt wukl@sustech.edu.cn}). }
}

\date{}

\begin{document}
	
	\maketitle

	\vspace{-7mm}

	\begin{abstract}
All the existing entropy stable (ES) schemes for relativistic hydrodynamics (RHD) in the literature were restricted to the ideal equation of state (EOS), which however is often a poor approximation for most relativistic flows due to its inconsistency with the relativistic kinetic theory. 
This paper develops high-order ES finite difference schemes for RHD with general Synge-type EOS, which encompasses a range of special EOSs.
We first establish an entropy pair for the RHD equations with general Synge-type EOS in any space dimensions. 
We rigorously prove that the found entropy function is strictly convex and derive the associated entropy variables, laying the foundation for designing entropy conservative (EC) and ES schemes. 
Due to relativistic effects, one cannot explicitly express primitive variables, fluxes, and entropy variables  in terms of conservative variables. Consequently, this highly complicates the analysis of the entropy structure of the RHD equations, the investigation of entropy convexity, and the construction of EC numerical fluxes. 
By using a suitable set of parameter variables, we construct novel two-point EC fluxes in a  unified form for general Synge-type EOS. 
We obtain high-order EC schemes through linear combinations of the two-point EC fluxes. Arbitrarily high-order  accurate ES schemes are achieved by incorporating dissipation terms into the EC schemes, based on (weighted)  essentially non-oscillatory reconstructions. Additionally, we derive the general dissipation matrix for general Synge-type EOS based on the scaled eigenvectors of the RHD system. We also define a suitable average of the dissipation  matrix at the cell interfaces to ensure that the resulting ES schemes can resolve stationary contact  discontinuities accurately. 
Several numerical examples are provided to validate the accuracy and effectiveness of our schemes for RHD with four special EOSs. 
		
		\vspace{2mm}
		\noindent	

	\end{abstract}

	\section{Introduction}

This paper is concerned with the development of stable high-order numerical methods for simulating the relativistic hydrodynamics (RHD), which have important applications in astrophysics, high energy physics, etc. 
The governing equations of $d$-dimensional special RHD can be written as a system of hyperbolic conservation laws:
\begin{equation} \label{eq:RHD}
    \frac{\partial \textbf{U}}{\partial t}+\sum\limits_{i=1}^d\frac{\partial \mathbf{F}_i(\mathbf{U})}{\partial x_i} = \textbf{0},
\end{equation}
where
\begin{equation} \label{Fi} 
\textbf{U} = \left(
\begin{array}{c}
D \\
\textbf{m}\\
E\\
\end{array} \right),\ \ \ \quad
\mathbf{F}_i(\textbf{U})=
\left(
\begin{array}{c}
D v_i \\
v_i\textbf{m}+p\mathbf{e}_i\\
m_i\\
\end{array} \right)
\end{equation}
with the mass density $D = \rho \gamma$, momentum vector $\mathbf{m} = \rho h \gamma^2 \mathbf{v}$, and energy $E = \rho h \gamma^2-p$.  
Here, $\rho,\mathbf{v}=(v_1,v_2,\cdots,v_d)^\top$, and $p$ denote the rest-mass density, the fluid velocity vector, and the pressure, respectively. The vector $\mathbf{e}_i$ denotes the $i$th column of the identity matrix of size $d$. The velocity is normalized such that the speed of
light is 1. Additionally, $\gamma = \frac{1}{\sqrt{1-|\mathbf{v}|^2}}$ is the Lorentz factor, $h = 1+e+\frac{p}{\rho}$ represents the specific enthalpy, and $e$ is the specific internal energy.

The RHD system \eqref{eq:RHD} is closed by an equation of state (EOS). A general EOS may be written as 
$$
h = h(p,\rho), 
$$ 
which typically satisfies the following inequality \cite{WuTang2017ApJS} for relativistic causality: 
\begin{equation}\label{assumption: strict convex 1}
	h\ \left(\frac{1}{\rho} - \frac{\partial h}{\partial p}(p,\rho)\right) < \frac{\partial h}{\partial \rho}(p,\rho) < 0
\end{equation}
such that the local sound speed $c_s < 1$; see \cite{WuTang2017ApJS} for more details. In this paper, we  assume that $h$ only depends 
on $\frac{p}{\rho}$, namely, 
\begin{equation}\label{eq:gEOS}
	h = h(\theta) = 1  + e (\theta) + \theta \quad \mbox{with} \quad \theta := \frac{p}{\rho}, 
\end{equation}
where $\theta$ is a temperature-like variable. For convenience, we will refer to such a general class of EOS \eqref{eq:gEOS} as Synge-type EOS, because the Synge EOS \eqref{eq:PEOS} for perfect gas and its common approximations belong to this type.
Then the condition \eqref{assumption: strict convex 1} can be equivalently reformulated as
\begin{equation} \label{assumption:1_theta}
	h(\theta)> 
	 \frac{\theta\left(1+e'(\theta)\right)}{e'(\theta)},\qquad e'(\theta)>0,
\end{equation}
which are used throughout this paper.
The condition \eqref{assumption:1_theta}, along with $\lim_{\theta \to 0^+} e(\theta)=0$, implies that $e(\theta) +1 > \sqrt{\theta^2+1} $ or equivalently $h(\theta)>\sqrt{\theta^2+1} + \theta$. Note that this requirement is less stringent than the Taub inequality, $(h-\theta)(h-4\theta)\ge 1$, proposed in \cite{taub1948relativistic}. Consequently, our assumptions on the general EOS \eqref{eq:gEOS} are rather mild, accommodating a wide range of commonly used EOSs.  
A special example of such EOS is the ideal EOS:  
\begin{equation}\label{ID-EOS}
	h(\theta) = 1 + \frac{\Gamma}{\Gamma -1} \theta,  
\end{equation}
where the constant $\Gamma \in (1,2]$ stands for the adiabatic index. 
The ideal EOS \eqref{ID-EOS} is commonly used in non-relativistic fluid dynamics and has also been borrowed to the study of relativistic flows. However, for most relativistic astrophysical flows, the ideal EOS \eqref{ID-EOS} is a poor approximation due to its inconsistency with the relativistic kinetic theory (see \cite{taub1948relativistic}). Furthermore, when the adiabatic index $\Gamma > 2$, the ideal EOS \eqref{ID-EOS} allows for superluminal wave propagation, which violates the principles of special relativity. 
In the relativistic case, 
the correct EOS for the single-component perfect gas was given by Synge in \cite{Synge1957}: 
\begin{equation}\label{eq:PEOS}
	h = \frac{ K_3(1/\theta) }{K_2(1/\theta)},
\end{equation}
where $K_2$ and $K_3$ 
 are the second kind modified Bessel functions of order two and three, respectively. 
Due to the presence of the complicated modified Bessel functions, the EOS \eqref{eq:PEOS} is computationally expensive and thus rarely used in the literature. 
Several efforts have been made to derive simplified EOSs that offer greater accuracy than the the ideal EOS \eqref{ID-EOS}, while being simpler than the EOS \eqref{eq:PEOS}. 
Ryu, Chattopadhyay, and Choi \cite{ryu2006equation} proposed the following EOS
\begin{equation}\label{hEOS1}
	h(\theta) = \frac{2(6\theta^2+4\theta+1)}{3\theta+2}. 
\end{equation}
Sokolov, Zhang, and Sakai \cite{sokolov2001simple} suggested the following EOS 
\begin{equation}\label{hEOS2}
	h(\theta) = 2\theta+\sqrt{1+4\theta^2}.
\end{equation}
Mathews  \cite{mathews1971hydromagnetic} gave the following EOS 
\begin{equation}\label{hEOS3}
	h(\theta) = \frac{5}{2}\theta+\sqrt{1+\frac{9}{4}\theta^2},
\end{equation}
which was later employed by Mignone, Plewa, and Bodo \cite{mignone2005piecewise} in numerical RHD. 
Following \cite{mignone2005piecewise,ryu2006equation}, we will abbreviate the 
EOSs \eqref{ID-EOS}, \eqref{hEOS1}, \eqref{hEOS2}, and \eqref{hEOS3} as 
ID-EOS, RC-EOS, IP-EOS, and TM-EOS, respectively. 
We remark that the above EOSs \eqref{ID-EOS}--\eqref{hEOS3} all belong to the general class of EOSs in the form \eqref{eq:gEOS} and satisfy the condition \eqref{assumption:1_theta}.

The study of the RHD system presents significant difficulties due to its inherent nonlinearity, making analytical approaches difficult to employ. As a result, numerical simulations have become the primary method for investigating the underlying physical principles in RHD. To the best of our knowledge, the earliest numerical studies of RHD can be traced back to references \cite{may1966hydrodynamic,wilson1972numerical}, in which finite difference methods incorporating the artificial viscosity technique were employed to solve the RHD equations in either Lagrangian or Eulerian coordinates. Over the past few decades, numerous high-resolution and high-order accurate numerical methods have been developed to numerically solve the RHD equations. These methods encompass a variety of techniques, such as finite volume methods  (e.g.~\cite{mignone2005hllc,tchekhovskoy2007wham,BalsaraKim2016,chen2021second}), finite difference methods (e.g.~\cite{dolezal1995,del2002efficient,radice2012thc,WuTang2015}), and discontinuous Galerkin methods (e.g.~\cite{radice2011discontinuous,zhao2013runge,KIDDER201784,teukolsky2016formulation}). Adaptive mesh refinement \cite{2006Zhang} and adaptive moving mesh \cite{he2012adaptive1} techniques  have been developed to further enhance the resolution of discontinuities and complex RHD flow structures. 
The physical-constraint-preserving high-order accurate schemes were also designed to maintain 
the positivity of density and pressure as well as the subluminal constraint on the velocity; see  \cite{WuTang2015,WuTang2017ApJS,Wu2017,WuTangM3AS,wu2021provably,WuMEP2021,chen2022physical,WuShu2021GQL}. 
For more related works, interested readers can refer to the review articles \cite{marti2003numerical,Marti2015}, the textbook \cite{rezzolla2013relativistic}, a limited list of some recent papers \cite{endeve2019thornado,wu2020entropy,marquina2019capturing,mewes2020numerical}, as well as the references therein.

The RHD equations \eqref{eq:RHD} exhibit a nonlinear hyperbolic nature, leading to solutions that may be discontinuous with the presence of shocks or contact discontinuities. To address this, weak solutions are typically considered. However, weak solutions may not be unique. To identify the physically relevant solution among the set of weak solutions, admissibility criteria in the form of entropy conditions are commonly imposed. 
Numerically, it is desirable to develop schemes that  satisfy a discrete version of the entropy condition, known as {\em entropy stable} (ES) schemes. Such ES schemes ensure that entropy is conserved in smooth regions while being dissipated across discontinuities, thereby following the entropy principle of physics in an accurate and robust manner. Moreover, the ES methods allow for controlling the amount of dissipation introduced into the schemes to guarantee entropy stability. Thus, the development of ES schemes for the RHD equations \eqref{eq:RHD} is both highly desirable and meaningful.

The study of entropy stability analysis has been extensively carried out for first-order accurate schemes and scalar conservation laws; see  \cite{crandall1980monotone,harten1976finite,osher1984riemann,osher1988convergence}. For hyperbolic systems of conservation laws, most of the attention has been paid to exploring ES schemes that focus on a single given entropy function. 
The framework of ES schemes originates from Tadmor \cite{tadmor1987numerical,tadmor2003entropy}, who systematically established a solid foundation for constructing the second-order entropy conservative (EC) numerical fluxes and first-order ES fluxes. 
Lefloch, Mercier, and Rohde \cite{lefloch2002fully} proposed a general approach for constructing higher-order accurate EC fluxes.
Building upon these developments, Fjordholm, Mishra, and Tadmor \cite{fjordholm2012arbitrarily} developed a general approach to construct ES schemes with arbitrary order of accuracy. 
This approach combines high-order EC numerical fluxes with the essentially non-oscillatory (ENO) reconstruction that satisfies the sign property \cite{fjordholm2013eno}. 
On the other hand, high-order ES schemes have also been constructed via the summation-by-parts (SBP) procedure \cite{fisher2013high,carpenter2014entropy,gassner2013skew}. 
ES space-time discontinuous Galerkin schemes have been investigated in \cite{Barth1998,Barth2006,hiltebrand2014entropy}, where the proof of entropy stability requires exact integration. Recently, a framework for designing ES high-order discontinuous Galerkin methods through suitable numerical quadrature has been proposed in \cite{chen2017entropy}. In this study, the SBP operators established in \cite{fisher2013high,carpenter2014entropy,gassner2013skew} were used and generalized to triangles. 
The key building blocks of high-order accurate ES schemes are the two-point EC numerical fluxes. 
In \cite{tadmor1987numerical,tadmor2003entropy}, Tadmor proposed a general way to derive the two-point EC numerical fluxes, whose formula contains a path integration. 
This leads to difficulties or much cost during the computation since the integration may not have an explicit formula.   
Consequently, researchers have focused on developing \emph{affordable} two-point EC numerical fluxes with explicit formulas.  Several notable advancements have been made in this area for various equations, including the compressible Euler systems  \cite{roe2006affordable,ismail2009affordable,chandrashekar2013kinetic,ranocha2018comparison,zhao2022strictly,li2022high}, shallow water equations \cite{gassner2016well,duan2021}, and magnetohydrodynamics (MHD)  \cite{chandrashekar2016entropy,winters2016affordable,liu2018entropy}. 
In \cite{abgrall2018general}, Abgrall proposed a general framework for residual distribution (RD) schemes to satisfy additional conservation relations, leading to the construction of EC and ES schemes by incorporating suitable correction terms. This entropy correction approach was further extended to time-dependent hyperbolic problems by Abgrall, {\"O}ffner, and Ranocha in \cite{abgrall2022reinterpretation} to design schemes that simultaneously satisfy multiple desired properties.  For the first time, the entropy correction method was used in \cite{abgrall2022reinterpretation} to obtain fully-discrete EC/ES RD schemes.

In recent years, significant efforts have been devoted to developing effective ES schemes for RHD; see \cite{duan2019high, bhoriya2020entropy, duan2021entropy, biswas2022entropy, duan2022high}. The focus of these studies was on the RHD equations \eqref{eq:RHD} with ID-EOS \eqref{ID-EOS}. 
The authors of \cite{duan2019high} and \cite{bhoriya2020entropy} proposed high-order accurate ES finite difference schemes for RHD using two-point EC fluxes and suitable entropy dissipation operators. 
In subsequent work \cite{duan2021entropy, duan2022high}, Duan and Tang extended these schemes to adaptive moving meshes in curvilinear coordinates. 
Additionally, the study of ES schemes was extended to the relativistic MHD equations in \cite{wu2020entropy, duan2020high}. 
It was proven in \cite{wu2020entropy} that conservative relativistic MHD equations are not symmetrizable and do not admit a thermodynamic entropy pair, and a symmetrizable relativistic MHD system with convex thermodynamic entropy pair was proposed in \cite{wu2020entropy}. 
Based on the symmetrizable relativistic MHD equations, high-order ES schemes were developed within the finite difference framework \cite{wu2020entropy} and the discontinuous Galerkin framework \cite{duan2020high}. 
The high-order ES adaptive moving mesh methods were also well studied for relativistic MHD in 
\cite{duan2022high}.

It is worth noting that all the existing work on EC and ES schemes for RHD and relativistic MHD was limited to the ID-EOS \eqref{ID-EOS}. The study of ES schemes for RHD with more accurate EOSs has not been explored yet. 
This paper makes the first effort on constructing explicit EC fluxes and developing high-order ES schemes for 
the RHD equations \eqref{eq:RHD} with general Synge-type EOS \eqref{eq:gEOS}, which covers a wide range of EOSs \eqref{ID-EOS}--\eqref{hEOS3} as special examples.  The difficulties of this work are multi-faceted and include the following aspects:
\begin{itemize}
	\item The convex entropy and entropy fluxes for the RHD system with a general EOS are unclear. 
	\item Due to the nonlinear coupling between the RHD equations \eqref{eq:RHD}, 
	{\em the primitive variables $\mathbf{V}:=(\rho,\mathbf{v},p)^\top$, the fluxes, and the entropy variables all cannot be explicitly expressed by the conservative variables $\mathbf{U}=(D,\mathbf{m},E)^\top$.}  
	This makes it difficult to analyze the entropy structure of the RHD equations \eqref{eq:RHD}, study the convexity of entropy, and construct EC numerical fluxes. 
	\item Developing a unified EC flux formulation for RHD with general EOS is quite nontrivial.
\end{itemize} 
The efforts in this paper are summarized as follows:
\begin{itemize}
    \item We discover an admissible entropy pair for the RHD equations with general Synge-type EOS \eqref{eq:gEOS} in any space dimension. We rigorously prove that the found entropy function is strictly convex, 
    under the relativistic causality condition \eqref{assumption:1_theta}. 
    Furthermore, we derive the entropy variables associated with the convex entropy. 
    These findings lay the foundation for designing EC and ES schemes for RHD with general Synge-type EOS. 
    Due to relativistic effects, 
    the formulation of 
    the Hessian matrix of the entropy function with respect to the conservative variables is quite complicated, making it very difficult to study the convexity of the entropy function. 
    \item We construct the novel two-point EC fluxes in a unified form for RHD with general Synge-type EOS. 
    The construction involves carefully selecting a set of parameter variables that can express the entropy variables and potential fluxes 
     in simple explicit forms. 
    We remark that constructing EC fluxes is highly technical and involves complex reformulation and decomposition of the jumps of the entropy variables. 
    \item We develop semidiscrete high-order accurate EC and ES schemes for the RHD equations with general Synge-type EOS. Second-order EC schemes use the proposed two-point EC fluxes, while higher-order EC schemes are constructed by linearly combining the two-point EC fluxes. Arbitrarily high-order accurate ES schemes are obtained by adding dissipation terms into the EC schemes, based on ENO or weighted ENO (WENO) reconstructions. 
    Moreover, we derive the general dissipation matrix, based on the scaled eigenvectors of the RHD system, for general Synge-type EOS. 
    We also define a suitable average of the dissipation matrix at the cell interfaces, ensuring that the resulting ES schemes can resolve stationary contact discontinuities exactly.  
    \item  We implement the proposed one-dimensional (1D) and two-dimensional (2D) high-order EC and ES schemes coupled with strong-stability-preserving high-order Runge--Kutta time discretization.  
    Several numerical examples are provided to validate the accuracy and effectiveness of our schemes for RHD with various special EOSs.
\end{itemize}

This paper is structured as follows. Section \ref{section:2} presents the entropy pair for the RHD system with   general Synge-type EOS \eqref{eq:gEOS}, and establishes the convexity of the associated entropy function. Additionally, this section derives the relevant entropy variables. In Section \ref{section:3}, we construct the 1D EC and ES schemes. We further discuss the extensions to 2D in Section \ref{section:4}. Section \ref{section:5} presents the numerical experiments, and finally, Section \ref{section:6} provides the concluding remarks.

\section{Entropy analysis for RHD equations}
\label{section:2}


In this section, we seek an admissible entropy pair for the RHD equations \eqref{eq:RHD} with general Synge-type EOS \eqref{eq:gEOS}. 
Furthermore, we will prove that the found entropy function is strictly convex, and then derive the entropy variables associated with the convex entropy.

\subsection{Entropy pair}

First, we recall the definition of an entropy pair. 

\begin{definition}
For the $d$-dimensional RHD system \eqref{eq:RHD}, a continuously differentiable function $\eta:\mathbb{R}^d\longrightarrow\mathbb{R}$ is called the entropy function if there exist $d$ functions $q_i:\mathbb{R}^d\longrightarrow\mathbb{R}$, called entropy fluxes, such that
\begin{equation} \label{mathEntropyDef}
    \left(\frac{\partial \eta}{\partial \mathbf{U}}\right)^\top  \frac{\partial \mathbf{F}_i}{\partial \mathbf{U}}=  \left(\frac{\partial q_i}{\partial \mathbf{U}}\right)^\top,\quad  i = 1,\cdots,d.
\end{equation}
In this case, we call $(\eta,\mathbf{q})$ an entropy pair of \eqref{eq:RHD}, where $\mathbf{q}=(q_1,\cdots,q_d)^\top$, $d=1,2,3$.
\end{definition} 

\begin{theorem}\label{thm:entropypair}
    Define
    \begin{align} \label{entropyFunDef}
        \eta(\mathbf{U}):=-DS, \qquad 
        \mathbf{q}(\mathbf{U}):=-DS\mathbf{v}
    \end{align}
	with 
	    \begin{equation} \label{S_Def}
		S:=-\ln{\rho}+\int^{\theta} \frac{e'(x)}{x}dx.
	\end{equation}
	Then $(\eta(\mathbf{U}),\mathbf{q}(\mathbf{U}))$ forms an entropy pair for the RHD system \eqref{eq:RHD} with general Synge-type EOS \eqref{eq:gEOS}. 
\end{theorem}
\begin{proof}
    Let us verify that $(\eta(\mathbf{U}),\mathbf{q}(\mathbf{U}))$ satisfies the condition \eqref{mathEntropyDef}. 
    Unfortunately, direct calculations of  $ \frac{\partial \eta}{\partial \mathbf{U}}$, $\frac{\partial \mathbf{F}_i}{\partial \mathbf{U}},$ and $\frac{\partial q_i}{\partial \mathbf{U}}$ are very difficult, because these quantities cannot be explicitly formulated in terms of $\mathbf{U}$. 
    Since $\eta,\ \mathbf{F}_i,\ q_i$, and the conservative variables $\mathbf{U}$ can all be explicitly expressed by the primitive variables $\mathbf{V}=(\rho,\mathbf{v},p)^\top$, we can calculate $ \frac{\partial \eta}{\partial \mathbf{U}},\ \frac{\partial \mathbf{F}_i}{\partial \mathbf{U}},\ {\rm and}\frac{\partial q_i}{\partial \mathbf{U}}$ following the chain rule
    \begin{equation}\label{eq:chainrule}
    \begin{aligned}
        \left(\frac{\partial \eta}{\partial \mathbf{U}}\right)^\top&=\left(\frac{\partial \eta}{\partial \mathbf{V}}\right)^\top\frac{\partial \mathbf{V}}{\partial \mathbf{U}},\\ 
        \frac{\partial \mathbf{F}_i}{\partial \mathbf{U}}\ \ \ &=\ \ \frac{\partial \mathbf{F}_i}{\partial \mathbf{V}}\ \ \ \frac{\partial \mathbf{V}}{\partial \mathbf{U}},\quad i = 1,2,\cdots,d,\\ 
        \left(\frac{\partial q}{\partial \mathbf{U}}\right)^\top&=\left(\frac{\partial q}{\partial \mathbf{V}}\right)^\top\frac{\partial \mathbf{V}}{\partial \mathbf{U}},
    \end{aligned}
    \end{equation}
    where the matrix $\frac{\partial \mathbf{V}}{\partial \mathbf{U}}$ can be calculated through the inverse of $\frac{\partial\mathbf{U}}{\partial\mathbf{V}}$ which is easy to compute:  
    \begin{equation} \label{U_V}
        \frac{\partial\mathbf{U}}{\partial\mathbf{V}}=
        \begin{pmatrix}
            \gamma & \rho\gamma^3\mathbf{v}^\top & 0\\
            \gamma^2\Big(h-\theta\left(1+e'(\theta)\right)\Big)\mathbf{v} & \rho h\gamma^2\mathbf{I}_d+2\rho h\gamma^4\mathbf{v}\mathbf{v}^\top & \gamma^2\left(1+e'(\theta)\right)\mathbf{v}\\
            \gamma^2\Big(h-\theta\left(1+e'(\theta)\right)\Big) & 2\rho h\gamma^4\mathbf{v}^\top & \gamma^2\left(1+e'(\theta)\right)-1
        \end{pmatrix}. 
    \end{equation}
    Calculating the inverse of $\frac{\partial \mathbf{U}}{\partial \mathbf{V}}$ gives 
    \begin{equation} \label{V_U}
        \frac{\partial \mathbf{V}}{\partial \mathbf{U}}=
        \frac{1}{{\delta_{\theta}}}
        \begin{pmatrix}
            h\gamma^{-1}\left(e'(\theta)-|\mathbf{v}|^2\right) & -\left(e'(\theta)+|\mathbf{v}|^2\right)\mathbf{v}^\top & \left(1+e'(\theta)\right)|\mathbf{v}|^2\\
            \frac{\gamma^{-3}}{\rho}\Big(h-\theta\left(1+e'(\theta)\right)\Big)\mathbf{v} & \mathbf{M}_1 & -\frac{\left(1+e'(\theta)\right)\left(1-|\mathbf{v}|^2\right)}{\rho}\mathbf{v}\\
            -h\gamma^{-1}\Big(h-\theta\left(1+e'(\theta)\right)\Big) & -\Big(h+\theta\left(1+e'(\theta)\right)\Big)\mathbf{v}^\top & h+\theta\left(1+e'(\theta)\right)|\mathbf{v}|^2
        \end{pmatrix},
    \end{equation}
    where $\delta_{\theta}=he'(\theta)-\theta(1+e'(\theta))|\mathbf{v}|^2$, and
    $\mathbf{M}_1$ is a $d\times d$ matrix for the $d$-dimensional RHD system and is defined by 
    \begin{equation*}
        \mathbf{M}_1:=\frac{1-|\mathbf{v}|^2}{\rho h}\bigg(\delta_{\theta}\mathbf{I}_d+\Big(h+\theta\left(1+e'(\theta)\right)\Big)\mathbf{v}\mathbf{v}^\top\bigg)
    \end{equation*}
	with $\mathbf{I}_d$ denoting the $d\times d$ identity matrix. 
	For $d=1,2,3$, the formulas of $\frac{\partial \eta}{\partial \mathbf{V}}$, $\frac{\partial \mathbf{F}_i}{\partial \mathbf{V}}$ and $\frac{\partial q_i}{\partial \mathbf{V}}$ $(i=1,\cdots,d)$ can also be directly calculated as
    \begin{equation} \label{eta_V}
        \frac{\partial \eta}{\partial \mathbf{V}}=
        \begin{pmatrix}
            \gamma\left(1+e'(\theta)-{S}\right),-\ \rho\gamma^3{S}\mathbf{v}^\top,-\frac{\gamma e'(\theta)}{\theta}
        \end{pmatrix}^\top,
    \end{equation}
    \begin{equation} \label{Fi_V}
        \frac{\partial \mathbf{F}_i}{\partial \mathbf{V}}=
        \begin{pmatrix}
            \gamma v_i & \rho\gamma^3\Big(v_i\mathbf{v}^\top+\left(1-|\mathbf{v}|^2\right)\mathbf{e}_i^\top\Big) & 0\\
            \gamma^2\Big(h-\theta(1+e'(\theta))\Big)v_i\mathbf{v} & \mathbf{M}_2 & \gamma^2(1+e'(\theta))v_i\mathbf{v}+\mathbf{e}_i\\
            \gamma^2\Big(h-\theta(1+e'(\theta))\Big)v_i & \rho h\gamma^4\Big(2v_i\mathbf{v}^\top+(1-|\mathbf{v}|^2)\mathbf{e}_i^\top\Big) & \gamma^2(1+e'(\theta))v_i
        \end{pmatrix}
    \end{equation}
    with $\mathbf{M}_2$ being a $d\times d$ matrix given by 
    \begin{equation*}
        \mathbf{M}_2:=\rho h\gamma^4\Big(2v_i\mathbf{v}\mathbf{v}^\top+(1-|\mathbf{v}|^2)\big(v_i\mathbf{I}_d+\mathbf{v}\mathbf{e}_i^\top\big)\Big),
    \end{equation*}
    and
    \begin{equation} \label{qi_V}
        \frac{\partial q_i}{\partial \mathbf{V}}=
        \begin{pmatrix}
            \gamma\left(1+e'(\theta)-{S}\right)v_i,-\ \rho\gamma^3{S}\Big(\left(1-|\mathbf{v}|^2\right)\mathbf{e}_i^\top+v_i\mathbf{v}^\top\Big),-\frac{\gamma e'(\theta)v_i}{\theta}
        \end{pmatrix}^\top.
    \end{equation}
    Based on \eqref{V_U} and \eqref{eta_V}, we can use the chain rule \eqref{eq:chainrule} to calculate the derivatives of the entropy function $\eta$ with respect to the conservative variables $\mathbf{U}$: 
    \begin{align*}
        \left(\frac{\partial\eta}{\partial\mathbf{U}}\right)^\top
        &=\left(\frac{\partial\eta}{\partial\mathbf{V}}\right)^\top
         \frac{\partial\mathbf{V}}{\partial\mathbf{U}}.
    \end{align*}
    Let ${\bf c}_i$ be the $i$th column of the matrix $\frac{\partial\mathbf{V}}{\partial\mathbf{U}}$. Then we have
    \begin{align*}
          \left(\frac{\partial\eta}{\partial\mathbf{V}}\right)^\top {\bf c}_1
        &=\frac{1}{\delta_{\theta}}\left(h\left(1+e'(\theta)-S\right)(e'(\theta)-|\mathbf{v}|^2)-\left(h-\theta(1+e'(\theta))\right)S|\mathbf{v}|^2+\frac{he'(\theta)\left(h-\theta(1+e'(\theta))\right)}{\theta}\right)\\
        &=\frac{1}{\delta_{\theta}}\left(\left(-he'(\theta)+\theta(1+e'(\theta))|\mathbf{v}|^2\right)S-h(1+e'(\theta))|\mathbf{v}|^2+\frac{h^2e'(\theta)}{\theta}\right)\\
        &=\frac{1}{\delta_{\theta}}\left(-he'(\theta)+\theta(1+e'(\theta))|\mathbf{v}|^2\right)\left(S-\frac{h}{\theta}\right)=\frac{h-\theta S}{\theta},\\
        \left(\frac{\partial\eta}{\partial\mathbf{V}}\right)^\top {\bf c}_2
        &=\frac{1}{\delta_{\theta}}\left(-\gamma(1+e'(\theta)-S)(e'(\theta)+|\mathbf{v}|^2)\mathbf{v}^\top-\frac{\gamma S\mathbf{v}^\top\left(\delta_{\theta}\mathbf{I}_d+\left(h+\theta(1+e'(\theta))\right)\mathbf{v}\mathbf{v}^\top\right)}{h}\right)\\
        &\quad +\frac{1}{\delta_{\theta}}\left(\frac{h\gamma e'(\theta)}{\theta}+\gamma e'(\theta)(1+e'(\theta))\right)\mathbf{v}^\top\\
        &=\frac{1}{\delta_{\theta}}\left(\gamma S\left(e'(\theta)+|\mathbf{v}|^2-\frac{\delta_{\theta}+\left(h+\theta(1+e'(\theta))\right)|\mathbf{v}|^2}{h}\right)\mathbf{v}^\top-\gamma(1+e'(\theta))|\mathbf{v}|^2\mathbf{v}^\top+\frac{h\gamma e'(\theta)}{\theta}\mathbf{v}^\top\right)\\
        &=\frac{\gamma}{\delta_{\theta}}\left(\frac{he'(\theta)}{\theta}-(1+e'(\theta))|\mathbf{v}|^2\right)\mathbf{v}^\top=\frac{\gamma \mathbf{v}^\top}{\theta},\\
        \left(\frac{\partial\eta}{\partial\mathbf{V}}\right)^\top {\bf c}_3 
        &=\frac{1}{\delta_{\theta}}\left(\gamma(1+e'(\theta)-S)(1+e'(\theta))|\mathbf{v}|^2+\gamma S(1+e'(\theta))|\mathbf{v}|^2-\frac{\gamma e'(\theta)}{\theta}\left(h+\theta(1+e'(\theta))|\mathbf{v}|^2\right)\right)\\
        &=\frac{\gamma}{\delta_{\theta}}\left((1+e'(\theta))|\mathbf{v}|^2-\frac{he'(\theta)}{\theta}\right)=-\frac{\gamma}{\theta}.
    \end{align*}
    Hence, we obtain
    \begin{equation} \label{generalEntropyVar:cal}
        \frac{\partial\eta}{\partial\mathbf{U}}=\frac1{\theta}
		(h-\theta S, \gamma {\bf v}^\top, -\gamma )^\top.
    \end{equation}
    Let $\frac{\partial \mathbf{F}_i}{\partial \mathbf{V}}=:
    \begin{pmatrix}
        \hat{\bf F}_i^{(1)} & \hat{\mathbf{F}}_i^{(2)} & \hat{\bf F}_i^{(3)}
    \end{pmatrix}$ and
    $\frac{\partial q_i}{\partial \mathbf{V}}=:
    \begin{pmatrix}
        \hat{q}_i^{(1)} & \hat{\mathbf{q}}_i^{(2)} & \hat{q}_i^{(3)}
    \end{pmatrix}$. Then we have
    \begin{align*}
        \left(\frac{\partial\eta}{\partial\mathbf{U}}\right)^\top\hat{\bf F}_i^{(1)}-\hat{q}_i^{(1)} &= \frac{h-\theta S}{\theta}\gamma v_i+\frac{\gamma {\bf v}^\top}{\theta} \gamma^2\Big(h-\theta(1+e'(\theta))\Big)v_i\mathbf{v}\\
        &\quad-\frac{\gamma}{\theta}\gamma^2\Big(h-\theta(1+e'(\theta))\Big)v_i-\gamma\left(1+e'(\theta)-{S}\right)v_i\\
        &= \frac{h-\theta S}{\theta}\gamma v_i-\frac{\gamma v_i}{\theta}\left(h-\theta(1+e'(\theta))\right)-\gamma v_i(1+e'(\theta)-S)\\
        &=\gamma v_i\left(\frac{h-\theta S}{\theta}-\frac{h-\theta(1+e'(\theta))}{\theta}-1-e'(\theta)+S\right)\\
        &=0,\\
        \left(\frac{\partial\eta}{\partial\mathbf{U}}\right)^\top\hat{\bf F}_i^{(2)}-\hat{\bf q}_i^{(2)} &= \frac{h-\theta S}{\theta}\rho\gamma^3\Big(v_i\mathbf{v}^\top+\left(1-|\mathbf{v}|^2\right)\mathbf{e}_i^\top\Big)\\
        &\quad+\frac{\gamma {\bf v}^\top}{\theta}\rho h\gamma^4\Big(2v_i\mathbf{v}\mathbf{v}^\top+(1-|\mathbf{v}|^2)\big(v_i\mathbf{I}_d+\mathbf{v}\mathbf{e}_i^\top\big)\Big)\\
        &\quad-\frac{\gamma}{\theta}\rho h\gamma^4\Big(2v_i\mathbf{v}^\top+(1-|\mathbf{v}|^2)\mathbf{e}_i^\top\Big)+ \rho\gamma^3{S}\Big(\left(1-|\mathbf{v}|^2\right)\mathbf{e}_i^\top+v_i\mathbf{v}^\top\Big)\\
        &=\frac{\rho h\gamma^3}{\theta}\left(2v_i\mathbf{v}^\top+\left(1-|\mathbf{v}|^2\right)\mathbf{e}_i^\top+2\gamma^2v_i|\mathbf{v}|^2\mathbf{v}^\top+|\mathbf{v}|^2\mathbf{e}_i^\top-2\gamma^2v_i\mathbf{v}^\top-\mathbf{e}_i^\top\right)\\
        &=\frac{\rho h\gamma^3}{\theta}\left(2v_i\mathbf{v}^\top-2\gamma^2v_i\mathbf{v}^\top(1-|\mathbf{v}|^2)\right)\\
        &=0,\\
        \left(\frac{\partial\eta}{\partial\mathbf{U}}\right)^\top\hat{\bf F}_i^{(3)}-\hat{q}_i^{(3)} &= \frac{\gamma {\bf v}^\top}{\theta}\left(\gamma^2(1+e'(\theta))v_i\mathbf{v}+\mathbf{e}_i\right)-\frac{\gamma^3}{\theta}(1+e'(\theta))v_i\\
        &=-\frac{\gamma^3v_i(1+e'(\theta))}{\theta}(1-|\mathbf{v}|^2)+\frac{\gamma\mathbf{v}^\top\mathbf{e}_i}{\theta}+\frac{\gamma e'(\theta)v_i}{\theta}\\
        &=0,
    \end{align*}
    which imply  
    \begin{equation}\label{VUlast}
        \left(\frac{\partial \eta}{\partial\mathbf{U}}\right)^\top\frac{\partial\mathbf{F}_i}{\partial\mathbf{V}}=\left(\frac{\partial q_i}{\partial\mathbf{V}}\right)^\top,\qquad i = 1,\cdots,d.
    \end{equation}
    Multiplying both sides of the equations \eqref{VUlast} by $\frac{\partial\mathbf{V}}{\partial\mathbf{U}}$ from right,  we obtain \eqref{mathEntropyDef}, which indicates that $(\eta(\mathbf{U}),\mathbf{q}(\mathbf{U}))$ forms an entropy pair. The proof is completed. 
\end{proof}

As direct consequences of Theorem \ref{thm:entropypair}, we have the following remarks for four specific EOSs.

\begin{remark1}[ID-EOS]
	For the ideal EOS \eqref{ID-EOS}, we have 
	\begin{equation} \label{IDe}
		e(\theta) = \frac{\theta}{\Gamma-1}.
	\end{equation}
	Then by \eqref{S_Def}, we obtain 
	\begin{equation} \label{IDS}
		S_{ID}:=-\ln{\rho}+\frac{1}{\Gamma-1}\int\frac{1}{\theta}d\theta=-\ln{\rho}+\frac{1}{\Gamma-1}\ln{\theta}=\frac{1}{\Gamma-1}\ln{\frac{p}{\rho^\Gamma}}.
	\end{equation}
	Our entropy pair $(-DS_{ID},-DS_{ID}\mathbf{v})$ for the RHD system \eqref{eq:RHD} with the ideal EOS is consistent with the result in \cite{duan2019high}.
\end{remark1} 

\begin{remark1}[RC-EOS]
For the RC-EOS \eqref{hEOS1}, we have 
\begin{equation} \label{RCe}
	e(\theta) = \frac{3\theta(3\theta+1)}{3\theta+2}.
\end{equation}
Then by \eqref{S_Def}, we obtain 
\begin{equation*}
	\begin{aligned}
		S_{RC}&=-\ln{\rho}+\int \left(\frac{3}{2\theta}+\frac{9}{2(3\theta+2)}+\frac{9}{(3\theta+2)^2}\right)d\theta\\
		&=-\ln{\rho}+\frac{3}{2}\ln{\theta}+\frac{3}{2}\ln{(3\theta+2)}-\frac{3}{3\theta+2}.
	\end{aligned}
\end{equation*}
Hence $(-DS_{RC},-DS_{RC}\mathbf{v})$ forms an entropy pair for the RHD system \eqref{eq:RHD} with the RC-EOS \eqref{hEOS1}. 
\end{remark1}

\begin{remark1}[IP-EOS]
	For the {IP-EOS} \eqref{hEOS2}, we have 
	\begin{equation} \label{IPe}
		e(\theta) = \theta-1+\sqrt{1+4\theta^2}.
	\end{equation}
	Then by \eqref{S_Def}, we obtain 
	\begin{equation*}
		\begin{aligned}
			S_{IP}&=-\ln{\rho}+\int \left(\frac{1}{\theta}+\frac{4}{\sqrt{1+4\theta^2}}\right)d\theta\\
			&=-\ln{\rho}+\ln{\theta}+2\ln{\left(2\theta+\sqrt{1+4\theta^2}\right)}.
		\end{aligned}
	\end{equation*}
	Hence $(-DS_{IP},-DS_{IP}\mathbf{v})$ forms an entropy pair for the RHD system \eqref{eq:RHD} with the IP-EOS \eqref{hEOS2}. 
\end{remark1}

\begin{remark1}[TM-EOS]
	For the {TM-EOS} \eqref{hEOS3}, we have 
	\begin{equation} \label{TMe}
		e(\theta) = \frac{3}{2}\theta-1+\sqrt{1+\frac{9}{4}\theta^2}.
	\end{equation}
	Then by \eqref{S_Def}, we obtain 
	\begin{equation*}
		\begin{aligned}
			{S}_{TM}&=-\ln{\rho}+\int \left(\frac{3}{2}+\frac{\frac{9}{4}\theta}{\sqrt{1+\frac{9}{4}\theta^2}}\right)d\theta\\
			&=-\ln{\rho}+\frac{3}{2}\ln{\theta}+\frac{3}{2}\ln{\left(\frac{3}{2}\theta+\sqrt{1+\frac{9}{4}\theta^2}\right)}.
		\end{aligned}
	\end{equation*}
	Hence $(-DS_{TM},-DS_{TM}\mathbf{v})$ forms an entropy pair for the RHD system \eqref{eq:RHD} with the TM-EOS \eqref{hEOS3}. 
\end{remark1}

\subsection{Convexity of entropy function}
In this subsection, we show the convexity of the entropy function $\eta({\bf U})$ defined in \eqref{entropyFunDef}. 

\begin{theorem}\label{thm:strict convex entropy -DS}
	For the RHD system \eqref{eq:RHD} with general Synge-type EOS \eqref{eq:gEOS} satisfying the condition \eqref{assumption:1_theta}, 
	the entropy function $\eta({\bf U})$ defined in \eqref{entropyFunDef} is strictly convex with respect to the conservative variables  $\mathbf{U}$, provided that $\mathbf{U}\in \mathcal{G}$ with
	$\mathcal{G} = \{\mathbf{U} = (D,\mathbf{m},E)^\top: \rho > 0,\ \theta > 0\}.$
\end{theorem}

\begin{proof}
	To show the convexity of the entropy function $\eta$, it suffices to verify the positive definiteness of 
	the Hessian matrix of the entropy function $\eta$, which can be written as
		\begin{equation*}
			\eta_{\mathbf{U}\mathbf{U}} = \frac{\partial^2 \eta}{\partial \mathbf{U}^2}= -(\mathbf{e}_1S_{\mathbf{U}}^\top+S_{\mathbf{U}}\mathbf{e}_1^\top+DS_{\mathbf{U}\mathbf{U}}),
		\end{equation*}
		where $\mathbf{e}_1=(1,\mathbf{0}_{d+1}^\top)^\top$, and $\mathbf{0}_{d+1}$ is the $(d+1)\times 1$ zero vector.
		As $S$ can be explicitly expressed by $\mathbf{V}$ but not $\mathbf{U}$, we can calculate $S_{\mathbf{U}}$ and $S_{\mathbf{U}\mathbf{U}}$ following the chain rule
		\begin{equation} \label{S_U chain rule}
			S_{\mathbf{U}}^\top = \left(\frac{\partial S}{\partial \mathbf{V}}\right)^\top\frac{\partial\mathbf{V}}{\partial \mathbf{U}},\qquad
			S_{\mathbf{UU}}=\frac{\partial^2 S}{\partial \mathbf{U}^2} = \frac{\partial^2 S}{\partial \mathbf{V}\partial\mathbf{U}}\frac{\partial\mathbf{V}}{\partial \mathbf{U}}=S_{\mathbf{UV}}\frac{\partial\mathbf{V}}{\partial \mathbf{U}}.
		\end{equation}
		Due to the explicit relation between $S$ and $\mathbf{V}$, we can derive $\frac{\partial S}{\partial \mathbf{V}}$ directly as
		\begin{equation} \label{S_V}
			\frac{\partial S}{\partial \mathbf{V}}=\begin{pmatrix}
				-\frac{1+e'(\theta)}{\rho} & \mathbf{0}_d^\top & \frac{e'(\theta)}{\rho\theta}
			\end{pmatrix}^\top.
		\end{equation}
		Combining  \eqref{S_V} with \eqref{V_U}, we obtain 
		\begin{equation} \label{S_U}
			S_{\mathbf{U}}=\frac{1}{\rho\theta}\begin{pmatrix}
				-h\gamma^{-1} & -\mathbf{v}^\top & 1
			\end{pmatrix}^\top.
		\end{equation}
		The derivative of $S_{\mathbf{U}}$ with respect to $\mathbf{V}$ gives
		\begin{equation} \label{S_UV}
			S_{\mathbf{UV}}=\begin{pmatrix}
				\frac{1+e'(\theta)}{\rho^2\gamma} & \frac{h\gamma}{\rho\theta}\mathbf{v}^\top & \frac{h-\theta(1+e'(\theta))}{\rho^2\theta^2\gamma} \\
				\mathbf{0}_d & -\frac{1}{\rho\theta}\mathbf{I}_d & \frac{1}{\rho^2\theta^2}\mathbf{v}\\
				0 & \mathbf{0}_d^\top & -\frac{1}{\rho^2\theta^2}
			\end{pmatrix},
		\end{equation}
		where $\mathbf{I}_d$ denotes the $d\times d$ identity matrix.
		Therefore, we have
		\begin{equation*}
			\eta_{\mathbf{UU}}=-\left(\mathbf{e}_1S_{\mathbf{U}}^\top+S_{\mathbf{U}}\mathbf{e}_1^\top+DS_{\mathbf{U}\mathbf{V}}\frac{\partial \mathbf{V}}{\partial \mathbf{U}}\right)
			= \frac{1}{\rho^2 \theta^3 \left(1+e'(\theta)\right)\ \left(\frac{1}{c_s^2} - |\mathbf{v}|^2\right)}
			\begin{pmatrix}
				a_1 & a_2\mathbf{v}^\top & a_3 \\ \\
				a_2\mathbf{v} & \mathbf{A}_1 & a_4\mathbf{v} \\ \\
				a_3 & a_4\mathbf{v}^\top & a_5 
			\end{pmatrix},
	\end{equation*}
	where $c_s$ denotes the sound speed with 
	\begin{equation} \label{cs_theta}
    	c_s^2 =  \frac{\theta(1+e'(\theta))}{he'(\theta)} \in (0,1)
	\end{equation} under the condition \eqref{assumption:1_theta}, 
	and 
\begin{equation*}
	a_1 := h\gamma^{-1} \Delta,\quad
	\Delta := \rho h^2-\rho\theta h\left(2+|\mathbf{v}|^2\right)+\rho\theta^2\left(1+e'(\theta)\right),
\end{equation*}
\begin{equation*} 
a_2 := \rho \Big(h^2-\theta h|\mathbf{v}|^2-\theta^2\left(1+e'(\theta)\right)\Big),\quad
a_3 := -\rho\Big(h^2-\theta h-\theta^2\left(1+e'(\theta)\right)|\mathbf{v}|^2\Big),
\end{equation*}
\begin{equation*}
\mathbf{A}_1:=\frac{\rho}{h}\left(\frac{\theta\left(-h+\left(1+e'(\theta)\right)\left(h-\theta|\mathbf{v}|^2\right)\right)}{\gamma}\mathbf{I}_d+\gamma \left(h+\theta\left(1+e'(\theta)\right)\right)\left(h+\theta\left(1-|\mathbf{v}|^2\right)\right)\mathbf{v}\mathbf{v}^\top\right),
\end{equation*}
\begin{equation*} 
a_4 := -\rho \gamma \left(h+\theta\left(1+e'(\theta)\right)\right),\quad 
a_5 := \rho \gamma \left(h+\theta\left(1+e'(\theta)\right)|\mathbf{v}|^2\right).
\end{equation*}
We observe that $\Delta >0$ and $a_1>0$, because 
\begin{equation*}\label{A11lessthan0:2}
	\begin{aligned}
		\Delta &= \rho h^2-\rho\theta h\left(2+|\mathbf{v}|^2\right)+\rho\theta^2\left(1+e'(\theta)\right)\\
		& > \rho h^2-3\rho\theta h+\rho\theta^2\left(1+e'(\theta)\right)\\
		& > \rho h^2-3\rho\theta h+\rho\theta^2\frac{h}{h-\theta}\\
		&=\frac{\rho h}{h-\theta}\left(h-2\theta\right)^2\ge 0,
	\end{aligned}
\end{equation*}
where the condition \eqref{assumption:1_theta} has been used. 
Let us define the invertible matrix
\begin{equation*}
{\bf P}_1:= 
\begin{pmatrix}
	1 & \mathbf{0}_d^\top & 0 \\
	-\frac{a_2}{a_1}\mathbf{v} & \mathbf{I}_d & \mathbf{0}_d \\
	-\frac{a_3}{a_1} & \mathbf{0}_d^\top & 1 
\end{pmatrix}.
\end{equation*}
Then we have 
\begin{equation*}
{\bf P}_1 \eta_{\mathbf{UU}} {\bf P}_1^\top =  \frac{1}{\rho^2 \theta^3 \left(1+e'(\theta)\right)\ \left(\frac{1}{c_s^2} - |\mathbf{v}|^2\right)}
	\begin{pmatrix}
		a_1 & \mathbf{0}_{d+1}^\top \\
		\mathbf{0}_{d+1} & \delta\mathbf{A}_2
	\end{pmatrix}
\end{equation*}
with
\begin{equation*}
	\delta := \frac{\rho\theta^2\left(1+e'(\theta)\right)}{a_1}\left(\frac{1}{c_s^2}- |\mathbf{v}|^2\right)>0,\quad {\rm and} \quad
	\mathbf{A}_2:=  \begin{pmatrix}
		\Delta\left(1-|\mathbf{v}|^2\right)\mathbf{I}_d+\delta_3\mathbf{v}\mathbf{v}^\top & -\delta_1\mathbf{v} \\
		-\delta_1\mathbf{v}^\top & \delta_2 
	\end{pmatrix},
\end{equation*}
where 
\begin{equation}\label{delta123}
\begin{aligned}
	\delta_1&:=\rho\Big(h^2+\theta h|\mathbf{v}|^2+\theta^2\left(1+e'(\theta)\right)\Big)=\Delta+2\rho\theta h(1+|\mathbf{v}|^2)>0,\quad\\
	\delta_2&:=\rho\Big(h^2|\mathbf{v}|^2+\theta h+\theta^2\left(1+e'(\theta)\right)|\mathbf{v}|^2\Big)=\Delta|\mathbf{v}|^2+\rho\theta h(1+|\mathbf{v}|^2)^2>0,\quad\\
	\delta_3&:=\rho\Big(h^2+\left(2-|\mathbf{v}|^2\right)\theta h+\theta^2\left(1+e'(\theta)\right)\Big)=\Delta+4\rho\theta h>0.
\end{aligned}
\end{equation}
Let us study the matrix $\mathbf{A}_2$. 
We consider
\begin{equation*}
{\bf P}_2:= 
\begin{pmatrix}
	\mathbf{I}_d & \frac{\delta_1}{\delta_2}\mathbf{v} \\
	\mathbf{0}_d^\top & 1 
\end{pmatrix},
\end{equation*}
and we have
\begin{equation*}
    \mathbf{P}_2\mathbf{A}_2  \mathbf{P}_2^\top = 
    \begin{pmatrix}
        \mathbf{B}-\frac{\delta_1^2}{\delta_2}\mathbf{v}\mathbf{v}^\top & \mathbf{0}_d\\
        \mathbf{0}_d^\top & \delta_2
    \end{pmatrix}.
\end{equation*}
Note that the eigenvalues of
\begin{equation*}
    \mathbf{C}:=\mathbf{B}-\frac{\delta_1^2}{\delta_2}\mathbf{v}\mathbf{v}^\top = 
    \Delta\left(1-|\mathbf{v}|^2\right)\mathbf{I}_d+\left(\delta_3-\frac{\delta_1^2}{\delta_2}\right)\mathbf{v}\mathbf{v}^\top
\end{equation*}
are
\begin{align*}
\lambda_{\mathbf{C}}^{(1)} & = \Delta+\left(\delta_3-\frac{\delta_1^2}{\delta_2}-\Delta\right)|\mathbf{v}|^2 = \frac{1}{\delta_2} \rho\theta h\left(1-|\mathbf{v}|^2\right)^{2}\Delta > 0,
\\
 \lambda_{\mathbf{C}}^{(2)}&=\cdots=\lambda_{\mathbf{C}}^{(d)}=\Delta\left(1-\Vert\mathbf{v}\Vert^2\right) >0,\quad \mbox{if}~d\ge 2.
\end{align*}
This implies that the matrix $\mathbf{C}$ is positive definite, yielding that $\mathbf{P}_2\mathbf{A}_2  \mathbf{P}_2^\top$ is also positive definite. 
Hence, the matrix $\mathbf{A}_2$ is positive definite, implying ${\bf P}_1\eta_{\mathbf{UU}} {\bf P}_1^\top$ is also positive definite. Since ${\bf P}_1\eta_{\mathbf{UU}} {\bf P}_1^\top$ and $\eta_{\mathbf{UU}}$ are congruent, the Hessian matrix $\eta_{\mathbf{UU}}$ is positive definite. The proof is completed.
\end{proof}

\subsection{Entropy variables}

In this subsection, we derive the entropy variables corresponding to the convex entropy $\eta$, which will be useful for constructing the ES schemes.

\begin{theorem} \label{lem:generalEOS:entropyVarPot}
	The entropy variables $\bf W$ associated with the entropy function $\eta$ defined in \eqref{entropyFunDef} 
	are given by 
	\begin{equation} \label{generalEntropyVar}
		\mathbf{W}=\frac1{\theta}
		(h-\theta S, \gamma {\bf v}^\top, -\gamma )^\top,
	\end{equation}
	and the associated entropy potential fluxes are 
	\begin{equation} \label{generalEntropyPot}
		\psi_i := \rho\gamma v_i, \quad i = 1,\cdots,d,
	\end{equation}
	where $h$ represents the specific enthalpy, and $S$ is defined in \eqref{S_Def}.
\end{theorem}

\begin{proof}
	By definition, the entropy variable $\textbf{W}$ is the gradient of the entropy function $\eta$ with respect to the conservative vector $\textbf{U}$, which we have obtained in \eqref{generalEntropyVar:cal}. The potential flux $\psi_i$ is defined as $\psi_i:=\textbf{W}^\top\textbf{F}_i-q_i$, which can be calculated by using \eqref{Fi} and \eqref{entropyFunDef} as follows: 
	\begin{align*}
	    \psi_i &=\frac{h-\theta S}{\theta}Dv_i+\frac{\gamma\mathbf{v}^\top}{\theta}(v_i\mathbf{m}+p\mathbf{e}_i)-\frac{\gamma}{\theta}\mathbf{m}_i+DSv_i\\
	    &=\frac{h}{\theta}\rho\gamma v_i+\frac{\gamma v_i}{\theta}\sum\limits_{i=1}^d\rho h\gamma^2v_i^2+\rho\gamma v_i-\frac{\rho h\gamma^3v_i}{\theta}\\
	    &=\frac{h}{\theta}\rho\gamma v_i+\frac{h}{\theta}\rho\gamma^3 v_i|\mathbf{v}|^2+\rho\gamma v_i-\frac{h}{\theta}\rho h\gamma^3v_i\\
	    &=\rho\gamma v_i, \qquad i = 1,\cdots,d.
	\end{align*}
\end{proof}

\section{1D entropy stable schemes}
\label{section:3}

In this section, we construct the ES schemes for the 1D RHD equations. 

\subsection{Two-point entropy conservative flux}
We first derive the unified formula of two-point EC numerical flux for the 1D RHD system with general Synge-type EOS \eqref{eq:gEOS}. 

\begin{definition}[\cite{tadmor2003entropy}]\label{def:EC}
	A consistent two-point numerical flux $\widetilde{\mathbf{F}}_i^{EC} ( {\mathbf U}_L, {\mathbf U}_R )$ is EC if
	\begin{equation} \label{ECdef}
		( {\mathbf W}_R - {\mathbf W}_L )^\top\widetilde{\textbf{F}}^{EC}_i ( {\mathbf U}_L, {\mathbf U}_R ) = \psi_{i,R} - \psi_{i,L}, \qquad i=1,\dots,d,
	\end{equation}
	where the entropy variables $\textbf{W}$ are defined in \eqref{generalEntropyVar}, and $\psi_i$ is the potential flux defined in \eqref{generalEntropyPot}. The subscript $L$ and $R$ means that the quantities are associated with the ``left'' state ${\mathbf U}_L$ and the ``right'' state ${\mathbf U}_R$, respectively. 
\end{definition}

For convenience, we introduce some notations. The jump and the arithmetic average of a quantity $a$ across a cell interface are denoted by
\begin{equation} \label{jump}
    [\![a]\!]:=a_{R}-a_{L},
\end{equation}
and
\begin{equation} \label{arithmeticAVG}
    \{\!\{a\}\!\}:=\frac{a_{R}+a_{L}}{2},
\end{equation}
respectively. 
Based on these notations, we have the following useful formulas
	\begin{align}
		\relax[\![ab]\!]&=\{\!\{a\}\!\}[\![b]\!]+[\![a]\!]\{\!\{b\}\!\},\label{product}\\
		[\![a^2]\!]&=2\{\!\{a\}\!\}[\![a]\!],\label{square}\\
		\left[\!\left[\frac{a}{b}\right]\!\right]&=\frac{[\![a]\!]-[\![b]\!]\left\{\!\left\{\frac{a}{b}\right\}\!\right\}}{\{\!\{b\}\!\}},\label{division}\\
		\left[\!\left[\sqrt{a}\right]\!\right]&=\frac{[\![a]\!]}{2\left\{\!\left\{\sqrt{a}\right\}\!\right\}}.\label{sqrt}
\end{align}
We will also employ the logarithmic mean 
\begin{equation} 
	\{\!\{a\}\!\}_{ln}=\frac{[\![a]\!]}{\left[\!\left[\ln{a}\right]\!\right]}\label{logAVG}.
\end{equation}
which was proposed in \cite{ismail2009affordable}.

 To design a simple two-point EC numerical flux, we choose a set of  variables $\textbf{z}=(z_1,z_2,z_3)^\top$ as
\begin{equation} \label{1DalgeVar}
    z_1 = \rho, \quad z_2 = \frac{\rho}{p}, \quad z_3 = \gamma v_1
\end{equation}
following \cite{wu2020entropy}. 
After careful investigation, we find  
a unified simple two-point EC flux \eqref{1D2ptECFlux} for the 1D RHD system with general Synge-type EOS \eqref{eq:gEOS}.

\begin{theorem} \label{thm:ECflux}
The two-point EC numerical flux for the 1D RHD system with general Synge-type EOS \eqref{eq:gEOS} can be written into a unified form as 
\begin{equation} \label{1D2ptECFlux}
    \widetilde{\textbf{F}}^{EC}_1 ( {\mathbf U}_L, {\mathbf U}_R )=\left(\{\!\{z_1\}\!\}_{\ln}\{\!\{z_3\}\!\},\ \widehat{\rho h}\{\!\{z_3\}\!\}^2+\frac{\{\!\{z_1\}\!\}}{\{\!\{z_2\}\!\}},\ \widehat{\rho h}\{\!\{\gamma\}\!\}\{\!\{z_3\}\!\}\right)^\top
\end{equation}
with 
\begin{align}\nonumber
	\widehat{\rho h}&:=\frac{\frac{\{\!\{z_1\}\!\}}{\{\!\{z_2\}\!\}}+\{\!\{z_1\}\!\}_{\ln}{\mathcal E}}{\{\!\{\gamma\}\!\}^2-\{\!\{z_3\}\!\}^2}, 
    \\ \label{1.8}
	{\mathcal E} &:= \frac{[\![W_1]\!]-[\![\ln z_1]\!]}{[\![z_2]\!]}. 
\end{align}
Here, ${\mathcal E}$ can be reformulated as 
\begin{equation} \label{1.8:int}
    {\mathcal E} = 1+ \int_0^1e\left(\frac{1}{z_{2,L}+s(z_{2,R}-z_{2,L})}\right)ds,
\end{equation}
where $z_{2,L}$ and $z_{2,R}$ represent the ``left'' and ``right'' states of the parameter variable $z_2$ chosen in \eqref{1DalgeVar}, respectively, and the explicit calculation of ${\mathcal E}$ depends on the particular choice of the EOS and will be given in Theorem \ref{thm:1De}. 
\end{theorem}

\begin{proof}

By using the set of variables \eqref{1DalgeVar}, we can express the entropy variables and the potential flux as
\begin{equation*}
    \mathbf{W}=\left(z_2h-{S}, z_3z_2, -z_2\sqrt{1+z_3^2}\right)^\top,
\end{equation*}
\begin{equation*}
    \psi_1=z_1z_3.
\end{equation*}
Then we write the jumps of entropy variables $\textbf{W}$ and the potential flux $\psi_1$ in terms of jumps and arithmetic averages of the variables \eqref{1DalgeVar} as follows:
\begin{equation} \label{thm3:temp}
	\begin{aligned}
		\relax[\![W_1]\!] &\overset{\eqref{1.8}}{=} {\mathcal E}[\![z_2]\!]+[\![\ln{z_1}]\!],\\
		[\![W_2]\!] &= [\![z_3z_2]\!] \overset{\eqref{product}}{=} \{\!\{z_3\}\!\}[\![z_2]\!]+[\![z_3]\!]\{\!\{z_2\}\!\},\\
		[\![W_3]\!] &= \left[\!\left[-z_2\sqrt{1+z_3^2}\right]\!\right]
		\\
		& \overset{\eqref{product}}{=}-[\![z_2]\!]\{\!\{\gamma\}\!\}-\{\!\{z_2\}\!\}\left[\!\left[\sqrt{1+z_3^2}\right]\!\right]\\
		&\overset{\eqref{sqrt}}{=}-[\![z_2]\!]\{\!\{\gamma\}\!\}-\{\!\{z_2\}\!\}\frac{\left[\!\left[1+z_3^2\right]\!\right]}{2\{\!\{\gamma\}\!\}}\\
		&\overset{\eqref{square}}{=}-[\![z_2]\!]\{\!\{\gamma\}\!\}-\{\!\{z_2\}\!\}\frac{\{\!\{z_3\}\!\}[\![z_3]\!]}{\{\!\{\gamma\}\!\}},\\
		[\![\psi_1]\!] &= [\![z_1z_3]\!] \overset{\eqref{product}}{=}\{\!\{z_1\}\!\}[\![z_3]\!]+[\![z_1]\!]\{\!\{z_3\}\!\},
	\end{aligned}
\end{equation}
According to Definition \ref{def:EC}, the two-point EC numerical flux $\widetilde{\textbf{F}}^{EC}_1 ( {\mathbf U}_L, {\mathbf U}_R )=:\left(\widetilde{F}^{(1)}_{1},\widetilde{F}^{(2)}_{1},\widetilde{F}^{(3)}_{1}\right)^\top$ for the 1D RHD system satisfies
\begin{equation} \label{EClinearsystem}
    [\![W_1]\!]\widetilde{F}^{(1)}_{1}+[\![W_2]\!]\widetilde{F}^{(2)}_{1}+[\![W_3]\!]\widetilde{F}^{(3)}_{1}=[\![\psi_1]\!].
\end{equation}
Substituting \eqref{thm3:temp} into \eqref{EClinearsystem}, we obtain 
\begin{align*}
\left({\mathcal E}[\![z_2]\!]+[\![\ln{z_1}]\!]\right)\widetilde{F}^{(1)}_{1}+\left(\{\!\{z_3\}\!\}[\![z_2]\!]+[\![z_3]\!]\{\!\{z_2\}\!\}\right)\widetilde{F}^{(2)}_{1}&-\left([\![z_2]\!]\{\!\{\gamma\}\!\}+\{\!\{z_2\}\!\}\frac{\{\!\{z_3\}\!\}[\![z_3]\!]}{\{\!\{\gamma\}\!\}}\right)\widetilde{F}^{(3)}_{1}\\
&=\{\!\{z_1\}\!\}[\![z_3]\!]+[\![z_1]\!]\{\!\{z_3\}\!\}.
\end{align*}
Collecting the terms containing $[\![z_1]\!],[\![z_2]\!]$, and $[\![z_3]\!]$, respectively, the above equation can be reformulated as
\begin{align*}
\left(\frac{\widetilde{F}^{(1)}_{1}}{\{\!\{z_1\}\!\}_{\ln}}-\{\!\{z_3\}\!\}\right)[\![z_1]\!]&+\left(\mathcal{E}\widetilde{F}^{(1)}_{1}+\{\!\{z_3\}\!\}\widetilde{F}^{(2)}_{1}-\{\!\{\gamma\}\!\}\widetilde{F}^{(3)}_{1}\right)[\![z_2]\!]\\
&+\left(\{\!\{z_2\}\!\}\widetilde{F}^{(2)}_{1}-\frac{\{\!\{z_2\}\!\}\{\!\{z_3\}\!\}}{\{\!\{\gamma\}\!\}}\widetilde{F}^{(3)}_{1}-\{\!\{z_1\}\!\}\right)[\![z_3]\!]=0.
\end{align*}
Hence, the coefficients of $[\![z_1]\!],[\![z_2]\!],[\![z_3]\!]$ should all equal zero. Specifically, we have
\begin{equation*}
\left\{
\begin{aligned}
&\ \frac{\widetilde{F}^{(1)}_{1}}{\{\!\{z_1\}\!\}_{\ln}} = \{\!\{z_3\}\!\},\\
&\ \mathcal{E}\widetilde{F}^{(1)}_{1}+\{\!\{z_3\}\!\}\widetilde{F}^{(2)}_{1}-\{\!\{\gamma\}\!\}\widetilde{F}^{(3)}_{1}=0,\\
&\ \{\!\{z_2\}\!\}\widetilde{F}^{(2)}_{1}-\frac{\{\!\{z_2\}\!\}\{\!\{z_3\}\!\}}{\{\!\{\gamma\}\!\}}\widetilde{F}^{(3)}_{1}=\{\!\{z_1\}\!\}.
\end{aligned}
\right.
\end{equation*}
Solving the above equations for $\left(\widetilde{F}^{(1)}_{1},\widetilde{F}^{(2)}_{1},\widetilde{F}^{(3)}_{1}\right)^\top$, we obtain
\begin{equation*}
\left\{
\begin{aligned}
&\ \widetilde{F}^{(1)}_{1}=\{\!\{z_1\}\!\}_{\ln}\{\!\{z_3\}\!\},\\
&\ \widetilde{F}^{(2)}_{1}=\widehat{\rho h}\{\!\{z_3\}\!\}^2+\frac{\{\!\{z_1\}\!\}}{\{\!\{z_2\}\!\}},\\
&\ \widetilde{F}^{(3)}_{1}=\widehat{\rho h}\{\!\{\gamma\}\!\}\{\!\{z_3\}\!\},
\end{aligned}
\right.
\end{equation*}
which leads to \eqref{1D2ptECFlux}. Next, we verify that $\mathcal E$ defined in \eqref{1.8} can be 
reformulated as \eqref{1.8:int}. 
Using \eqref{generalEntropyVar} and \eqref{S_Def}, we recast $\mathcal{E}$ as 
 \begin{equation} \label{1.8:1}
 	\begin{aligned}
 		{\mathcal E} &= \frac{[\![W_1]\!]-[\![\ln z_1]\!]}{[\![z_2]\!]}\overset{\eqref{generalEntropyVar}}{=}
 		\frac{[\![z_2h-S]\!]-[\![\ln z_1]\!]}{[\![z_2]\!]}\\
 		&\overset{\eqref{S_Def}}{=}
 		\frac{\left[\!\left[z_2h-\int^{\frac{1}{z_2}} \frac{e'(x)}{x}dx\right]\!\right]}{[\![z_2]\!]}\overset{\eqref{eq:gEOS}}{=}:\frac{[\![\mathcal{F}(z_2)]\!]}{[\![z_2]\!]},
 	\end{aligned}
 \end{equation}
where $\mathcal{F}(z_2) := z_2h (z_2)-\int^{\frac{1}{z_2}} \frac{e'(x)}{x}dx $ is a function of the parameter variable $z_2$ with its derivative given by 
\begin{equation} \label{F_dev}
	\mathcal{F}'(z_2)=h\left(\frac{1}{z_2}\right)-\frac{1}{z_2}\overset{\eqref{eq:gEOS}}{=}1+e\left(\frac{1}{z_2}\right).
\end{equation}
Note that 
\begin{equation*}
	[\![\mathcal{F}(z_2)]\!]=\mathcal{F}(z_{2,R})-\mathcal{F}(z_{2,L})=\int_0^1 \mathcal{F}'(z_{2,L}+s(z_{2,R}-z_{2,L}))(z_{2,R}-z_{2,L})ds,
\end{equation*}
from which we can deduce that
\begin{equation*} 
	\begin{aligned}
		{\mathcal E} &= \frac{[\![\mathcal{F}(z_2)]\!]}{[\![z_2]\!]}=
		\int_0^1\mathcal{F}'(z_{2,L}+s(z_{2,R}-z_{2,L}))ds\\
		&\overset{\eqref{F_dev}}{=}1+\int_0^1e\left(\frac{1}{z_{2,L}+s(z_{2,R}-z_{2,L})}\right)ds.
	\end{aligned}
\end{equation*}
Hence, we obtain \eqref{1.8:int} and complete the proof.

\end{proof}

Theorem \ref{thm:ECflux} provides a unified formula of the two-point EC flux for 1D RHD with general Synge-type EOS.  
Note that the quantity $\mathcal E$ involved in the formula requires the evaluation of an integral $\int_0^1e\left(\frac{1}{z_{2,L}+s(z_{2,R}-z_{2,L})}\right)ds$, which depends on the specific form of the adopted EOS. 
In order to exactly achieve the EC property, this integral should be calculated exactly. 
For some EOSs, it may be difficult to explicitly express this integral, if the function $e(\theta)$ is very complicated. 
In the following, we provide an alternative way to derive the explicit forms of $\mathcal E$ for four special EOSs.

	\begin{theorem}\label{thm:1De}
		For the 1D RHD equations with ID-EOS \eqref{ID-EOS}, RC-EOS \eqref{hEOS1}, IP-EOS \eqref{hEOS2}, and TM-EOS \eqref{hEOS3}, the two-point EC numerical fluxes are all of the unified form in \eqref{1D2ptECFlux}, 
		where  ${\mathcal E}$ can be calculated as follows:  
		\begin{itemize}
			\item For ID-EOS \eqref{ID-EOS}, we have 
			\begin{equation} \label{e_hat:ideal}
				{\mathcal E}=1+\frac{1}{(\Gamma-1)\{\!\{z_2\}\!\}_{\ln}}.
			\end{equation}
			\item For RC-EOS \eqref{hEOS1}, we have 
			\begin{equation} \label{e_hat:1}
				{\mathcal E}=1+\frac{3}{\{\!\{z_2\}\!\}_{\ln}} - \frac{3}{\{\!\{2z_2+3\}\!\}_{\ln}}.
			\end{equation}
			\item For IP-EOS \eqref{hEOS2}, we have 
			\begin{equation}\label{e_hat:2}
				{\mathcal E}=\frac{1}{\{\!\{z_2\}\!\}_{\ln}}+\left\{\!\left\{\sqrt{1+\frac{4}{z_2^2}}\right\}\!\right\}-\frac{\left\{\!\left\{\frac{2}{z_2}\right\}\!\right\}}{\{\!\{z_2\}\!\}}\left(\frac{\{\!\{z_2\}\!\}\left\{\!\left\{\frac{2}{z_2}\right\}\!\right\}}{\left\{\!\left\{\sqrt{1+\frac{4}{z_2^2}}\right\}\!\right\}}-\frac{2\left(1+\frac{\left\{\!\left\{\frac{2}{z_2}\right\}\!\right\}}{\left\{\!\left\{\sqrt{1+\frac{4}{z_2^2}}\right\}\!\right\}}\right)}{\left\{\!\left\{\frac{2}{z_2}+\sqrt{1+\frac{4}{z_2^2}}\right\}\!\right\}_{\ln}}\right).
			\end{equation}
			\item For TM-EOS \eqref{hEOS3}, we have 
			\begin{equation}\label{e_hat:3}
				{\mathcal E}=\frac{3}{2\{\!\{z_2\}\!\}_{\ln}}+\left\{\!\left\{\sqrt{1+\frac{9}{4z_2^2}}\right\}\!\right\}-\frac{\left\{\!\left\{\frac{3}{2z_2}\right\}\!\right\}}{\{\!\{z_2\}\!\}}\left(\frac{\{\!\{z_2\}\!\}\left\{\!\left\{\frac{3}{2z_2}\right\}\!\right\}}{\left\{\!\left\{\sqrt{1+\frac{9}{4z_2^2}}\right\}\!\right\}}-\frac{3}{2}\frac{1+\frac{\left\{\!\left\{\frac{3}{2z_2}\right\}\!\right\}}{\left\{\!\left\{\sqrt{1+\frac{9}{4z_2^2}}\right\}\!\right\}}}{\left\{\!\left\{\frac{3}{2z_2}+\sqrt{1+\frac{9}{4z_2^2}}\right\}\!\right\}_{\ln}}\right).
			\end{equation}
		\end{itemize}
	\end{theorem}
	
	\begin{proof}
		We verify the formulas of ${\mathcal E}$ for the four special EOSs separately, by first rewriting the jump of $W_1$  and then substituting it into \eqref{1.8} to calculate ${\mathcal E}$. One can also follow \eqref{1.8:int} to give a direct calculation of $\mathcal{E}$. 
		
	 \textbf{ID-EOS}: From \eqref{generalEntropyVar} and \eqref{IDS}, we can reformulate the jump of $W_1$ as follows:
			\begin{equation*}
				\begin{aligned}
					[\![W_1]\!] &=\left[\!\left[\frac{\Gamma}{\Gamma-1}+z_2 -S_{ID}\right]\!\right]=[\![z_2]\!]-[\![{S}_{ID}]\!].\\
					&=[\![z_2]\!]-\left[\!\left[-\ln(z_1)-\frac{1}{\Gamma-1}\ln(z_2)\right]\!\right]\\
					&=[\![z_2]\!]+[\![\ln(z_1)]\!]+\frac{1}{\Gamma-1}[\![\ln(z_2)]\!].
				\end{aligned}
			\end{equation*}
			Substituting it into \eqref{1.8}, we have
			\begin{equation} \label{ID1.8}
				\begin{aligned}
					{\mathcal E}&=\frac{[\![z_2]\!]+[\![\ln(z_1)]\!]+\frac{1}{\Gamma-1}[\![\ln(z_2)]\!]-[\![\ln(z_1)]\!]}{[\![z_2]\!]}\\
					&=1+\frac{[\![\ln(z_2)]\!]}{(\Gamma-1)[\![z_2]\!]}
					\\
					&
					\overset{\eqref{logAVG}}{=}1+\frac{1}{(\Gamma-1)\{\!\{z_2\}\!\}_{\ln}}.
				\end{aligned}    
			\end{equation} 
			Since the function $e(\theta)=\frac{\theta}{\Gamma-1}$ is very simple, the explicit form of 
			  $\mathcal{E}$ can also be easily derived by using the integral formulation \eqref{1.8:int}: 
			\begin{align*}
			    \mathcal{E} &= 1+\frac{1}{\Gamma-1}\int_0^1\frac{1}{z_{2,L}+s(z_{2,R}-z_{2,L})}ds\\
			    &=1+\frac{[\![\ln(z_2)]\!]}{[\![z_2]\!]}\overset{\eqref{logAVG}}{=}1+\frac{1}{(\Gamma-1)\{\!\{z_2\}\!\}_{\ln}},
			\end{align*}
			which is consistent with \eqref{ID1.8}.

			\textbf{RC-EOS}: From \eqref{generalEntropyVar}, we can derive the jump term of $W_1$ by
			\begin{equation*}
				\begin{aligned}
					[\![W_1]\!]&=\left[\!\left[z_2h-{S}_{RC}\right]\!\right]\\
					&=\left[\!\left[4+z_2+\ln(z_1)+3\ln(z_2)-\frac{3}{2}\ln(2z_2+3)\right]\!\right]\\
					&=[\![z_2]\!]+[\![\ln(z_1)]\!]+3[\![\ln(z_2)]\!]-\frac{3}{2}\left[\!\left[\ln(2z_2+3)\right]\!\right].
				\end{aligned}
			\end{equation*}
			Substituting it into \eqref{1.8}, we have
			\begin{equation} \label{RC1.8}
				\begin{aligned}
					{\mathcal E}&=\frac{[\![z_2]\!]+[\![\ln(z_1)]\!]+3[\![\ln(z_2)]\!]-\frac{3}{2}[\![\ln(2z_2+3)]\!]-[\![\ln(z_1)]\!]}{[\![z_2]\!]}\\
					&=1+3\frac{[\![\ln(z_2)]\!]}{[\![z_2]\!]}-\frac{[\![\ln(2z_2+3)]\!]}{[\![z_2]\!]}
					\\
					&
					\overset{\eqref{logAVG}}{=}1+\frac{3}{\{\!\{z_2\}\!\}_{\ln}}-\frac{3}{\{\!\{2z_2+3\}\!\}_{\ln}}.
				\end{aligned}    
			\end{equation}
			On the other hand,	
				since the function $e(\theta)=\frac{3\theta(3\theta+1)}{3\theta+2}$ is fairly simple, the explicit form of 
				$\mathcal{E}$ can also be easily derived by using the integral formulation \eqref{1.8:int}: 
			\begin{align*}
			    \mathcal{E} &= 1+\int_0^1\frac{3(3+z_{2,L}+s(z_{2,R}-z_{2,L}))}{(2(z_{2,L}+s(z_{2,R}-z_{2,L}))+3)(z_{2,L}+s(z_{2,R}-z_{2,L}))}ds\\
			    &=1+\int_0^1\left(\frac{3}{z_{2,L}+s(z_{2,R}-z_{2,L})}-\frac{3}{2(z_{2,L}+s(z_{2,R}-z_{2,L}))+3}\right)ds\\
			    &=1+\frac{3[\![\ln(z_2)]\!]}{[\![z_2]\!]}-\frac{3[\![\ln(2z_2+3)]\!]}{2[\![z_2]\!]}\overset{\eqref{logAVG}}{=}1+\frac{3}{\{\!\{z_2\}\!\}_{\ln}}-\frac{3}{\{\!\{2z_2+3\}\!\}_{\ln}},
			\end{align*}
			which is consistent with \eqref{RC1.8}.

		 \textbf{IP-EOS}: From \eqref{generalEntropyVar}, we can derive the jump term of $W_1$ by
			\begin{align}
			    \nonumber
				[\![W_1]\!]&=\left[\!\left[z_2h-{S}_{IP}\right]\!\right]
				\\ \label{IP1.8}
				&=\left[\!\left[2+z_2\sqrt{1+\frac{4}{z_2^2}}+\ln(z_1)+\ln(z_2)-2\ln\left(\frac{2}{z_2}+\sqrt{1+\frac{4}{z_2^2}}\right)\right]\!\right]
				\\ \nonumber
				&\overset{\eqref{product}}{=}[\![z_2]\!]\left\{\!\left\{\sqrt{1+\frac{4}{z_2^2}}\right\}\!\right\}+\{\!\{z_2\}\!\}\left[\!\left[\sqrt{1+\frac{4}{z_2^2}}\right]\!\right]+[\![\ln(z_1)]\!]+[\![\ln(z_2)]\!]-2\left[\!\left[\ln\left(\frac{2}{z_2}+\sqrt{1+\frac{4}{z_2^2}}\right)\right]\!\right]
				\\ \nonumber
				&\overset{\eqref{logAVG}}{=}[\![z_2]\!]\left\{\!\left\{\sqrt{1+\frac{4}{z_2^2}}\right\}\!\right\}+\{\!\{z_2\}\!\}\left[\!\left[\sqrt{1+\frac{4}{z_2^2}}\right]\!\right]+[\![\ln(z_1)]\!]+[\![\ln(z_2)]\!]-2\frac{\left[\!\left[\frac{2}{z_2}+\sqrt{1+\frac{4}{z_2^2}}\right]\!\right]}{\left\{\!\left\{\frac{2}{z_2}+\sqrt{1+\frac{4}{z_2^2}}\right\}\!\right\}_{\ln}}
				\\ \nonumber
				&\overset{\eqref{sqrt}}{=}[\![z_2]\!]\left\{\!\left\{\sqrt{1+\frac{4}{z_2^2}}\right\}\!\right\}+\{\!\{z_2\}\!\}\frac{\left[\!\left[1+\left(\frac{2}{z_2}\right)^2\right]\!\right]}{2\left\{\!\left\{\sqrt{1+\frac{4}{z_2^2}}\right\}\!\right\}}+[\![\ln(z_1)]\!]+[\![\ln(z_2)]\!]
				\\ \nonumber
				&\qquad-2\frac{\left[\!\left[\frac{2}{z_2}\right]\!\right]}{\left\{\!\left\{\frac{2}{z_2}+\sqrt{1+\frac{4}{z_2^2}}\right\}\!\right\}_{\ln}}-\frac{2}{\left\{\!\left\{\frac{2}{z_2}+\sqrt{1+\frac{4}{z_2^2}}\right\}\!\right\}_{\ln}}\frac{\left[\!\left[1+\left(\frac{2}{z_2}\right)^2\right]\!\right]}{2\left\{\!\left\{\sqrt{1+\frac{4}{z_2^2}}\right\}\!\right\}}
				\\ \nonumber
				&\overset{\eqref{square}}{=}[\![z_2]\!]\left\{\!\left\{\sqrt{1+\frac{4}{z_2^2}}\right\}\!\right\}+\{\!\{z_2\}\!\}\frac{\left\{\!\left\{\frac{2}{z_2}\right\}\!\right\}\left[\!\left[\frac{2}{z_2}\right]\!\right]}{\left\{\!\left\{\sqrt{1+\frac{4}{z_2^2}}\right\}\!\right\}}+[\![\ln(z_1)]\!]+[\![\ln(z_2)]\!]-2\frac{\left[\!\left[\frac{2}{z_2}\right]\!\right]}{\left\{\!\left\{\frac{2}{z_2}+\sqrt{1+\frac{4}{z_2^2}}\right\}\!\right\}_{\ln}}
				\\ \nonumber
				&\qquad -\frac{2}{\left\{\!\left\{\frac{2}{z_2}+\sqrt{1+\frac{4}{z_2^2}}\right\}\!\right\}_{\ln}}\frac{\left\{\!\left\{\frac{2}{z_2}\right\}\!\right\}\left[\!\left[\frac{2}{z_2}\right]\!\right]}{\left\{\!\left\{\sqrt{1+\frac{4}{z_2^2}}\right\}\!\right\}}
				\\ \nonumber
				&\overset{\eqref{division}}{=}[\![z_2]\!]\left\{\!\left\{\sqrt{1+\frac{4}{z_2^2}}\right\}\!\right\}-\left(\frac{\{\!\{z_2\}\!\}\left\{\!\left\{\frac{2}{z_2}\right\}\!\right\}}{\left\{\!\left\{\sqrt{1+\frac{4}{z_2^2}}\right\}\!\right\}}-\frac{2\left(1+\frac{\left\{\!\left\{\frac{2}{z_2}\right\}\!\right\}}{\left\{\!\left\{\sqrt{1+\frac{4}{z_2^2}}\right\}\!\right\}}\right)}{\left\{\!\left\{\frac{2}{z_2}+\sqrt{1+\frac{4}{z_2^2}}\right\}\!\right\}_{\ln}}\right)\frac{[\![z_2]\!]\left\{\!\left\{\frac{2}{z_2}\right\}\!\right\}}{\{\!\{z_2\}\!\}}
				\\ \nonumber
				&\qquad+[\![\ln(z_1)]\!]+[\![\ln(z_2)]\!].
			\end{align}
			Substituting it into \eqref{1.8}, we have
			\begin{align*}
					{\mathcal E}&=\frac{[\![z_2]\!]\left\{\!\left\{\sqrt{1+\frac{4}{z_2^2}}\right\}\!\right\}-\left(\frac{\{\!\{z_2\}\!\}\left\{\!\left\{\frac{2}{z_2}\right\}\!\right\}}{\left\{\!\left\{\sqrt{1+\frac{4}{z_2^2}}\right\}\!\right\}}-\frac{2\left(1+\frac{\left\{\!\left\{\frac{2}{z_2}\right\}\!\right\}}{\left\{\!\left\{\sqrt{1+\frac{4}{z_2^2}}\right\}\!\right\}}\right)}{\left\{\!\left\{\frac{2}{z_2}+\sqrt{1+\frac{4}{z_2^2}}\right\}\!\right\}_{\ln}}\right)\frac{[\![z_2]\!]\left\{\!\left\{\frac{2}{z_2}\right\}\!\right\}}{\{\!\{z_2\}\!\}}+[\![\ln(z_2)]\!]}{[\![z_2]\!]}\\
					&\overset{\eqref{logAVG}}{=}\left\{\!\left\{\sqrt{1+\frac{4}{z_2^2}}\right\}\!\right\}-\frac{\left\{\!\left\{\frac{2}{z_2}\right\}\!\right\}}{\{\!\{z_2\}\!\}}\left(\frac{\{\!\{z_2\}\!\}\left\{\!\left\{\frac{2}{z_2}\right\}\!\right\}}{\left\{\!\left\{\sqrt{1+\frac{4}{z_2^2}}\right\}\!\right\}}-\frac{2\left(1+\frac{\left\{\!\left\{\frac{2}{z_2}\right\}\!\right\}}{\left\{\!\left\{\sqrt{1+\frac{4}{z_2^2}}\right\}\!\right\}}\right)}{\left\{\!\left\{\frac{2}{z_2}+\sqrt{1+\frac{4}{z_2^2}}\right\}\!\right\}_{\ln}}\right)+\frac{1}{\{\!\{z_2\}\!\}_{\ln}}.
			\end{align*}

			 \textbf{TM-EOS}: From \eqref{generalEntropyVar}, we can derive the jump term of $W_1$ by
			\begin{align}
			    \nonumber
				[\![W_1]\!]&=\left[\!\left[z_2h-{S}_{TM}\right]\!\right]
				\\ \label{TM1.8}
				&=\left[\!\left[\frac{5}{2}+z_2\sqrt{1+\frac{9}{4z_2^2}}+\ln(z_1)+\frac{3}{2}\ln(z_2)-\frac{3}{2}\ln\left(\frac{3}{2z_2}+\sqrt{1+\frac{9}{4z_2^2}}\right)\right]\!\right]
				\\ \nonumber
				&\overset{\eqref{product}}{=}[\![z_2]\!]\left\{\!\left\{\sqrt{1+\frac{9}{4z_2^2}}\right\}\!\right\}+\{\!\{z_2\}\!\}\left[\!\left[\sqrt{1+\frac{9}{4z_2^2}}\right]\!\right]+[\![\ln(z_1)]\!]+\frac{3}{2}[\![\ln(z_2)]\!]
				\\ \nonumber
				&\qquad{-\frac{3}{2}\left[\!\left[\ln\left(\frac{3}{2z_2}+\sqrt{1+\frac{9}{4z_2^2}}\right)\right]\!\right]}
				\\ \nonumber
				&\overset{\eqref{logAVG}}{=}[\![z_2]\!]\left\{\!\left\{\sqrt{1+\frac{9}{4z_2^2}}\right\}\!\right\}+\{\!\{z_2\}\!\}\left[\!\left[\sqrt{1+\frac{9}{4z_2^2}}\right]\!\right]+[\![\ln(z_1)]\!]+\frac{3}{2}[\![\ln(z_2)]\!]-\frac{3}{2}\frac{\left[\!\left[\frac{3}{2z_2}+\sqrt{1+\frac{9}{4z_2^2}}\right]\!\right]}{\left\{\!\left\{\frac{3}{2z_2}+\sqrt{1+\frac{9}{4z_2^2}}\right\}\!\right\}_{\ln}}
				\\ \nonumber
				&\overset{\eqref{sqrt}}{=}[\![z_2]\!]\left\{\!\left\{\sqrt{1+\frac{9}{4z_2^2}}\right\}\!\right\}+\{\!\{z_2\}\!\}\frac{\left[\!\left[1+\left(\frac{3}{2z_2}\right)^2\right]\!\right]}{2\left\{\!\left\{\sqrt{1+\frac{9}{4z_2^2}}\right\}\!\right\}}+[\![\ln(z_1)]\!]+\frac{3}{2}[\![\ln(z_2)]\!]
				\\ \nonumber
				&\qquad-\frac{3}{2}\frac{\left[\!\left[\frac{3}{2z_2}\right]\!\right]}{\left\{\!\left\{\frac{3}{2z_2}+\sqrt{1+\frac{9}{4z_2^2}}\right\}\!\right\}_{\ln}}-\frac{3}{2\left\{\!\left\{\frac{3}{2z_2}+\sqrt{1+\frac{9}{4z_2^2}}\right\}\!\right\}_{\ln}}\frac{\left[\!\left[1+\left(\frac{3}{2z_2}\right)^2\right]\!\right]}{2\left\{\!\left\{\sqrt{1+\frac{9}{4z_2^2}}\right\}\!\right\}}
				\\ \nonumber
				&\overset{\eqref{square}}{=}[\![z_2]\!]\left\{\!\left\{\sqrt{1+\frac{9}{4z_2^2}}\right\}\!\right\}+\{\!\{z_2\}\!\}\frac{\left\{\!\left\{\frac{3}{2z_2}\right\}\!\right\}\left[\!\left[\frac{3}{2z_2}\right]\!\right]}{\left\{\!\left\{\sqrt{1+\frac{9}{4z_2^2}}\right\}\!\right\}}+[\![\ln(z_1)]\!]+\frac{3}{2}[\![\ln(z_2)]\!]-\frac{3}{2}\frac{\left[\!\left[\frac{3}{2z_2}\right]\!\right]}{\left\{\!\left\{\frac{3}{2z_2}+\sqrt{1+\frac{9}{4z_2^2}}\right\}\!\right\}_{\ln}}
				\\ \nonumber
				&\qquad -\frac{3}{2\left\{\!\left\{\frac{3}{2z_2}+\sqrt{1+\frac{9}{4z_2^2}}\right\}\!\right\}_{\ln}}\frac{\left\{\!\left\{\frac{3}{2z_2}\right\}\!\right\}\left[\!\left[\frac{3}{2z_2}\right]\!\right]}{\left\{\!\left\{\sqrt{1+\frac{9}{4z_2^2}}\right\}\!\right\}}
				\\ \nonumber
				&\overset{\eqref{division}}{=}[\![z_2]\!]\left\{\!\left\{\sqrt{1+\frac{9}{4z_2^2}}\right\}\!\right\}-\left(\frac{\{\!\{z_2\}\!\}\left\{\!\left\{\frac{3}{2z_2}\right\}\!\right\}}{\left\{\!\left\{\sqrt{1+\frac{9}{4z_2^2}}\right\}\!\right\}}-\frac{\frac{3}{2}\left(1+\frac{\left\{\!\left\{\frac{3}{2z_2}\right\}\!\right\}}{\left\{\!\left\{\sqrt{1+\frac{9}{4z_2^2}}\right\}\!\right\}}\right)}{\left\{\!\left\{\frac{3}{2z_2}+\sqrt{1+\frac{9}{4z_2^2}}\right\}\!\right\}_{\ln}}\right)\frac{[\![z_2]\!]\left\{\!\left\{\frac{3}{2z_2}\right\}\!\right\}}{\{\!\{z_2\}\!\}}
				\\ \nonumber
				&\qquad+[\![\ln(z_1)]\!]+\frac{3}{2}[\![\ln(z_2)]\!].
			\end{align}
			Substituting it into \eqref{1.8}, we have
			\begin{align*}
				{\mathcal E}&=\frac{[\![z_2]\!]\left\{\!\left\{\sqrt{1+\frac{9}{4z_2^2}}\right\}\!\right\}-\left(\frac{\{\!\{z_2\}\!\}\left\{\!\left\{\frac{3}{2z_2}\right\}\!\right\}}{\left\{\!\left\{\sqrt{1+\frac{9}{4z_2^2}}\right\}\!\right\}}-\frac{\frac{3}{2}\left(1+\frac{\left\{\!\left\{\frac{3}{2z_2}\right\}\!\right\}}{\left\{\!\left\{\sqrt{1+\frac{9}{4z_2^2}}\right\}\!\right\}}\right)}{\left\{\!\left\{\frac{3}{2z_2}+\sqrt{1+\frac{9}{4z_2^2}}\right\}\!\right\}_{\ln}}\right)\frac{[\![z_2]\!]\left\{\!\left\{\frac{3}{2z_2}\right\}\!\right\}}{\{\!\{z_2\}\!\}}+\frac{3}{2}[\![\ln(z_2)]\!]}{[\![z_2]\!]}\\
				&\overset{\eqref{logAVG}}{=}\left\{\!\left\{\sqrt{1+\frac{9}{4z_2^2}}\right\}\!\right\}-\frac{\left\{\!\left\{\frac{3}{2z_2}\right\}\!\right\}}{\{\!\{z_2\}\!\}}\left(\frac{\{\!\{z_2\}\!\}\left\{\!\left\{\frac{3}{2z_2}\right\}\!\right\}}{\left\{\!\left\{\sqrt{1+\frac{9}{4z_2^2}}\right\}\!\right\}}-\frac{\frac{3}{2}\left(1+\frac{\left\{\!\left\{\frac{3}{2z_2}\right\}\!\right\}}{\left\{\!\left\{\sqrt{1+\frac{9}{4z_2^2}}\right\}\!\right\}}\right)}{\left\{\!\left\{\frac{3}{2z_2}+\sqrt{1+\frac{9}{4z_2^2}}\right\}\!\right\}_{\ln}}\right)+\frac{3}{2\{\!\{z_2\}\!\}_{\ln}}.
			\end{align*}

	The proof is completed. 
	\end{proof}

\begin{remark}
	It is easy to verify that the two-point EC numerical flux $\widetilde{\textbf{F}}^{EC}_1 ( {\mathbf U}_L, {\mathbf U}_R )$ in \eqref{1D2ptECFlux} is consistent with the flux function ${\bf F}_1({\bf U})$ in the {\rm 1D RHD} equations, and ${\mathcal E}-1=\int_0^1e\left(\frac{1}{z_{2,L}+s(z_{2,R}-z_{2,L})}\right)ds$ is consistent with the specific internal energy $e$. 
\end{remark}
\begin{remark}
	It is worth mentioning that 
	the choice of parameter variables in \eqref{1DalgeVar} follows \cite{wu2020entropy} and is different from \cite{duan2019high}. 
Thanks to this choice, we obtain the EC numerical flux $\widetilde{\textbf{F}}^{EC}_1 ( {\mathbf U}_L, {\mathbf U}_R )$ in a unified form \eqref{1D2ptECFlux} for general Synge-type EOS. Moreover, in the case of ID-EOS, 
the expressions of our EC numerical flux are simpler than those obtained in \cite{duan2019high} via a different set of  parameter variables. 
\end{remark}

\subsection{Entropy conservative schemes}
In this subsection, we construct EC schemes for the 1D RHD equations. To avoid confusing subscripts, we use $x$ to denote the 1D spatial coordinate, $\mathbf{F}$ to denote the 1D flux function ${\mathbf F}_1$, and $q$ to represent the entropy flux associated with the convex entropy function $\eta$ in the $x$-direction. The spatial domain is divided into cells $I_i=(x_{i-\frac{1}{2}},x_{i+\frac{1}{2}})$ by uniform meshes $x_1<x_2<\cdots<x_N$, with the mesh size $\Delta x=x_{i+1}-x_{i}$. A semi-discrete finite difference scheme of the 1D RHD equations can be written as
\begin{equation} \label{2.1}
	\frac{d}{dt}\mathbf{U}_i(t) +\frac{\widehat{\mathbf{F}}_{i+\frac{1}{2}}(t)-\widehat{\mathbf{F}}_{i-\frac{1}{2}}(t)}{\Delta x}=0,
\end{equation}
where $\mathbf{U}_i(t)\approx \mathbf{U}(x_i,t)$, and the numerical flux $\widehat{\mathbf{F}}$ is consistent with the flux $\mathbf{F}$.

\begin{definition}[\cite{tadmor2003entropy}]
	The semi-discrete scheme \eqref{2.1} is EC if its numerical solutions satisfy a discrete entropy equality
	\begin{equation} \label{2.2}
		\frac{d}{dt}\eta(\mathbf{U}_i) +\frac{1}{\Delta x}\left(\widetilde{q}_{i+\frac{1}{2}}-\widetilde{q}_{i-\frac{1}{2}}\right)=0
	\end{equation}
	for some numerical entropy flux $\widetilde{q}$ consistent with the entropy flux $q$.
\end{definition}

\subsubsection{Second-order entropy conservative schemes}

If we take $\widehat{F}_{i+\frac{1}{2}}$ as the two-point EC numerical flux  $\widetilde{\mathbf{F}}^{EC}_1(\mathbf{U}_i,\mathbf{U}_{i+1})$ in \eqref{1D2ptECFlux}, then the scheme \eqref{2.1} becomes
\begin{equation} \label{2.3}
	\frac{d\mathbf{U}_i}{dt} =-\frac{\widetilde{\mathbf{F}}^{EC}_1(\mathbf{U}_i,\mathbf{U}_{i+1})-\widetilde{\mathbf{F}}^{EC}_1(\mathbf{U}_{i-1},\mathbf{U}_i)}{\Delta x},
\end{equation}
which is second-order accurate. 
 To verify the EC property of this scheme, we 
 follow the framework of Tadmor \cite{tadmor1987numerical,tadmor2003entropy} to 
  show the discrete entropy equality \eqref{2.2}. Note that 
\begin{equation} \label{2ndECComTemp}
	\frac{d}{dt}\eta(\mathbf{U}_i)=\eta'(\mathbf{U}_i)\frac{d\mathbf{U}_i}{dt}=-\mathbf{W}_i^\top\left(\frac{\widetilde{\mathbf{F}}^{EC}_1(\mathbf{U}_i,\mathbf{U}_{i+1})-\widetilde{\mathbf{F}}^{EC}_1(\mathbf{U}_{i-1},\mathbf{U}_i)}{\Delta x}\right). 
\end{equation}
Recalling that the definition of jump and arithmetic average operators in \eqref{jump} and \eqref{arithmeticAVG}, one obtains 
\begin{equation} \label{notationTemp}
	\mathbf{W}_i=\{\!\{\mathbf{W}\}\!\}_{i+\frac{1}{2}}-\frac{1}{2}[\![\mathbf{W}]\!]_{i+\frac{1}{2}}=\{\!\{\mathbf{W}\}\!\}_{i-\frac{1}{2}}+\frac{1}{2}[\![\mathbf{W}]\!]_{i-\frac{1}{2}}.
\end{equation}
Combining \eqref{2ndECComTemp}--\eqref{notationTemp} with the property \eqref{ECdef} of the two-point EC flux, we have
\begin{equation*}
	\begin{aligned}
		\frac{d}{dt}\eta(\mathbf{U}_i)&=-\frac{\Big(\{\!\{\mathbf{W}\}\!\}_{i+\frac{1}{2}}-\frac{1}{2}[\![\mathbf{W}]\!]_{i+\frac{1}{2}}\Big)^\top\widetilde{\mathbf{F}}^{EC}_1(\mathbf{U}_i,\mathbf{U}_{i+1})-\Big(\{\!\{\mathbf{W}\}\!\}_{i-\frac{1}{2}}+\frac{1}{2}[\![\mathbf{W}]\!]_{i-\frac{1}{2}}\Big)^\top\widetilde{\mathbf{F}}^{EC}_1(\mathbf{U}_{i-1},\mathbf{U}_i)}{\Delta x}\\
		&=-\frac{\{\!\{\mathbf{W}\}\!\}_{i+\frac{1}{2}}^\top\widetilde{\mathbf{F}}^{EC}_1(\mathbf{U}_i,\mathbf{U}_{i+1})-\{\!\{\mathbf{W}\}\!\}_{i-\frac{1}{2}}^\top\widetilde{\mathbf{F}}^{EC}_1(\mathbf{U}_{i-1},\mathbf{U}_i)-\frac{1}{2}\Big([\![\psi_1]\!]_{i+\frac{1}{2}}+[\![\psi_1]\!]_{i-\frac{1}{2}}\Big)}{\Delta x}\\
		&=-\frac{1}{\Delta x}\left(\left(\{\!\{\mathbf{W}\}\!\}_{i+\frac{1}{2}}^\top\widetilde{\mathbf{F}}^{EC}_1(\mathbf{U}_i,\mathbf{U}_{i+1})-\{\!\{\psi_1\}\!\}_{i+\frac{1}{2}}\right)-\left(\{\!\{\mathbf{W}\}\!\}_{i-\frac{1}{2}}^\top\widetilde{\mathbf{F}}^{EC}_1(\mathbf{U}_{i-1},\mathbf{U}_i)-\{\!\{\psi_1\}\!\}_{i-\frac{1}{2}}\right)\right)\\
	\end{aligned}    
\end{equation*}
If we take the numerical entropy flux $\widetilde{q}_{i+\frac{1}{2}}$ as 
\begin{equation} \label{2ndECentropyFlux}
	  \widetilde{q} (\mathbf{U}_i,\mathbf{U}_{i+1}) :=\{\!\{\mathbf{W}\}\!\}_{i+\frac{1}{2}}^\top \widetilde{\mathbf{F}}^{EC}_1(\mathbf{U}_i,\mathbf{U}_{i+1})-\{\!\{\psi_1\}\!\}_{i+\frac{1}{2}},
\end{equation}
then the discrete entropy equality \eqref{2.2} is satisfied. Moreover, the numerical entropy flux   \eqref{2ndECentropyFlux} is clearly consistent with the entropy flux $q$. Therefore, the scheme \eqref{2.3} is a second-order EC scheme with the corresponding numerical entropy flux given by \eqref{2ndECentropyFlux}.

\subsubsection{High-order entropy conservative schemes}
The semi-discrete EC scheme \eqref{2.3} is only second-order accurate in space. To obtain higher-order EC schemes, we can consider a linear combination of the two-point EC fluxes (cf.~\cite{lefloch2002fully}):
\begin{equation} \label{2.5}
	\widetilde{\mathbf{F}}_{i+\frac{1}{2}}^{2k}=\sum\limits_{r=1}^k\alpha_{k,r}\sum\limits_{s=0}^{r-1}\widetilde{\mathbf{F}}^{EC}_1(\mathbf{U}_{i-s},\mathbf{U}_{i-s+r}),
\end{equation} 
where the constants $\{\alpha_{k,r}\}$ satisfy
\begin{equation} \label{2.6}
	\sum\limits_{r=1}^kr\alpha_{k,r}=1,\quad\quad\sum\limits_{r=1}^{k}r^{2s-1}\alpha_{k,r}=0,\ s=2,\cdots,k.
\end{equation}
For example, the fourth- and sixth-order EC fluxes are given by
\begin{equation*}
	\begin{aligned}
		\widetilde{\mathbf{F}}_{i+\frac{1}{2}}^4=&\frac{4}{3}\widetilde{\mathbf{F}}^{EC}_1(\mathbf{U}_i,\mathbf{U}_{i+1})-\frac{1}{6}\left(\widetilde{\mathbf{F}}^{EC}_1(\mathbf{U}_i,\mathbf{U}_{i+2})+\widetilde{\mathbf{F}}^{EC}_1(\mathbf{U}_{i-1},\mathbf{U}_{i+1})\right),\\
		\widetilde{\mathbf{F}}_{i+\frac{1}{2}}^6=&\frac{3}{2}\widetilde{\mathbf{F}}^{EC}_1(\mathbf{U}_i,\mathbf{U}_{i+1})-\frac{3}{10}\left(\widetilde{\mathbf{F}}^{EC}_1(\mathbf{U}_i,\mathbf{U}_{i+2})+\widetilde{\mathbf{F}}^{EC}_1(\mathbf{U}_{i-1},\mathbf{U}_{i+1})\right)\\+&\frac{1}{30}\left(\widetilde{\mathbf{F}}^{EC}_1(\mathbf{U}_i,\mathbf{U}_{i+3})+\widetilde{\mathbf{F}}^{EC}_1(\mathbf{U}_{i-1},\mathbf{U}_{i+2})+\widetilde{\mathbf{F}}^{EC}_1(\mathbf{U}_{i-2},\mathbf{U}_{i+1})\right).
	\end{aligned}
\end{equation*}

If we take the numerical flux $\widehat{\mathbf{F}}_{i+\frac{1}{2}}$ in \eqref{2.1} as 
$\widetilde{\mathbf{F}}_{i+\frac{1}{2}}^{2k}$, then we obtain a 
$2k$th-order semi-discrete scheme
\begin{equation} \label{2.7}
	\frac{d\mathbf{U}_i}{dt}=-\frac{\widetilde{\mathbf{F}}_{i+\frac{1}{2}}^{2k}-\widetilde{\mathbf{F}}_{i-\frac{1}{2}}^{2k}}{\Delta x}. 
\end{equation}
Following \cite{lefloch2002fully}, one can verify that the high-order 
scheme \eqref{2.7} is EC with the corresponding numerical entropy flux given by
\begin{equation} \label{2.8}
	\widetilde{q}_{i+\frac{1}{2}}^{2k}=\sum\limits_{r=1}^k\alpha_{k,r}\sum\limits_{s=0}^{r-1}\widetilde{q}(\mathbf{U}_{i-s},\mathbf{U}_{i-s+r}),
\end{equation}
where the constants $\{\alpha_{k,r}\}$ are defined in \eqref{2.6}, and $\widetilde{q}$ is defined in \eqref{2ndECentropyFlux}.

\subsection{Entropy stable schemes}

The above EC schemes may produce oscillations when solutions of the RHD equations contain discontinuities. 
Hence, we need to add some dissipation terms to  
guarantee the entropy stability \cite{tadmor1987numerical,tadmor2003entropy}.

If the numerical flux in \eqref{2.1} is taken as
\begin{equation} \label{3.1}
	\widehat{\mathbf F}_{i+\frac{1}{2}}=\widetilde{\mathbf F}_{i+\frac{1}{2}}-\frac{1}{2}\mathbf{D}_{i+\frac{1}{2}}\ [\![\mathbf{W}]\!]_{i+\frac{1}{2}},
\end{equation}
where $\widetilde{\mathbf F}_{i+\frac{1}{2}}$ is an EC flux, and ${\mathbf D}_{i+\frac{1}{2}}$ is a positive semi-definite matrix, 
then the scheme \eqref{2.1} is first-order accurate and satisfies the discrete entropy inequality
\begin{equation} \label{3.2}
	\frac{d}{dt}\eta(\mathbf{U}_i)+\frac{1}{\Delta x}\Big(\widehat{q}_{i+\frac{1}{2}}-\widehat{q}_{i-\frac{1}{2}}\Big)=-\frac{1}{4\Delta x} \Big([\![\mathbf{W}]\!]_{i+\frac{1}{2}}^\top\mathbf{D}_{i+\frac{1}{2}}\ [\![\mathbf{W}]\!]_{i+\frac{1}{2}}+[\![\mathbf{W}]\!]_{i-\frac{1}{2}}^\top\mathbf{D}_{i-\frac{1}{2}}\ [\![\mathbf{W}]\!]_{i-\frac{1}{2}}\Big)\leq 0
\end{equation}
with 
\begin{equation} \label{3.4}
	\widehat{q}_{i+\frac{1}{2}} :=\widetilde{q}_{i+\frac{1}{2}}-\frac{1}{2}\{\!\{\mathbf{W}\}\!\}_{i+\frac{1}{2}}^\top\mathbf{D}_{i+\frac{1}{2}}[\![\mathbf{W}]\!]_{i+\frac{1}{2}}.
\end{equation}
Therefore, the resulting scheme \eqref{2.1} with \eqref{3.1} is ES, and the corresponding numerical entropy flux is given by \eqref{3.4}. 

In the following, we will discuss how to define the positive semi-definite matrix ${\mathbf D}_{i+\frac{1}{2}}$ and how to generalize the first-order ES scheme to design high-order ES schemes.

\subsubsection{Dissipation matrix}
The positive semi-definite matrix in equation \eqref{3.1} is referred to as the dissipation matrix. It can be defined as follows:
\begin{equation} \label{3.3}
	\mathbf{D}_{i+\frac{1}{2}}:=\mathbf{R}_{i+\frac{1}{2}}\left|\mathbf{\Lambda}_{i+\frac{1}{2}}\right|\mathbf{R}_{i+\frac{1}{2}}^\top,
\end{equation}
where the matrix $\mathbf{R}$ is formed by the suitably scaled right eigenvectors of the Jacobian matrix $\frac{\partial {\mathbf F}_1({\mathbf U})}{\partial {\mathbf U}}$ of the 1D RHD system, and it satisfies
\begin{equation} \label{scaledMatExist}
	\frac{\partial\mathbf{F}_1}{\partial\mathbf{U}}=\mathbf{R}\mathbf{\Lambda}\mathbf{R}^{-1},\quad\quad \frac{\partial \mathbf{U}}{\partial \mathbf{W}}=\mathbf{R}\mathbf{R}^\top.
\end{equation}
The formula of $\mathbf{R}$ is derived in Theorem \ref{thm:1DDissipationMat}. 
In \eqref{scaledMatExist}, the diagonal matrix $\mathbf{\Lambda} = {\rm diag}\{ \lambda_1, \lambda_2, \lambda_3 \}$, where 
	\begin{equation*}
	\lambda_1 = \frac{v_1-c_s}{1-v_1c_s},\qquad
	\lambda_2 = v_1,\qquad
	\lambda_3 = \frac{v_1+c_s}{1+v_1c_s}.
\end{equation*}
are the three eigenvalues of the Jacobian matrix $\frac{\partial {\mathbf F}_1({\mathbf U})}{\partial {\mathbf U}}$. 
There are two common ways to define $|\mathbf{\Lambda}|$ in \eqref{3.3}:
\begin{equation} \label{Roe-type}
	|\mathbf{\Lambda}| = {\rm diag}\{|\lambda_1|,|\lambda_2|,|\lambda_3|\},
\end{equation}
or 
\begin{equation}\label{Rusanov-type}
	|\mathbf{\Lambda}| = \max\{|\lambda_1|,|\lambda_2|,|\lambda_3|\}\mathbf{I}. 
\end{equation}
The definition \eqref{Roe-type} gives the Roe-type dissipation term, while 
 \eqref{Rusanov-type} leads to the Rusanov-type dissipation term.

\begin{theorem} \label{thm:1DDissipationMat}
For the 1D RHD system with general Synge-type EOS \eqref{eq:gEOS}, the scaled eigenvector matrix ${\mathbf R}$ satisfying \eqref{scaledMatExist} is given by
		\begin{equation} \label{1DDissMat}
			\mathbf{R}:= \tilde{\mathbf{R}}\sqrt{\tilde{\mathbf{D}}}:=
			\begin{pmatrix}
				1 & 1 & 1 \\
				(v_1-c_s)h\gamma & \Big(h-\theta\left(1+e'(\theta)\right)\Big)\gamma v_1 & (v_1+c_s)h\gamma \\
				(1-v_1c_s)h\gamma & \Big(h-\theta\left(1+e'(\theta)\right)\Big)\gamma & (1+v_1c_s)h\gamma 
			\end{pmatrix}
			\begin{pmatrix}
				\sqrt{d_1} & 0 & 0 \\
				0 & \sqrt{d_2} & 0 \\
				0 & 0 & \sqrt{d_3} 
			\end{pmatrix},
		\end{equation}
	where the scaling coefficients $d_j,\ j = 1,2,3$, are defined as  
	\begin{equation} \label{1DdissipationMatrixCoeff}
		\left\{~ 
		\begin{aligned}
			d_1 &= \left(\frac{\rho e'(\theta)}{2(1+e'(\theta))}-\frac{\rho\theta v_1}{2c_sh}\right)\gamma,\\
			d_2 &= \frac{\rho\gamma}{1+e'(\theta)},\\
			d_3 &= \left(\frac{\rho e'(\theta)}{2(1+e'(\theta))}+\frac{\rho\theta v_1}{2c_sh}\right)\gamma.
		\end{aligned}
	\right.
	\end{equation} 
\end{theorem}

	\begin{proof}
		Note that $\tilde{\mathbf{R}}$ defined in \eqref{1DDissMat} is a right eigenvector matrix of the Jacobian matrix $\frac{\partial {\mathbf F}_1({\mathbf U})}{\partial {\mathbf U}}$; see   \cite{zhao2013runge}. 
		Since $\sqrt{\tilde{\mathbf{D}}}$ is a diagonal matrix, we know that 
		$\mathbf{R}= \tilde{\mathbf{R}}\sqrt{\tilde{\mathbf{D}}}$ is also a right eigenvector matrix of  $\frac{\partial {\mathbf F}_1({\mathbf U})}{\partial {\mathbf U}}$. Thus, $\mathbf{R}$ satisfies $ \frac{\partial\mathbf{F}_1}{\partial\mathbf{U}}=\mathbf{R}\mathbf{\Lambda}\mathbf{R}^{-1}$. 
		 We only need to verify that $\mathbf{R}$ satisfies  
		  $\frac{\partial \mathbf{U}}{\partial \mathbf{W}}=\mathbf{R}\mathbf{R}^\top$. 
		  Next, we would like to derive the formula of $\frac{\partial \mathbf{U}}{\partial \mathbf{W}}$. 
		  Since $\mathbf{U}$ cannot be explicitly formulated by $\mathbf{W}$, we derive $\frac{\partial \mathbf{U}}{\partial \mathbf{W}}$ by the chain rule
		\begin{equation} \label{U_W_chainRule}
			\frac{\partial \mathbf{U}}{\partial \mathbf{W}}=\frac{\partial \mathbf{U}}{\partial \mathbf{V}}\frac{\partial \mathbf{V}}{\partial \mathbf{W}},
		\end{equation}
		where $\frac{\partial \mathbf{V}}{\partial \mathbf{W}}$ is the inverse of $\frac{\partial\mathbf{W}}{\partial\mathbf{V}}$.
		As $\mathbf{W}$ is explicitly expressed by $\mathbf{V}$, a direct calculation gives 
		\begin{equation} \label{W_V:1D}
			\frac{\partial\mathbf{W}}{\partial\mathbf{V}} = \begin{pmatrix}
				\frac{h}{\rho\theta} & 0 & \frac{\theta-h}{\rho\theta^2}\\
				\frac{\gamma v_1}{\rho\theta} & \frac{\gamma^3}{\theta} & -\frac{\gamma v_1}{\rho\theta^2}\\
				-\frac{\gamma}{\rho\theta} & -\frac{\gamma^3 v_1}{\theta} & \frac{\gamma}{\rho\theta^2}
			\end{pmatrix},
		\end{equation}
		and we then obtain the inverse $\frac{\partial \mathbf{V}}{\partial \mathbf{W}}$ as
		\begin{equation} \label{V_W:1D}
			\frac{\partial\mathbf{V}}{\partial\mathbf{W}} = \begin{pmatrix}
				\rho & \rho\gamma v_1(h-\theta) & \rho\gamma(h-\theta)\\
				0 & \frac{\theta}{\gamma} & \frac{\theta v_1}{\gamma}\\
				\rho\theta & \rho h\gamma\theta v_1 & \rho h \gamma\theta
			\end{pmatrix}.
		\end{equation}
		Combining \eqref{V_W:1D} with \eqref{U_V} for the case $d=1$, we can compute $\frac{\partial \mathbf{U}}{\partial\mathbf{W}}$ by \eqref{U_W_chainRule} as
		\begin{equation} \label{U_W:1D}
			\frac{\partial\mathbf{U}}{\partial\mathbf{W}} = \begin{pmatrix}
				\rho\gamma & \rho h\gamma^2 v_1 & \rho(-\theta+h\gamma^2) \\
				\rho h \gamma^2 v_1 & \rho\gamma^3\left(\theta^2v_1^2(1+e'(\theta))+\theta h+h^2v_1^2\right) & \rho\gamma^3 v_1\left(\theta^2(1+e'(\theta))+h^2+\theta hv_1^2\right)\\
				\rho(-\theta+h\gamma^2) & \rho\gamma^3 v_1\left(\theta^2(1+e'(\theta))+h^2+\theta hv_1^2\right) & \rho\gamma^3\left(\theta^2(1+e'(\theta))+h^2-2\theta h+3\theta hv_1^2\right)
			\end{pmatrix}.
		\end{equation}
		Then we substitute the scaling coefficients $d_j,\ j = 1,2,3$, in \eqref{1DdissipationMatrixCoeff} into the scaling eigenvector matrix \eqref{1DDissMat} to compute $\mathbf{R}\mathbf{R}^\top$. 
		Let ${\bf r}_i$ be the $i$th row of the scaled eigenvector matrix. We have 
		\begin{align*}
		    {\bf r}_1 {\bf r}_1^\top &= d_1+d_2+d_3=\frac{\rho e'(\theta)\gamma}{1+e'(\theta)}+\frac{\rho\gamma}{1+e'(\theta)}=\rho\gamma=\left(\frac{\partial\mathbf{U}}{\partial\mathbf{W}}\right)_{1,1},\\
		    {\bf r}_1 {\bf r}_2^\top &= (v_1-c_s)h\gamma d_1+\left(h-\theta(1+e'(\theta)\right)\gamma v_1d_2+(v_1+c_s)h\gamma d_3\\
		    &=v_1h\gamma^2\frac{\rho e'(\theta)}{1+e'(\theta)}+c_sh\gamma^2\frac{\rho\theta v_1}{c_sh}+\frac{\rho h\gamma^2v_1}{1+e'(\theta)}-\rho\gamma^2v_1\theta
		    =\rho h\gamma^2v_1=\left(\frac{\partial\mathbf{U}}{\partial\mathbf{W}}\right)_{1,2},\\
		    {\bf r}_1 {\bf r}_3^\top &= (1-v_1c_s)h\gamma d_1+\left(h-\theta(1+e'(\theta)\right)\gamma d_2+(1+v_1c_s)h\gamma d_3\\
		    &=h\gamma^2\frac{\rho e'(\theta)}{1+e'(\theta)}+v_1c_sh\gamma\frac{\rho\theta v_1\gamma}{c_sh}+\frac{\rho h\gamma^2}{1+e'(\theta)}-\rho\gamma^2\theta
		    =\rho h\gamma^2-\rho\theta=\left(\frac{\partial\mathbf{U}}{\partial\mathbf{W}}\right)_{1,3},\\
		    {\bf r}_2 {\bf r}_2^\top &= (v_1-c_s)^2h^2\gamma^2d_1+\left(h-\theta(1+e'(\theta)\right)^2\gamma^2v_1^2d_2+(v_1+c_s)^2h^2\gamma^2d_3\\
		    &=(v_1^2+c_s^2)h^2\gamma^3\frac{\rho e'(\theta)}{1+e'(\theta)}+2v_1c_sh^2\gamma^3\frac{\rho\theta v_1}{c_sh}+h^2\gamma^2v_1^2\frac{\rho\gamma}{1+e'(\theta)}\\
		    &\quad+\rho\theta^2(1+e'(\theta))\gamma^3v_1^2-2\theta h(1+e'(\theta))\gamma^2v_1^2\frac{\rho\gamma}{1+e'(\theta)}\\
		    &= \rho h^2\gamma^3v_1^2+\rho\theta h\gamma^3+\rho\theta^2(1+e'(\theta))\gamma^3v_1^2\\
		    &=\rho\gamma^3\left(h^2v_1^2+\theta h+\theta^2 v_1^2(1+e'(\theta))\right)=\left(\frac{\partial\mathbf{U}}{\partial\mathbf{W}}\right)_{2,2},\\
		    {\bf r}_2 {\bf r}_3^\top &= (v_1-c_s)(1-v_1c_s)h^2\gamma^2d_1+\left(h-\theta(1+e'(\theta)\right)^2\gamma^2v_1d_2+(v_1+c_s)(1+v_1c_s)h^2\gamma^2d_3\\
		    &=h^2\gamma^3v_1\frac{\rho e'(\theta)}{1+e'(\theta)}\left(1+c_s^2\right)+c_s(1+v_1^2)h^2\gamma^3\frac{\rho\theta v_1}{c_sh}+\frac{\rho\gamma^3v_1}{1+e'(\theta)}\left(h-\theta(1+e'(\theta))\right)^2\\
		    &=\rho\gamma^3v_1\left(\frac{h^2e'(\theta)}{1+e'(\theta)}+h\theta(2+v_1^2)+\frac{\left(h-\theta(1+e'(\theta))\right)^2}{1+e'(\theta)}\right)\\
		    &=\rho\gamma^3v_1\left(h^2+\theta^2(1+e'(\theta))+h\theta v_1^2\right)=\left(\frac{\partial\mathbf{U}}{\partial\mathbf{W}}\right)_{2,3},\\
		    {\bf r}_3 {\bf r}_3^\top &= (1-v_1c_s)^2h^2\gamma^2d_1+\left(h-\theta(1+e'(\theta))\right)^2\gamma^2d_2+(1+v_1c_s)^2h^2\gamma^2d_3\\
		    &=(1+v_1^2c_s^2)h^2\gamma^2\frac{\rho e'(\theta)\gamma}{1+e'(\theta)}+2v_1c_sh^2\gamma^2\frac{\rho\theta v_1\gamma}{c_sh}+\left(h-\theta(1+e'(\theta))\right)^2\frac{\rho\gamma^3}{1+e'(\theta)}\\
		    &=\rho\gamma^3\left(\frac{h^2e'(\theta)}{1+e'(\theta)}\left(1+\frac{\theta(1+e'(\theta)v_1^2}{he'(\theta)}\right)+2\theta v_1^2h+\frac{\left(h-\theta(1+e'(\theta))\right)^2}{1+e'(\theta)}\right)\\
		    &=\rho\gamma^3\left(h^2+3\theta hv_1^2-2\theta h+\theta^2(1+e'(\theta))\right)=\left(\frac{\partial\mathbf{U}}{\partial\mathbf{W}}\right)_{3,3}.
		\end{align*}
		Noting that  both matrices ${\bf R} {\bf R}^\top$ and $\frac{\partial\mathbf{U}}{\partial\mathbf{W}}$ are symmetric, we have
		\begin{equation*}
		    {\bf r}_2 {\bf r}_1^\top = \left(\frac{\partial\mathbf{U}}{\partial\mathbf{W}}\right)_{2,1},\quad
		    {\bf r}_3 {\bf r}_1^\top = \left(\frac{\partial\mathbf{U}}{\partial\mathbf{W}}\right)_{3,1},\quad
		   {\bf r}_3 {\bf r}_2^\top = \left(\frac{\partial\mathbf{U}}{\partial\mathbf{W}}\right)_{3,2}.
		\end{equation*}
		Hence, we have verified that $\frac{\partial \mathbf{U}}{\partial \mathbf{W}}=\mathbf{R}\mathbf{R}^\top$. The proof is completed.

	\end{proof}

	The dissipation matrix $\mathbf{D}_{i+\frac{1}{2}}$ is defined at the cell interface $x_{i+\frac12}$. In order to calculate it, we need to estimate the ``averaged'' states at 
	$x_{i+\frac12}$. Following the discussion in \cite{chandrashekar2013kinetic} for the non-relativistic Euler equations, we seek an appropriate average for $\mathbf{D}_{i+\frac{1}{2}}$, such that the resulting ES scheme can accurately resolve stationary contact discontinuities. Consider the following initial condition
	\begin{equation} \label{SCD:init}
		\rho(x,0) = \left\{
		\begin{aligned}
			&\rho_L,\quad x<0, \\
			&\rho_R,\quad x>0,
		\end{aligned}
		\right.
		\quad\quad
		v_1(x,0) = 0,\quad\quad
		p(x,0) = p = const,
	\end{equation}
	which represents a stationary contact discontinuity corresponding to the $\lambda_2-$field. 
	For the above initial condition, our EC numerical flux reduces to 
	$\widetilde{\mathbf{F}}_{i+\frac{1}{2}}=
		(0, p, 0 )^\top.$ 
	Hence, in order to preserve the stationary contact discontinuity, we require for all $i$ that 
	$\widehat{\mathbf{F}}_{i+\frac{1}{2}}=
	(0, p, 0 )^\top,$ or equivalently, 
	\begin{equation} \label{SCD:relation}
		\mathbf{D}_{i+\frac{1}{2}}\left[\!\left[\mathbf{W}\right]\!\right]_{i+\frac{1}{2}}={\mathbf 0}. 
	\end{equation}

	\begin{theorem} \label{thm:enthalpyEvaluation}
		Assume that the ``averaged" specific enthalpy $h_{i+\frac{1}{2}}$ in the dissipation matrix $\mathbf{D}_{i+\frac{1}{2}}$ satisfies
		\begin{equation} \label{3.13}
			h_{i+\frac{1}{2}}  
			={\mathcal E}_{i+\frac{1}{2}}+\frac{1}{\left\{\!\left\{ z_2 \right\}\!\right\}_{\ln,{i+\frac{1}{2}}}},
		\end{equation}
		where $z_2=\frac{\rho}{p}=\frac{1}{\theta}$ and the calculation of ${\mathcal E}$ is given in Theorem \ref{thm:1De}. Then, the ES scheme \eqref{2.1} with numerical flux \eqref{3.1} and Roe-type dissipation term can exactly resolve stationary contact discontinuities.
	\end{theorem}

	\begin{proof}
		For the stationary contact wave \eqref{SCD:init} with $v_1=0$, the dissipation matrix $\mathbf{D}$ reduces to 
		\begin{equation} \label{3.5}
			\begin{aligned}
				\mathbf{D} &=\begin{pmatrix}
					1 & 1 & 1 \\
					-c_sh & 0 & c_sh \\
					h & h-\theta\left(1+e'(\theta)\right) & h 
				\end{pmatrix}
				\begin{pmatrix}
					c_s & 0 & 0 \\
					0 & 0 & 0 \\
					0 & 0 & c_s  
				\end{pmatrix}
				\begin{pmatrix}
					d_1 & 0 & 0 \\
					0 & d_2 & 0 \\
					0 & 0 & d_3 
				\end{pmatrix}
				\begin{pmatrix}
					1 & 1 & 1 \\
					-c_sh & 0 & c_sh \\
					h & h-\theta\left(1+e'(\theta)\right) & h 
				\end{pmatrix}^\top
			\end{aligned}
		\end{equation}
		with 
		\begin{equation} \label{3.12}
			d_1=d_3=\frac{\rho e'(\theta)}{2(1+e'(\theta))},\quad d_2 = \frac{\rho}{1+e'(\theta)}.
		\end{equation}
		Taking $v_1=0$ and $z_2 = \frac{1}{\theta}$, 
		the entropy variables in \eqref{generalEntropyVar} become  
		\begin{equation} \label{3.6}
			\mathbf{W} = \begin{pmatrix}
				z_2h-S & 0 & -z_2
			\end{pmatrix}^\top.
		\end{equation}
		It follows that 
		\begin{equation} \label{3.10}
			\mathbf{D}_{i+\frac{1}{2}}\left[\!\left[\mathbf{W}\right]\!\right]_{i+\frac{1}{2}} = (c_s)_{i+\frac{1}{2}}\ \frac{\rho_{i+\frac{1}{2}} e'(\theta_{i+\frac{1}{2}})}{1+e'(\theta_{i+\frac{1}{2}})}\left(\left[\!\left[z_2 h-S\right]\!\right]_{i+\frac{1}{2}}-h_{i+\frac{1}{2}}\left[\!\left[z_2\right]\!\right]_{i+\frac{1}{2}}\right)\begin{pmatrix}
				1\\
				0\\
				h_{i+\frac{1}{2}}
			\end{pmatrix}.
		\end{equation}
		For the stationary contact wave \eqref{SCD:init} with constant pressure, we have
		\begin{equation} \label{SCD:pconst}
			[\![\ln p]\!]_{i+\frac12} = 0,\qquad [\![\ln z_2]\!]_{i+\frac12} = [\![\ln \rho]\!]_{i+\frac12}-[\![\ln p]\!]_{i+\frac12}=[\![\ln \rho]\!]_{i+\frac12}.
		\end{equation}
		If $h_{i+\frac12}$ satisfies \eqref{3.13}, we obtain 
			\begin{align*}
			h_{i+\frac{1}{2}} &= 
			{\mathcal E}_{i+\frac{1}{2}}+\frac{1}{\left\{\!\left\{ z_2 \right\}\!\right\}_{\ln,{i+\frac{1}{2}}}}
			\\ &
		 \overset{\eqref{logAVG}}{=}  {\mathcal E}_{i+\frac{1}{2}} 
		 + \frac{\left[\!\left[\ln (z_2) \right]\!\right]_{i+\frac{1}{2}}}{\left[\!\left[z_2\right]\!\right]_{i+\frac{1}{2}}}
		 \\ &
		  \overset{\eqref{SCD:pconst}}{=} {\mathcal E}_{i+\frac{1}{2}}+\frac{\left[\!\left[\ln(\rho)\right]\!\right]_{i+\frac{1}{2}}}{\left[\!\left[z_2\right]\!\right]_{i+\frac{1}{2}}}
		  \\
		  &
		 \overset{\eqref{1.8}}{=} 
			 \frac{\left[\!\left[W_1-\ln(\rho)\right]\!\right]_{i+\frac{1}{2}}}{\left[\!\left[z_2\right]\!\right]_{i+\frac{1}{2}}}+\frac{\left[\!\left[\ln(\rho)\right]\!\right]_{i+\frac{1}{2}}}{\left[\!\left[z_2\right]\!\right]_{i+\frac{1}{2}}} 
			 \\ &
			 = \frac{\left[\!\left[W_1\right]\!\right]_{i+\frac{1}{2}}}{\left[\!\left[z_2\right]\!\right]_{i+\frac{1}{2}}}
			 =  \frac{\left[\!\left[z_2 h-S\right]\!\right]_{i+\frac{1}{2}}}{\left[\!\left[z_2\right]\!\right]_{i+\frac{1}{2}}},
	\end{align*}
	which implies 
	$$
	\left[\!\left[z_2 h-S\right]\!\right]_{i+\frac{1}{2}}=h_{i+\frac{1}{2}}\left[\!\left[z_2\right]\!\right]_{i+\frac{1}{2}}.
	$$
	This together with \eqref{3.10} yields $\mathbf{D}_{i+\frac{1}{2}}\left[\!\left[\mathbf{W}\right]\!\right]_{i+\frac{1}{2}} = {\mathbf 0}$. 
	The proof is completed.
	\end{proof}

The formula \eqref{3.13} determines the averaged state $h_{i+\frac{1}{2}}$, with 
$(z_2)_{i+\frac{1}{2}} = \left\{\!\left\{ z_2 \right\}\!\right\}_{\ln,{i+\frac{1}{2}}}$. 
We take $\theta_{i+\frac{1}{2}}=\frac{1}{\left\{\!\left\{ z_2 \right\}\!\right\}_{\ln,{i+\frac{1}{2}}}}$. 
The rest-mass density $\rho$ and the velocity $v_1$ can be evaluated by either the arithmetic or logarithmic average. In this paper, we choose the logarithmic average for $\rho_{i+\frac{1}{2}}$ and the arithmetic average for $(v_1)_{i+\frac{1}{2}}$. Other
quantities in $\mathbf{D}_{i+\frac{1}{2}}$, such as $c_s$, $\gamma$, $e'(\theta)$, are computed by using the averaged  states $\theta_{i+\frac{1}{2}}, \rho_{i+\frac{1}{2}}$, and $(v_1)_{i+\frac{1}{2}}$.

	\begin{remark}
		For ID-EOS \eqref{ID-EOS}, the evaluation for $h_{i+\frac{1}{2}}$ in the dissipation matrix $\mathbf{D}_{i+\frac{1}{2}}$ in \cite{duan2019high} is calculated by taking $(z_2)_{i+\frac{1}{2}}$ as the  logarithmic average. This is consistent with our result in Theorem \ref{thm:enthalpyEvaluation}, because 
		\begin{align*} 
			h_{i+\frac{1}{2}} &= {\mathcal E}_{i+\frac{1}{2}}+\frac{1}{\left\{\!\left\{ z_2 \right\}\!\right\}_{\ln,{i+\frac{1}{2}}}} 
			\\
			&\overset{\eqref{e_hat:ideal}}{=} 1+\frac{1}{(\Gamma-1)\left\{\!\left\{ z_2 \right\}\!\right\}_{\ln,{i+\frac{1}{2}}}}+\frac{1}{\left\{\!\left\{ z_2 \right\}\!\right\}_{\ln,{i+\frac{1}{2}}}}
			\\
			&=1+\frac{\Gamma}{(\Gamma-1)\left\{\!\left\{ z_2 \right\}\!\right\}_{\ln,{i+\frac{1}{2}}}}.
		\end{align*}
		
	\end{remark}

\subsubsection{High-order entropy stable schemes}
As mentioned previously, the numerical scheme \eqref{2.1} using the numerical flux \eqref{3.1} is only first-order accurate.  This is due to the calculation of the jump $\left[\!\left[\mathbf{W}\right]\!\right]_{i+\frac{1}{2}}$ at the cell interface $x_{i+\frac{1}{2}}$ using only $\mathbf{W}_{i}$ and $\mathbf{W}_{i+1}$. To achieve higher-order accuracy for the ES schemes, it is necessary to estimate the jump more precisely \cite{fjordholm2012arbitrarily}. This can be accomplished by employing ENO or WENO reconstruction techniques for the scaled entropy variables ${\bm \omega}:=\mathbf{R}_{i+\frac{1}{2}}^\top\mathbf{W}$. The reconstructed values of the scaled entropy variables ${\bm \omega}$, denoted by ${\bm \omega}_{i+\frac{1}{2}}^-$ and ${\bm \omega}_{i+\frac{1}{2}}^+$ for the left and right limiting values at the interface $x_{i+\frac{1}{2}}$, respectively. For the ENO-based method \cite{fjordholm2012arbitrarily}, the corresponding $2k$th-order ES flux is defined by adding the $2k$th-order dissipation terms to the $2k$th-order EC flux:
\begin{equation} \label{ENObased}
	\widehat{\mathbf{F}}_{i+\frac{1}{2}}^{2k}=\widetilde{\mathbf{F}}_{i+\frac{1}{2}}^{2k}-\frac{1}{2}\mathbf{R}_{i+\frac{1}{2}}\left|\mathbf{\Lambda}_{i+\frac{1}{2}}\right|\left \llangle{\bm \omega}\right\rrangle_{i+\frac{1}{2}}^{ENO},
\end{equation}
where $\widetilde{\mathbf{F}}_{i+\frac{1}{2}}^{2k}$ is the $2k$th-order EC flux defined in \eqref{2.5}, $\mathbf{R}_{i+\frac{1}{2}}$ and $\mathbf{\Lambda}_{i+\frac{1}{2}}$ are defined in \eqref{3.3}, and $\left \llangle{\bm \omega}\right\rrangle_{i+\frac{1}{2}}^{ENO}:={\bm \omega}_{i+\frac{1}{2}}^+-{\bm \omega}_{i+\frac{1}{2}}^-$ denotes the jump of the scaled entropy variables at the interface $x_{i+\frac{1}{2}}$.
For the scheme \eqref{2.1} using the $2k$th-order ES numerical flux defined in \eqref{ENObased}, we have
\begin{equation} \label{ENOESComTemp}
	\begin{aligned}
		\frac{d}{dt}\eta(\mathbf{U}_i)
		&=-\frac{\mathbf{W}_i^\top\Big(\widetilde{\mathbf{F}}_{i+\frac{1}{2}}^{2k}-\widetilde{\mathbf{F}}_{i-\frac{1}{2}}^{2k}\Big)}{\Delta x} +\frac{\mathbf{W}_i^\top\Big(\mathbf{R}_{i+\frac{1}{2}}\left|\mathbf{\Lambda}_{i+\frac{1}{2}}\right|\left\llangle{\bm \omega}\right\rrangle_{i+\frac{1}{2}}^{ENO}-\mathbf{R}_{i-\frac{1}{2}}\left|\mathbf{\Lambda}_{i-\frac{1}{2}}\right|\left\llangle{\bm \omega}\right\rrangle_{i-\frac{1}{2}}^{ENO}\Big)}{2\Delta x}\\
		&=-\frac{\widehat{q}^{2k}_{i+\frac{1}{2}}-\widehat{q}^{2k}_{i-\frac{1}{2}}+\frac{1}{4} \left([\![{\bm \omega}]\!]_{i+\frac{1}{2}}^\top\left|\mathbf{\Lambda}_{i+\frac{1}{2}}\right|\left\llangle{\bm \omega}\right\rrangle_{i+\frac{1}{2}}^{ENO}+[\![{\bm \omega}]\!]_{i-\frac{1}{2}}^\top\left|\mathbf{\Lambda}_{i-\frac{1}{2}}\right|\left\llangle{\bm \omega}\right\rrangle_{i-\frac{1}{2}}^{ENO}\right)}{\Delta x},
	\end{aligned}
\end{equation}
where 
\begin{equation} \label{high-orderESThm}
	\widehat{q}_{i+\frac{1}{2}}^{2k}=\widetilde{q}_{i+\frac{1}{2}}^{2k}-\frac{1}{2}\{\!\{\mathbf{W}\}\!\}_{i+\frac{1}{2}}^\top\mathbf{R}_{i+\frac{1}{2}}\left|\mathbf{\Lambda}_{i+\frac{1}{2}}\right|\left\llangle{\bm \omega}\right\rrangle_{i+\frac{1}{2}}^{ENO}
\end{equation}
with $\widetilde{q}_{i+\frac{1}{2}}^{2k}$ defined in \eqref{2.8}.
Since the ENO reconstruction satisfies the sign property \cite{fjordholm2013eno}:
\begin{equation} \label{signproperty}
	{\rm sign}(\left\llangle{\bm \omega}\right\rrangle_{i+\frac{1}{2}}^{ENO})={\rm sign}([\![{\bm \omega}]\!]_{i+\frac{1}{2}}),
\end{equation}
then we have
\begin{equation*}
	\frac{d}{dt}\eta(\mathbf{U}_i)+\frac{1}{\Delta x}\left(\widehat{q}_{i+\frac{1}{2}}^{2k}-\widehat{q}_{i-\frac{1}{2}}^{2k}\right)=-\frac{1}{4\Delta x}  \Big([\![{\bm \omega}]\!]_{i+\frac{1}{2}}^\top\left|\mathbf{\Lambda}_{i+\frac{1}{2}}\right|\left\llangle{\bm \omega}\right\rrangle_{i+\frac{1}{2}}^{ENO}+[\![{\bm \omega}]\!]_{i-\frac{1}{2}}^\top\left|\mathbf{\Lambda}_{i-\frac{1}{2}}\right|\left\llangle{\bm \omega}\right\rrangle_{i-\frac{1}{2}}^{ENO}\Big)\leq 0,
\end{equation*}
which implies the discrete entropy inequality \eqref{3.2} for the numerical entropy flux \eqref{high-orderESThm}. Hence, the scheme \eqref{2.1} with the ENO-based numerical flux \eqref{ENObased} is ES.

For the WENO-based method, we follow the idea in \cite{biswas2018low}. 
The WENO-based high-order accurate ES flux is defined as
\begin{equation} \label{WENObased}
	\widehat{\mathbf{F}}_{i+\frac{1}{2}}^{2k}=\widetilde{\mathbf{F}}_{i+\frac{1}{2}}^{2k}-\frac{1}{2}\mathbf{R}_{i+\frac{1}{2}}\left|\mathbf{\Lambda}_{i+\frac{1}{2}}\right|\left \llangle{\bm \omega}\right\rrangle_{i+\frac{1}{2}}^{WENO},
\end{equation}
where the $lth$ component of the jump $\left \llangle{\bm \omega}\right\rrangle_{i+\frac{1}{2}}^{WENO}$ of the scaled entropy variable at the interface $x_{i+\frac{1}{2}}$ is defined by
\begin{equation*}
	\left\llangle{\omega}_l\right\rrangle_{i+\frac{1}{2}}^{WENO}=\theta_{l,i+\frac{1}{2}}\left({\omega}_{l,i+\frac{1}{2}}^+-{\omega}_{l,i+\frac{1}{2}}^-\right)
\end{equation*}
with
\begin{equation} \label{switchop}
	\theta_{l,i+\frac{1}{2}}:=\left\{
	\begin{aligned}
		&1,\quad \mbox{if}~({\omega}_{l,i+\frac{1}{2}}^+-{ \omega}_{l,i+\frac{1}{2}}^-)[\![\omega_l]\!]_{i+\frac{1}{2}}>0, \\
		&0,\quad {\rm otherwise}.
	\end{aligned}
	\right.
\end{equation}
The switch operator $\theta_{l,i+\frac{1}{2}}$ in \eqref{switchop} is introduced to ensure the sign property
\begin{equation} \label{WENOsignproperty}
	{\rm sign}(\left\llangle{\bm \omega}\right\rrangle_{i+\frac{1}{2}}^{WENO})={\rm sign}([\![{\bm \omega}]\!]_{i+\frac{1}{2}}).
\end{equation}
Hence, by using the same approach as ENO-based method, we can verify that the scheme \eqref{2.1} with the WENO-based numerical flux \eqref{WENObased} is ES, and the corresponding numerical entropy flux is given by 
\begin{equation} \label{WENOhigh-orderESThm}
	\widehat{q}_{i+\frac{1}{2}}^{2k}=\widetilde{q}_{i+\frac{1}{2}}^{2k}-\frac{1}{2}\{\!\{\mathbf{W}\}\!\}_{i+\frac{1}{2}}^\top\mathbf{R}_{i+\frac{1}{2}}\left|\mathbf{\Lambda}_{i+\frac{1}{2}}\right|\left\llangle{\bm \omega}\right\rrangle_{i+\frac{1}{2}}^{WENO}. 
\end{equation}

\section{2D entropy stable schemes}\label{section:4}

The EC and ES schemes for the 2D RHD equations can be constructed in a dimension-by-dimension fashion, and the construction is analogous to the 1D case. Hence we only present the derivation of two-point EC fluxes and the dissipation matrix, which are the key ingredients of EC and ES schemes.

\subsection{Two-point entropy conservative flux}

In this subsection, we derive a unified formula of two-point EC numerical fluxes for the 2D RHD equations with  general Synge-type EOS \eqref{eq:gEOS}, based on the entropy variables \eqref{lem:generalEOS:entropyVarPot} and the entropy potential \eqref{generalEntropyPot}. To design a simple two-point EC numerical flux, we choose a set of variables $\textbf{z}=(z_1,z_2,z_3,z_4)^\top$ as
\begin{equation} \label{2DalgeVar}
	z_1 = \rho,\quad z_2 = \frac{\rho}{p},\quad z_3 = \gamma v_1,\quad z_4 = \gamma v_2.
\end{equation}

\begin{theorem} \label{thm:2DECflux}
	The two-point EC numerical fluxes  for the 2D RHD equations with general Synge-type EOS \eqref{eq:gEOS} can be written into a unified form as
	\begin{equation} \label{2D2ptECFlux_x}
		\widetilde{\textbf{F}_1}^{EC}(\mathbf{U}_L,\mathbf{U}_R)=\bigg(\{\!\{z_1\}\!\}_{\ln}\{\!\{z_3\}\!\},\ \widehat{\rho h}\{\!\{z_3\}\!\}^2+\frac{\{\!\{z_1\}\!\}}{\{\!\{z_2\}\!\}},\ \widehat{\rho h}\{\!\{z_3\}\!\}\{\!\{z_4\}\!\},\ \widehat{\rho h}\{\!\{\gamma\}\!\}\{\!\{z_3\}\!\}\bigg)^\top,
	\end{equation}
	and
	\begin{equation} \label{2D2ptECFlux_y}
		\widetilde{\textbf{F}_2}^{EC}(\mathbf{U}_L,\mathbf{U}_R)=\bigg(\{\!\{z_1\}\!\}_{\ln}\{\!\{z_4\}\!\},\ \widehat{\rho h}\{\!\{z_3\}\!\}\{\!\{z_4\}\!\},\ \widehat{\rho h}\{\!\{z_4\}\!\}^2+\frac{\{\!\{z_1\}\!\}}{\{\!\{z_2\}\!\}},\  \widehat{\rho h}\{\!\{\gamma\}\!\}\{\!\{z_4\}\!\}\bigg)^\top
	\end{equation}
	with 
	\begin{equation*}
	    \widehat{\rho h}:=\frac{\frac{\{\!\{z_1\}\!\}}{\{\!\{z_2\}\!\}}+\{\!\{z_1\}\!\}_{\ln}{\mathcal E}}{\{\!\{\gamma\}\!\}^2-\{\!\{z_3\}\!\}^2-\{\!\{z_4\}\!\}^2},
	\end{equation*}
	where $\mathcal E$ is defined in \eqref{1.8}, and a unified formulation of $\mathcal E$ is given by \eqref{1.8:int} via an integral. 
		The explicit calculation of ${\mathcal E}$ depends on the particular choice of the EOS (see Theorem \ref{thm:1De}). 
\end{theorem}

\begin{proof}
	By using the set of variables \eqref{2DalgeVar}, we can express the entropy variables and the entropy potential as
	\begin{equation*}
		\mathbf{W}=\left(z_2h-S, z_3z_2, z_4z_2, -z_2\sqrt{1+z_3^2+z_4^2}\right)^\top,
	\end{equation*}
	\begin{equation*}
		\psi_1=z_1z_3,\qquad \psi_2=z_1z_4.
	\end{equation*}
	Then we write the jumps of entropy variables \textbf{W} and the potential fluxes $\psi_1,\ \psi_2$ in terms of jumps and arithmetic averages of the variables \eqref{2DalgeVar} as follows:
		\begin{equation} \label{thm6:temp}
			\begin{aligned}
				\relax[\![W_1]\!] &\overset{\eqref{1.8}}{=}{\mathcal E}[\![z_2]\!]+[\![\ln{z_1}]\!],\\
				[\![W_2]\!] &= [\![z_3z_2]\!] \overset{\eqref{product}}{=} \{\!\{z_3\}\!\}[\![z_2]\!]+[\![z_3]\!]\{\!\{z_2\}\!\},\\
				[\![W_3]\!] &= [\![z_4z_2]\!] \overset{\eqref{product}}{=} \{\!\{z_4\}\!\}[\![z_2]\!]+[\![z_4]\!]\{\!\{z_2\}\!\},\\
				[\![W_4]\!] &= \left[\!\left[-z_2\sqrt{1+z_3^2+z_4^2}\right]\!\right]
				\\
				& \overset{\eqref{product}}{=}-[\![z_2]\!]\{\!\{\gamma\}\!\}-\{\!\{z_2\}\!\}\left[\!\left[\sqrt{1+z_3^2+z_4^2}\right]\!\right]\\
				&\overset{\eqref{sqrt}}{=}-[\![z_2]\!]\{\!\{\gamma\}\!\}-\{\!\{z_2\}\!\}\frac{[\![1+z_3^2+z_4^2]\!]}{2\{\!\{\gamma\}\!\}}\\
				&\overset{\eqref{square}}{=} -[\![z_2]\!]\{\!\{\gamma\}\!\}-\frac{\{\!\{z_2\}\!\}\left(\{\!\{z_3\}\!\}[\![z_3]\!]+\{\!\{z_4\}\!\}[\![z_4]\!]\right)}{\{\!\{\gamma\}\!\}},\\
				[\![\psi_1]\!] &= [\![z_1z_3]\!] \overset{\eqref{product}}{=} \{\!\{z_1\}\!\}[\![z_3]\!]+[\![z_1]\!]\{\!\{z_3\}\!\},\\
				[\![\psi_2]\!] &= [\![z_1z_4]\!] \overset{\eqref{product}}{=} \{\!\{z_1\}\!\}[\![z_4]\!]+[\![z_1]\!]\{\!\{z_4\}\!\},
			\end{aligned}
	\end{equation}
According to Definition \ref{def:EC}, the two-point EC numerical fluxes $\widetilde{\textbf{F}}^{EC}_1 ( {\mathbf U}_L, {\mathbf U}_R )=:\left(\widetilde{F}^{(1)}_{1},\widetilde{F}^{(2)}_{1},\widetilde{F}^{(3)}_{1},\widetilde{F}^{(4)}_{1}\right)^\top$ and $\widetilde{\textbf{F}}^{EC}_2 ( {\mathbf U}_L, {\mathbf U}_R )=:\left(\widetilde{F}^{(1)}_{2},\widetilde{F}^{(2)}_{2},\widetilde{F}^{(3)}_{2},\widetilde{F}^{(4)}_{2}\right)^\top$ for the 2D RHD system satisfy 
\begin{equation} \label{EClinearsystem2D}
\begin{aligned}
    &[\![W_1]\!]\widetilde{F}^{(1)}_{1}+[\![W_2]\!]\widetilde{F}^{(2)}_{1}+[\![W_3]\!]\widetilde{F}^{(3)}_{1}+[\![W_4]\!]\widetilde{F}^{(4)}_{1}=[\![\psi_1]\!],\\
    &[\![W_1]\!]\widetilde{F}^{(1)}_{2}+[\![W_2]\!]\widetilde{F}^{(2)}_{2}+[\![W_3]\!]\widetilde{F}^{(3)}_{2}+[\![W_4]\!]\widetilde{F}^{(4)}_{2}=[\![\psi_2]\!].
\end{aligned}
\end{equation}
Substituting \eqref{thm6:temp} into \eqref{EClinearsystem2D}, we obtain 
\begin{align*}
\left({\mathcal E}[\![z_2]\!]+[\![\ln{z_1}]\!]\right)\widetilde{F}^{(1)}_{1}&+\left(\{\!\{z_3\}\!\}[\![z_2]\!]+[\![z_3]\!]\{\!\{z_2\}\!\}\right)\widetilde{F}^{(2)}_{1}+\left(\{\!\{z_4\}\!\}[\![z_2]\!]+[\![z_4]\!]\{\!\{z_2\}\!\}\right)\widetilde{F}^{(3)}_{1}\\
&-\left([\![z_2]\!]\{\!\{\gamma\}\!\}+\frac{\{\!\{z_2\}\!\}\left(\{\!\{z_3\}\!\}[\![z_3]\!]+\{\!\{z_4\}\!\}[\![z_4]\!]\right)}{\{\!\{\gamma\}\!\}}\right)\widetilde{F}^{(4)}_{1}=\{\!\{z_1\}\!\}[\![z_3]\!]+[\![z_1]\!]\{\!\{z_3\}\!\},
\end{align*}
and
\begin{align*}
\left({\mathcal E}[\![z_2]\!]+[\![\ln{z_1}]\!]\right)\widetilde{F}^{(1)}_{2}&+\left(\{\!\{z_3\}\!\}[\![z_2]\!]+[\![z_3]\!]\{\!\{z_2\}\!\}\right)\widetilde{F}^{(2)}_{2}+\left(\{\!\{z_4\}\!\}[\![z_2]\!]+[\![z_4]\!]\{\!\{z_2\}\!\}\right)\widetilde{F}^{(3)}_{2}\\
&-\left([\![z_2]\!]\{\!\{\gamma\}\!\}+\frac{\{\!\{z_2\}\!\}\left(\{\!\{z_3\}\!\}[\![z_3]\!]+\{\!\{z_4\}\!\}[\![z_4]\!]\right)}{\{\!\{\gamma\}\!\}}\right)\widetilde{F}^{(4)}_{2}=\{\!\{z_1\}\!\}[\![z_4]\!]+[\![z_1]\!]\{\!\{z_4\}\!\}.
\end{align*}
Collecting the terms containing $[\![z_1]\!],[\![z_2]\!],[\![z_3]\!]$, and $[\![z_4]\!]$, respectively, the above two equations can be reformulated as
\begin{align*}
&\left(\frac{\widetilde{F}^{(1)}_{1}}{\{\!\{z_1\}\!\}_{\ln}}-\{\!\{z_3\}\!\}\right)[\![z_1]\!]+\left(\mathcal{E}\widetilde{F}^{(1)}_{1}+\{\!\{z_3\}\!\}\widetilde{F}^{(2)}_{1}+\{\!\{z_4\}\!\}\widetilde{F}^{(3)}_{1}-\{\!\{\gamma\}\!\}\widetilde{F}^{(4)}_{1}\right)[\![z_2]\!]\\
&+\left(\{\!\{z_2\}\!\}\widetilde{F}^{(2)}_{1}-\frac{\{\!\{z_2\}\!\}\{\!\{z_3\}\!\}}{\{\!\{\gamma\}\!\}}\widetilde{F}^{(4)}_{1}-\{\!\{z_1\}\!\}\right)[\![z_3]\!]+\left(\{\!\{z_2\}\!\}\widetilde{F}^{(3)}_{1}-\frac{\{\!\{z_2\}\!\}\{\!\{z_4\}\!\}}{\{\!\{\gamma\}\!\}}\widetilde{F}^{(4)}_{1}\right)[\![z_4]\!]=0,
\end{align*}
and
\begin{align*}
&\left(\frac{\widetilde{F}^{(1)}_{2}}{\{\!\{z_1\}\!\}_{\ln}}-\{\!\{z_4\}\!\}\right)[\![z_1]\!]+\left(\mathcal{E}\widetilde{F}^{(1)}_{2}+\{\!\{z_3\}\!\}\widetilde{F}^{(2)}_{2}+\{\!\{z_4\}\!\}\widetilde{F}^{(3)}_{2}-\{\!\{\gamma\}\!\}\widetilde{F}^{(4)}_{2}\right)[\![z_2]\!]\\
&+\left(\{\!\{z_2\}\!\}\widetilde{F}^{(2)}_{2}-\frac{\{\!\{z_2\}\!\}\{\!\{z_3\}\!\}}{\{\!\{\gamma\}\!\}}\widetilde{F}^{(4)}_{2}\right)[\![z_3]\!]+\left(\{\!\{z_2\}\!\}\widetilde{F}^{(3)}_{2}-\frac{\{\!\{z_2\}\!\}\{\!\{z_4\}\!\}}{\{\!\{\gamma\}\!\}}\widetilde{F}^{(4)}_{2}-\{\!\{z_1\}\!\}\right)[\![z_4]\!]=0.
\end{align*}
Hence, the coefficients of $[\![z_1]\!],[\![z_2]\!],[\![z_3]\!]$ should all equal zero. Specifically, we have
\begin{equation*}
\left\{
\begin{aligned}
&\ \frac{\widetilde{F}^{(1)}_{1}}{\{\!\{z_1\}\!\}_{\ln}} = \{\!\{z_3\}\!\},\\
&\ \mathcal{E}\widetilde{F}^{(1)}_{1}+\{\!\{z_3\}\!\}\widetilde{F}^{(2)}_{1}+\{\!\{z_4\}\!\}\widetilde{F}^{(3)}_{1}-\{\!\{\gamma\}\!\}\widetilde{F}^{(4)}_{1}=0,\\
&\ \{\!\{z_2\}\!\}\widetilde{F}^{(2)}_{1}-\frac{\{\!\{z_2\}\!\}\{\!\{z_3\}\!\}}{\{\!\{\gamma\}\!\}}\widetilde{F}^{(4)}_{1}=\{\!\{z_1\}\!\},\\
&\ \{\!\{z_2\}\!\}\widetilde{F}^{(3)}_{1}-\frac{\{\!\{z_2\}\!\}\{\!\{z_4\}\!\}}{\{\!\{\gamma\}\!\}}\widetilde{F}^{(4)}_{1}=0,
\end{aligned}
\right.
\end{equation*}
and
\begin{equation*}
\left\{
\begin{aligned}
&\ \frac{\widetilde{F}^{(1)}_{2}}{\{\!\{z_1\}\!\}_{\ln}} = \{\!\{z_4\}\!\},\\
&\ \mathcal{E}\widetilde{F}^{(2)}_{1}+\{\!\{z_3\}\!\}\widetilde{F}^{(2)}_{2}+\{\!\{z_4\}\!\}\widetilde{F}^{(3)}_{2}-\{\!\{\gamma\}\!\}\widetilde{F}^{(4)}_{2}=0,\\
&\ \{\!\{z_2\}\!\}\widetilde{F}^{(2)}_{2}-\frac{\{\!\{z_2\}\!\}\{\!\{z_3\}\!\}}{\{\!\{\gamma\}\!\}}\widetilde{F}^{(4)}_{2}=0,\\
&\ \{\!\{z_2\}\!\}\widetilde{F}^{(3)}_{2}-\frac{\{\!\{z_2\}\!\}\{\!\{z_4\}\!\}}{\{\!\{\gamma\}\!\}}\widetilde{F}^{(4)}_{2}=\{\!\{z_1\}\!\},
\end{aligned}
\right.
\end{equation*}
Solving the above two linear systems for $\left(\widetilde{F}^{(1)}_{1},\widetilde{F}^{(2)}_{1},\widetilde{F}^{(3)}_{1},\widetilde{F}^{(4)}_{1}\right)^\top$ and $\left(\widetilde{F}^{(1)}_{2},\widetilde{F}^{(2)}_{2},\widetilde{F}^{(3)}_{2},\widetilde{F}^{(4)}_{2}\right)^\top$, respectively, we obtain
\begin{equation*}
\left\{
\begin{aligned}
&\ \widetilde{F}^{(1)}_{1}=\{\!\{z_1\}\!\}_{\ln}\{\!\{z_3\}\!\},\\
&\ \widetilde{F}^{(2)}_{1}=\widehat{\rho h}\{\!\{z_3\}\!\}^2+\frac{\{\!\{z_1\}\!\}}{\{\!\{z_2\}\!\}},\\
&\ \widetilde{F}^{(3)}_{1}=\widehat{\rho h}\{\!\{z_3\}\!\}\{\!\{z_4\}\!\},\\
&\ \widetilde{F}^{(4)}_{1}=\widehat{\rho h}\{\!\{\gamma\}\!\}\{\!\{z_3\}\!\},
\end{aligned}
\right.
\end{equation*}
and
\begin{equation*}
\left\{
\begin{aligned}
&\ \widetilde{F}^{(1)}_{2}=\{\!\{z_1\}\!\}_{\ln}\{\!\{z_4\}\!\},\\
&\ \widetilde{F}^{(2)}_{2}=\widehat{\rho h}\{\!\{z_3\}\!\}\{\!\{z_4\}\!\},\\
&\ \widetilde{F}^{(3)}_{2}=\widehat{\rho h}\{\!\{z_4\}\!\}^2+\frac{\{\!\{z_1\}\!\}}{\{\!\{z_2\}\!\}},\\
&\ \widetilde{F}^{(4)}_{2}=\widehat{\rho h}\{\!\{\gamma\}\!\}\{\!\{z_4\}\!\},
\end{aligned}
\right.
\end{equation*}
which lead to \eqref{2D2ptECFlux_x} and \eqref{2D2ptECFlux_y}, respectively. 
The proof is completed.

\end{proof}

\subsection{Dissipation matrix}
In this subsection, we present the explicit formulas of the dissipation matrices for the 2D RHD equations with  general Synge-type EOS \eqref{eq:gEOS}. In the 2D case, we need two dissipation matrices
	\begin{equation} \label{2DdissMat}
		\mathbf{D}_{1} = \mathbf{R}_{1}\left|\mathbf{\Lambda}_{1}\right|\mathbf{R}_{1}^\top,\qquad 
		\mathbf{D}_{2} = \mathbf{R}_{2}\left|\mathbf{\Lambda}_{2}\right|\mathbf{R}_{2}^\top
	\end{equation}
	corresponding to the $x_1$- and $x_2$-directions, respectively, where the matrices $\mathbf{R}_1$ and $\mathbf{R}_2$ are respectively formed by the suitably scaled right eigenvectors of the Jacobian matrices $\frac{\partial\mathbf{F}_1(\mathbf{U})}{\partial \mathbf{U}}$ and $\frac{\partial\mathbf{F}_2(\mathbf{U})}{\partial \mathbf{U}}$ of the 2D RHD system, and they satisfy
\begin{equation} \label{2DscaledMatExist}
    \frac{\partial\mathbf{F}_k}{\partial\mathbf{U}}=\mathbf{R}_k\mathbf{\Lambda}_k\mathbf{R}_k^{-1},\quad\quad \frac{\partial \mathbf{U}}{\partial \mathbf{W}}=\mathbf{R}_k\mathbf{R}_k^\top,\qquad k = 1,2.
\end{equation}
The formulas of $\mathbf{R}_1$ and $\mathbf{R}_2$ will be derived in Theorem \ref{thm:2DDissipationMat}. In \eqref{2DscaledMatExist}, the diagonal matrix $\mathbf{\Lambda}_1 = {\rm diag}\{\lambda_1^{(1)},\lambda_1^{(2)},\lambda_1^{(3)},\lambda_1^{(4)}\}$, where
\begin{align*}
    \lambda_1^{(1)} &= \frac{v_1(1-c_s^2)-c_s\sqrt{(1-v_1^2-v_2^2)(1-v_1^2-v_2^2c_s^2)}}{1-(v_1^2+v_2^2)c_s^2},\\
    \lambda_1^{(2)} &= \lambda_1^{(3)} = v_1,\\
    \lambda_1^{(4)} &= \frac{v_1(1-c_s^2)+c_s\sqrt{(1-v_1^2-v_2^2)(1-v_1^2-v_2^2c_s^2)}}{1-(v_1^2+v_2^2)c_s^2}
\end{align*}
are the four eigenvalues of the Jacobian matrix $\frac{\partial\mathbf{F}_1(\mathbf{U})}{\partial \mathbf{U}}$, and the diagonal matrix $\mathbf{\Lambda}_2 = {\rm diag}\{\lambda_2^{(1)},\lambda_2^{(2)},\lambda_2^{(3)},\lambda_2^{(4)}\}$, where
\begin{align*}
    \lambda_2^{(1)} &= \frac{v_2(1-c_s^2)-c_s\sqrt{(1-v_1^2-v_2^2)(1-v_2^2-v_1^2c_s^2)}}{1-(v_1^2+v_2^2)c_s^2},\\
    \lambda_2^{(2)} &= \lambda_2^{(3)} = v_2,\\
    \lambda_2^{(4)} &= \frac{v_2(1-c_s^2)+c_s\sqrt{(1-v_1^2-v_2^2)(1-v_2^2-v_1^2c_s^2)}}{1-(v_1^2+v_2^2)c_s^2}
\end{align*}
are the four eigenvalues of the Jacobian matrix $\frac{\partial\mathbf{F}_2(\mathbf{U})}{\partial \mathbf{U}}$.	
	\begin{theorem} \label{thm:2DDissipationMat}
		For the 2D RHD system with general Synge-type EOS \eqref{eq:gEOS}, the $x_1$-directional scaled eigenvector matrix $\mathbf{R}_1$ satisfying \eqref{2DscaledMatExist} is given by
		\begin{equation} \label{2DxDissMat}
			\mathbf{R}_1:= \tilde{\mathbf{R}}_1\sqrt{\tilde{\mathbf{D}}_1}:=
			\begin{pmatrix}
				1 & \frac{1}{\gamma} & \gamma v_2 & 1 \\
				h\gamma{\Delta_{\lambda_1^{(1)}}\lambda_1^{(1)}} & \Delta_{\theta} v_1 & 2h\gamma^2 v_1v_2 & h\gamma{\Delta_{\lambda_1^{(4)}}\lambda_1^{(4)}} \\
				h\gamma v_2 & \Delta_{\theta} v_2 & h(1+2\gamma^2v_2^2) & h\gamma v_2\\
				h\gamma{\Delta_{\lambda_1^{(1)}}} & \Delta_{\theta} & 2h\gamma^2v_2 & h\gamma{\Delta_{\lambda_1^{(4)}}}
			\end{pmatrix}
			\begin{pmatrix}
				\sqrt{d_1^{(1)}} & 0 & 0 & 0\\
				0 & \sqrt{d_1^{(2)}} & 0 & 0\\
				0 & 0 & \sqrt{d_1^{(3)}} & 0\\
				0 & 0 & 0 & \sqrt{d_1^{(4)}}
			\end{pmatrix},
		\end{equation}
		where $\Delta_{\lambda_1^{(1)}}=\frac{1-v_1^2}{1-v_1\lambda_1^{(1)}}$, $\Delta_{\lambda_1^{(4)}} = \frac{1-v_1^2}{1-v_1\lambda_1^{(4)}}$, $\Delta_{\theta} = h-\theta(1+e'(\theta))$, and the scaling coefficients $d_1^{(j)},\ j = 1,2,3,4,$ are defined as
		\begin{equation}\label{2DxdissipationMatrixCoeff}
			\left\{
			\begin{aligned}
				d_1^{(1)} &= \frac{M_1-N_1}{2},\\
				d_1^{(2)} &= \frac{\rho\gamma^3}{1+e'(\theta)},\\
				d_1^{(3)} &= \frac{\rho\theta}{h(1-v_1^2)\gamma},\\
				d_1^{(4)} &= \frac{M_1+N_1}{2},
			\end{aligned}
			\right.
		\end{equation}
		with $$M_1=\rho \gamma\left(\frac{e'(\theta)}{1+e'(\theta)}-\frac{\theta v_2^2}{h(1-v_1^2)}\right), \quad N_1 = \frac{\rho\theta v_1\sqrt{1-v_1^2-v_2^2c_s^2}}{hc_s(1-v_1^2)}.$$
		The scaled eigenvector matrix $\mathbf{R}_2$ satisfying \eqref{2DscaledMatExist} is given by
		\begin{equation} \label{2DyDissMat}
			\mathbf{R}_2:= \tilde{\mathbf{R}}_2\sqrt{\tilde{\mathbf{D}}_2}:=
			\begin{pmatrix}
				1 & \gamma v_1 & \frac{1}{\gamma} & 1 \\
				h\gamma v_1 & h(1+2\gamma^2v_1^2) & \Delta_{\theta} v_1 & h\gamma v_1\\
				h\gamma{\Delta_{\lambda_2^{(1)}}}\lambda_2^{(1)} & 2h\gamma^2v_1 v_2 & \Delta_{\theta} v_2 & h\gamma{\Delta_{\lambda_2^{(4)}}}\lambda_2^{(4)} \\
				h\gamma{\Delta_{\lambda_2^{(1)}}} & 2h\gamma^2v_1 & \Delta_{\theta} & h\gamma{\Delta_{\lambda_2^{(4)}}} 
			\end{pmatrix}
			\begin{pmatrix}
				\sqrt{d_2^{(1)}} & 0 & 0 & 0\\
				0 & \sqrt{d_2^{(2)}} & 0 & 0\\
				0 & 0 & \sqrt{d_2^{(3)}} & 0\\
				0 & 0 & 0 & \sqrt{d_2^{(4)}}
			\end{pmatrix},
		\end{equation}
		where $\Delta_{\lambda_2^{(1)}}=\frac{1-v_2^2}{1-v_2\lambda_2^{(1)}}$, $\Delta_{\lambda_2^{(4)}}=\frac{1-v_2^2}{1-v_2\lambda_2^{(4)}}$, $\Delta_{\theta} = h-\theta(1+e'(\theta))$, and the scaling coefficients $d_2^{(j)},\ j = 1,2,3,4,$ are defined as
		\begin{equation} \label{2DydissipationMatrixCoeff}
			\left\{
			\begin{aligned}
				d_2^{(1)} &= \frac{M_2-N_2}{2},\\
				d_2^{(2)} &= \frac{\rho\theta}{h(1-v_2^2)\gamma},\\
				d_2^{(3)} &= \frac{\rho\gamma^3}{1+e'(\theta)},\\
				d_2^{(4)} &= \frac{M_2+N_2}{2},
			\end{aligned}
			\right.
		\end{equation}
		with $$M_2=\rho \gamma\left(\frac{e'(\theta)}{1+e'(\theta)}-\frac{\theta v_1^2}{h(1-v_2^2)}\right), \quad N_2 = \frac{\rho\theta v_2\sqrt{1-v_2^2-v_1^2c_s^2}}{hc_s(1-v_2^2)}.$$
	\end{theorem}

	\begin{proof}
		Note that $\tilde{\mathbf{R}}_1$ defined in \eqref{2DxDissMat} is a right eigenvector matrix of the Jacobian matrix $\frac{\partial {\mathbf F}_1({\mathbf U})}{\partial {\mathbf U}}$; see   \cite{zhao2013runge}. Since $\sqrt{\mathbf{\tilde{D}}_1}$ is a diagonal matrix, we know that $\mathbf{R}_1=\mathbf{\tilde{R}}_1\sqrt{\mathbf{\tilde{D}}_1}$ is also a right eigenvector matrix of  $\frac{\partial {\mathbf F}_1({\mathbf U})}{\partial {\mathbf U}}$. Hence, $\mathbf{R}_1$ satisfies $ \frac{\partial\mathbf{F}_1}{\partial\mathbf{U}}=\mathbf{R}_1\mathbf{\Lambda}_1\mathbf{R}_1^{-1}$. We only need to verify that $\mathbf{R}_1$ satisfies $\frac{\partial \mathbf{U}}{\partial \mathbf{W}}=\mathbf{R}_1\mathbf{R}_1^\top$. Since $\mathbf{U}$ cannot be explicitly formulated by $\mathbf{W}$, we derive $\frac{\partial\mathbf{U}}{\partial\mathbf{W}}$ by the chain rule \eqref{U_W_chainRule}. As $\mathbf{W}$ can be explicitly formulated by $\mathbf{V}$, a direct calculation leads to 
		\begin{equation*} 
			\frac{\partial\mathbf{W}}{\partial\mathbf{V}} = \begin{pmatrix}
				\frac{h}{\rho\theta} & 0 & 0 & \frac{\theta-h}{\rho\theta^2}\\
				\frac{\gamma v_1}{\rho\theta} & \frac{\gamma^3(1-v_2^2)}{\theta} & \frac{\gamma^3v_1v_2}{\theta} & -\frac{\gamma v_1}{\rho\theta^2}\\
				\frac{\gamma v_2}{\rho\theta} & \frac{\gamma^3v_1v_2}{\theta} & \frac{\gamma^3(1-v_1^2)}{\theta} & -\frac{\gamma v_2}{\rho\theta^2}\\
				-\frac{\gamma}{\rho\theta} & -\frac{\gamma^3 v_1}{\theta} & -\frac{\gamma^3 v_2}{\theta} & \frac{\gamma}{\rho\theta^2}
			\end{pmatrix},
		\end{equation*}
		whose inverse matrix is given by 
		\begin{equation} \label{V_W:2D}
			\frac{\partial\mathbf{V}}{\partial\mathbf{W}} = \begin{pmatrix}
				\rho & \rho\gamma v_1(h-\theta) & \rho\gamma v_2(h-\theta) & \rho\gamma(h-\theta)\\
				0 & \frac{\theta}{\gamma} & 0 & \frac{\theta v_1}{\gamma}\\
				0 & 0 & \frac{\theta}{\gamma} & \frac{\theta v_2}{\gamma}\\
				\rho\theta & \rho h\gamma\theta v_1 & \rho h\gamma\theta v_2 & \rho h \gamma\theta
			\end{pmatrix}.
		\end{equation}
		Combining \eqref{V_W:2D} with \eqref{U_V} for the case $d=2$, we can compute $\frac{\partial \mathbf{U}}{\partial\mathbf{W}}$ by \eqref{U_W_chainRule} as
		
\begin{equation} \label{U_W:2D}
	\frac{\partial\mathbf{U}}{\partial\mathbf{W}}=\begin{pmatrix}
		\rho\gamma & \rho h\gamma^2 v_1 & \rho h\gamma^2 v_2 & \rho(-\theta+h\gamma^2) \\
		\rho h \gamma^2 v_1 & \rho\gamma^3\sigma_1 & \rho\gamma^3v_1v_2\sigma_2 & \rho\gamma^3 v_1\sigma_4\\
		\rho h \gamma^2 v_2 & \rho\gamma^3v_1v_2\sigma_2 & \rho\gamma^3\sigma_3 & \rho\gamma^3 v_2\sigma_4\\
		\rho(-\theta+h\gamma^2) & \rho\gamma^3 v_1\sigma_4 & \rho\gamma^3 v_2\sigma_4 & \rho\gamma^3\left(\sigma_4-2\theta h\left(1-(v_1^2+v_2^2)\right)\right)
	\end{pmatrix}
\end{equation}
with 
\begin{equation*}
	\begin{aligned}
		\sigma_1 &= \theta^2v_1^2(1+e'(\theta))+\theta h+h^2v_1^2-\theta hv_2^2,\qquad \sigma_2 = \theta^2(1+e'(\theta))+\theta h+h^2,\\
		\sigma_3 &= \theta^2v_2^2(1+e'(\theta))+\theta h+h^2v_2^2-\theta hv_1^2,\qquad \sigma_4 = \theta^2(1+e'(\theta))+h^2+\theta h(v_1^2+v_2^2).
	\end{aligned}
\end{equation*}
Let ${\bf r}_i$ be the $ith$ row of the scaled eigenvector matrix $\mathbf{R}_1$. In the following, we would like to verify the relation $\frac{\partial\mathbf{U}}{\partial\mathbf{W}}=\mathbf{R}_1\mathbf{R}_1^\top$ by calculating ${\bf r}_i {\bf r}_j^\top,\ i,j = 1,2,3,4$, and then comparing the results with \eqref{U_W:2D}. 
	Since the Lorentz factor couples the velocities $v_1$ and $v_2$, the structures of the eigenvectors and eigenvalues are much more complicated than the 1D case, making the verification more difficult. 
	To simplify our calculation, we first observe the following identities 
\begin{align*}
    t_1&:=\frac{\lambda_1^{(1)}}{1-v_1\lambda_1^{(1)}}
    =\frac{(1-c_s^2)v_1-c_s\sqrt{(1-v_1^2-v_2^2)(1-v_1^2-v_2^2c_s^2)}}{(1-v_1^2-v_2^2c_s^2)+c_sv_1\sqrt{(1-v_1^2-v_2^2)(1-v_1^2-v_2^2c_s^2)}},\\
    t_4&:=\frac{\lambda_1^{(4)}}{1-v_1\lambda_1^{(4)}}
    =\frac{(1-c_s^2)v_1+c_s\sqrt{(1-v_1^2-v_2^2)(1-v_1^2-v_2^2c_s^2)}}{(1-v_1^2-v_2^2c_s^2)-c_sv_1\sqrt{(1-v_1^2-v_2^2)(1-v_1^2-v_2^2c_s^2)}},\\
    \tilde{t}_1&:=\frac{1}{1-v_1\lambda_1^{(1)}}=\left(1-\frac{v_1\left((1-c_s^2)v_1-c_s\sqrt{(1-v_1^2-v_2^2)(1-v_1^2-v_2^2c_s^2)}\right)}{1-(v_1^2+v_2^2)c_s^2}\right)^{-1},\\
    \tilde{t}_4&:=\frac{1}{1-v_1\lambda_1^{(1)}}=\left(1-\frac{v_1\left((1-c_s^2)v_1+c_s\sqrt{(1-v_1^2-v_2^2)(1-v_1^2-v_2^2c_s^2)}\right)}{1-(v_1^2+v_2^2)c_s^2}\right)^{-1},
\end{align*}
from which we can further deduce that
\begin{align} 
    \label{t1t4}
    t_1+t_4 &= \frac{2v_1}{1-v_1^2},\qquad
    t_4-t_1 = \frac{2c_s\sqrt{1-v_1^2-v_2^2}}{(1-v_1^2)\sqrt{1-v_1^2-v_2^2c_s^2}},
    \\ \label{t1t4sq}
    t_1^2+t_4^2 &= \frac{2}{(1-v_1^2)^2}\left(v_1^2+\frac{c_s^2(1-v_1^2-v_2^2)}{(1-v_1^2-v_2^2c_s^2}\right),\qquad
    t_4^2-t_1^2 = \frac{4c_sv_1\sqrt{1-v_1^2-v_2^2}}{(1-v_1^2)^2\sqrt{1-v_1^2-v_2^2c_s^2}},
    \\ \label{tildet1t4}
    \tilde{t}_1+\tilde{t}_4 &= \frac{2}{1-v_1^2},\qquad
    \tilde{t}_4-\tilde{t}_1 = \frac{2c_sv_1\sqrt{1-v_1^2-v_2^2}}{(1-v_1^2)\sqrt{1-v_1^2-v_2^2c_s^2}},
    \\ \label{tildet1t4sq}
    \tilde{t}_1^2+\tilde{t}_4^2 &= \frac{2}{(1-v_1^2)^2}\left(1+\frac{c_s^2v_1^2(1-v_1^2-v_2^2)}{1-v_1^2-v_2^2c_s^2}\right),\qquad
    \tilde{t}_4^2-\tilde{t}_1^2 = \frac{4c_sv_1\sqrt{1-v_1^2-v_2^2}}{(1-v_1^2)^2\sqrt{1-v_1^2-v_2^2c_s^2}},
    \\ \label{t1tpmt4t}
    t_1\tilde{t}_1+t_4\tilde{t}_4 &= \frac{2v_1}{(1-v_1^2)^2}\left(1+\frac{c_s^2(1-v_1^2-v_2^2)}{1-v_1^2-v_2^2c_s^2}\right),\qquad
    t_4\tilde{t}_4-t_1\tilde{t}_1 = \frac{2c_s(1+v_1^2)\sqrt{1-v_1^2-v_2^2}}{(1-v_1^2)^2\sqrt{1-v_1^2-v_2^2c_s^2}}.
\end{align}
Using \eqref{cs_theta} and \eqref{t1t4}--\eqref{t1tpmt4t}, we calculate 
		\begin{align*}
		    {\bf r}_1 {\bf r}_1^\top &= d_1^{(1)}+\frac{d_1^{(2)}}{\gamma^2}+\gamma^2v_2^2d_1^{(3)}+d_1^{(4)}=M_1+\frac{\rho \gamma}{1+e'(\theta)}+\frac{\rho\theta\gamma v_2^2}{h(1-v_1^2)}\\
		    &=\frac{\rho e'(\theta) \gamma}{1+e'(\theta)}+\frac{\rho \gamma}{1+e'(\theta)}=\rho\gamma=\left(\frac{\partial\mathbf{U}}{\partial\mathbf{W}}\right)_{1,1},\\
		    {\bf r}_1 {\bf r}_2^\top &= h\gamma(1-v_1^2) t_1d_1^{(1)}+\frac{\Delta_{\theta} v_1}{\gamma}d_1^{(2)}+2h\gamma^3v_1v_2^2d_1^{(3)}+h\gamma(1-v_1^2) t_4d_1^{(4)}\\
		    &=\frac{M_1}{2}h\gamma (1-v_1^2)(t_1+t_4)+\frac{N_1}{2}h\gamma (1-v_1^2)(t_4-t_1)+\frac{\rho\gamma^2\Delta_{\theta}v_1}{1+e'(\theta)}+\frac{2\rho\theta\gamma^2v_1v_2^2}{1-v_1^2}\\
		    &\overset{\eqref{t1t4}}{=}\rho h\gamma^2v_1\left(\frac{e'(\theta)}{1+e'(\theta)}-\frac{\theta v_2^2}{h(1-v_1^2)}\right)+\frac{\rho\theta v_1}{1-v_1^2}+\frac{\rho\gamma^2\Delta_{\theta}v_1}{1+e'(\theta)}+\frac{2\rho\theta\gamma^2v_1v_2^2}{1-v_1^2}\\
		    &=\rho\gamma^2v_1\left(h-\theta+\frac{\theta v_2^2}{1-v_1^2}\right)+\frac{\rho\theta v_1}{1-v_1^2}
		    =\rho h\gamma^2v_1 = \left(\frac{\partial\mathbf{U}}{\partial\mathbf{W}}\right)_{1,2},\\
		   {\bf r}_1 {\bf r}_3^\top &= h\gamma v_2d_1^{(1)}+\frac{\Delta_{\theta}v_2}{\gamma}d_1^{(2)}+h\gamma v_2(1+\gamma^2v_2^2)d_1^{(3)}+h\gamma v_2d_1^{(4)}\\
		    &=\rho h\gamma^2 v_2\left(\frac{e'(\theta)}{1+e'(\theta)}-\frac{\theta v_2^2}{h(1-v_1^2)}\right)+\frac{\rho\gamma^2\Delta_{\theta}v_2}{1+e'(\theta)}+\frac{\rho\theta v_2}{1-v_1^2}(1+2\gamma^2v_2^2)\\
		    &=\rho\gamma^2v_2(h-\theta)+\frac{\rho\theta v_2}{1-v_1^2}(1+\gamma^2v_2^2)
		    =\rho h\gamma^2v_2=\left(\frac{\partial\mathbf{U}}{\partial\mathbf{W}}\right)_{1,3},\\
		    {\bf r}_1 {\bf r}_4^\top &= h\gamma(1-v_1^2) \tilde{t}_1d_1^{(1)}+\frac{\Delta_{\theta}}{\gamma}d_1^{(2)}+2h\gamma^3v_2^2d_1^{(3)}+h\gamma(1-v_1^2) \tilde{t}_4d_1^{(4)}\\
		    &=\frac{M_1}{2}h\gamma (1-v_1^2)(\tilde{t}_1+\tilde{t}_4)+\frac{N_1}{2}h\gamma (1-v_1^2)(\tilde{t}_4-\tilde{t}_1)+\frac{\rho\gamma^2\Delta_{\theta}}{1+e'(\theta)}+\frac{2\rho\theta\gamma^2v_2^2}{1-v_1^2}\\
		    &\overset{\eqref{tildet1t4}}{=}\rho h\gamma^2\left(\frac{e'(\theta)}{1+e'(\theta)}-\frac{\theta v_2^2}{h(1-v_1^2)}\right)+\frac{\rho\theta v_1^2}{1-v_1^2}+\frac{\rho\gamma^2\Delta_{\theta}}{1+e'(\theta)}+\frac{2\rho\theta\gamma^2v_2^2}{1-v_1^2}\\
		    &=\rho\gamma^2(h-\theta)+\rho\theta\gamma^2(v_1^2+v_2^2)
		    =\rho h\gamma^2-\rho\theta = \left(\frac{\partial\mathbf{U}}{\partial\mathbf{W}}\right)_{1,4},\\
		    {\bf r}_2 {\bf r}_2^\top &= h^2\gamma^2(1-v_1^2)^2t_1^2d_1^{(1)}+\Delta_{\theta}^2v_1^2d_1^{(2)}+4h^2\gamma^4v_1^2v_2^2d_1^{(3)}+h^2\gamma^2(1-v_1^2)^2t_4^2d_1^{(4)}\\
		    &=h^2\gamma^2(1-v_1^2)^2\left(\frac{M_1}{2}(t_1^2+t_4^2)+\frac{N_1}{2}(t_4^2-t_1^2)\right)+\frac{\rho\gamma^3\Delta_{\theta}^2v_1^2}{1+e'(\theta)}+\frac{4\rho\theta h\gamma^3v_1^2v_2^2}{1-v_1^2}\\
		    &\overset{\eqref{t1t4sq}}{=}h^2\gamma^2\left(\rho\gamma\left(\frac{e'(\theta)}{1+e'(\theta)}-\frac{\theta v_2^2}{h(1-v_1^2)}\right)\left(v_1^2+\frac{c_s^2}{\gamma^2(1-v_1^2-v_2^2c_s^2}\right)+\frac{2\rho\theta v_1^2}{h\gamma(1-v_1^2)}\right)\\
		    &\quad +\frac{\rho\gamma^3\Delta_{\theta}^2v_1^2}{1+e'(\theta)}+\frac{4\rho\theta h\gamma^3v_1^2v_2^2}{1-v_1^2}\\
		    &=\rho h^2\gamma^3\left(v_1^2+\frac{c_s^2}{\gamma^2(1-v_1^2-v_2^2c_s^2}\right)\frac{e'(\theta)}{1+e'(\theta)}+\frac{\rho\gamma^3\Delta_{\theta}^2v_1^2}{1+e'(\theta)}\\
		    &\quad -\rho h^2\gamma^3\left(v_1^2+\frac{c_s^2}{\gamma^2(1-v_1^2-v_2^2c_s^2}\right)\frac{\theta v_2^2}{h(1-v_1^2)}+\frac{2\rho h\gamma\theta v_1^2}{1-v_1^2}+\frac{4\rho\theta h\gamma^3v_1^2v_2^2}{1-v_1^2}\\
		    &\overset{\eqref{cs_theta}}{=}\rho\gamma^3\left(h^2v_1^2+\theta^2(1+e'(\theta)v_1^2+\frac{\theta hv_1^2v_2^2}{1-v_1^2}+\frac{\theta h}{\gamma^2(1-v_1^2)}\right)\\
		    &=\rho\gamma^3\left(h^2v_1^2+\theta^2(1+e'(\theta)v_1^2+\theta h(1-v_2^2)\right) = \left(\frac{\partial\mathbf{U}}{\partial\mathbf{W}}\right)_{2,2},\\
		    {\bf r}_2 {\bf r}_3^\top &= h^2\gamma^2v_2(1-v_1^2)t_1d_1^{(1)}+\Delta_{\theta}^2v_1v_2d_1^{(2)}+2h^2\gamma^2v_1v_2(1+2\gamma^2v_2^2)d_1^{(3)}+h^2\gamma^2v_2(1-v_1^2)t_4d_1^{(4)}\\
		    &=h^2\gamma^2v_2(1-v_1^2)\left(\frac{M_1}{2}(t_1+t_4)+\frac{N_1}{2}(t_4-t_1)\right)+\frac{\rho\gamma^3\Delta_{\theta}^2v_1v_2}{1+e'(\theta)}+\frac{2\rho\theta h\gamma v_1v_2(1+2\gamma^2v_2^2)}{1-v_1^2}\\
		    &\overset{\eqref{t1t4}}{=}\rho h^2\gamma^3v_1v_2\left(\frac{e'(\theta)}{1+e'(\theta)}-\frac{\theta v_2^2}{h(1-v_1^2)}\right)+\frac{\rho\theta h\gamma v_1v_2}{1-v_1^2}+\frac{\rho\gamma^3\Delta_{\theta}^2v_1v_2}{1+e'(\theta)}+\frac{2\rho\theta h\gamma v_1v_2(1+2\gamma^2v_2^2)}{1-v_1^2}\\
		    &=\rho \gamma^3v_1v_2\left(h^2-2\theta h+\theta^2(1+e'(\theta))\right)+\frac{3\rho\theta h\gamma v_1v_2}{1-v_1^2}(1+\gamma^2v_2^2)\\
		    &=\rho\gamma^3v_1v_2\left(h^2+\theta h+\theta^2(1+e'(\theta))\right)=\left(\frac{\partial\mathbf{U}}{\partial\mathbf{W}}\right)_{2,3},\\
		    {\bf r}_2 {\bf r}_4^\top &= h^2\gamma^2(1-v_1^2)^2t_1\tilde{t}_1d_1^{(1)}+\Delta_{\theta}^2v_1d_1^{(2)}+4h^2\gamma^4v_1v_2^2d_1^{(3)}+h^2\gamma^2(1-v_1^2)^2t_4\tilde{t}_4\\
		    &=h^2\gamma^2(1-v_1^2)^2\left(\frac{M_1}{2}(t_1\tilde{t}_1+t_4\tilde{t}_4)+\frac{N_1}{2}(t_4\tilde{t}_4-t_1\tilde{t}_1\right)+\frac{\rho\gamma^3\Delta_{\theta}^2v_1}{1+e'(\theta)}+\frac{4\rho\theta h\gamma^3v_1v_2^2}{1-v_1^2}\\
		    &\overset{\eqref{t1tpmt4t}}{=}h^2\gamma^2v_1\left(\rho\gamma\left(\frac{e'(\theta)}{1+e'(\theta)}-\frac{\theta v_2^2}{h(1-v_1^2)}\right)\left(1+\frac{c_s^2(1-v_1^2-v_2^2)}{1-v_1^2-v_2^2c_s^2}\right)+\frac{\rho\theta(1+v_1^2)}{h\gamma(1-v_1^2)}\right)\\
		    &\quad +\frac{\rho\gamma^3\Delta_{\theta}^2v_1}{1+e'(\theta)}+\frac{4\rho\theta h\gamma^3v_1v_2^2}{1-v_1^2}\\
		    &=\rho h^2\gamma^3v_1\left(1+\frac{c_s^2(1-v_1^2-v_2^2)}{1-v_1^2-v_2^2c_s^2}\right)\frac{e'(\theta)}{1+e'(\theta)}+\frac{\rho\gamma^3\Delta_{\theta}^2v_1}{1+e'(\theta)}\\
		    &\quad -\frac{\rho h\gamma^3\theta v_1v_2^2}{1-v_1^2}\left(1+\frac{c_s^2(1-v_1^2-v_2^2)}{1-v_1^2-v_2^2c_s^2}\right)+\frac{\rho h\gamma\theta v_1(1+v_1^2)}{1-v_1^2}+\frac{4\rho\theta h\gamma^3v_1v_2^2}{1-v_1^2}\\
		    &\overset{\eqref{cs_theta}}{=}\rho\gamma^3v_1\left(h^2-2\theta h+\theta^2(1+e'(\theta))\right)+\frac{\rho h\gamma\theta v_1}{1-v_1^2}(2+v_1^2+3\gamma^2v_2^2)\\
		    &=\rho\gamma^3v_1\left(h^2+\theta h(v_1^2+v_2^2)+\theta^2(1+e'(\theta))\right)=\left(\frac{\partial\mathbf{U}}{\partial\mathbf{W}}\right)_{2,4},\\
		    {\bf r}_3 {\bf r}_3^\top &= h^2\gamma^2v_2^2d_1^{(1)}+\Delta_{\theta}^2v_2^2d_1^{(2)}+h^2(1+2\gamma^2v_2^2)^2d_1^{(3)}+h^2\gamma^2v_2^2d_1^{(4)}\\
		    &=\rho\gamma^3v_2^2\left(\frac{h^2e'(\theta)}{1+e'(\theta)}+\frac{\Delta_{\theta}^2}{1+e'(\theta)}\right)+\frac{\rho\theta h}{\gamma(1-v_1^2)}\left(-\gamma^4v_2^4+(1+2\gamma^2v_2^2)^2\right)\\
		    &=\rho\gamma^3(h^2v_2^2+\theta h(1-v_1^2)+\theta^2(1+e'(\theta))v_2^2) = \left(\frac{\partial\mathbf{U}}{\partial\mathbf{W}}\right)_{3,3},\\
		    {\bf r}_3 {\bf r}_4^\top &=  h^2\gamma^2v_2(1-v_1^2)\tilde{t}_1d_1^{(1)}+\Delta_{\theta}^2v_2d_1^{(2)}+2h^2\gamma^2v_2(1+2\gamma^2v_2^2)d_1^{(3)}+h^2\gamma^2v_2(1-v_1^2)\tilde{t}_4d_1^{(4)}\\
		    &=h^2\gamma^2v_2(1-v_1^2)\left(\frac{M_1}{2}(\tilde{t}_1+\tilde{t}_4)+\frac{N_1}{2}(\tilde{t}_4-\tilde{t}_1)\right)+\frac{\rho\gamma^3\Delta_{\theta}^2v_2}{1+e'(\theta)}+\frac{2\rho\theta h\gamma v_2(1+2\gamma^2v_2^2)}{1-v_1^2}\\
		    &\overset{\eqref{tildet1t4}}{=}h^2\gamma^2v_2\left(\rho\gamma\left(\frac{e'(\theta)}{1+e'(\theta)}-\frac{\theta v_2^2}{h(1-v_1^2)}\right)+\frac{\rho\theta v_1^2}{h\gamma(1-v_1^2)}\right)+\frac{\rho\gamma^3\Delta_{\theta}^2v_2}{1+e'(\theta)}+\frac{2\rho\theta h\gamma v_2(1+2\gamma^2v_2^2)}{1-v_1^2}\\
		    &=\rho\gamma^3v_2\left(\frac{h^2e'(\theta)}{1+e'(\theta)}+\frac{\Delta_{\theta}^2}{1+e'(\theta)}\right)+\frac{\rho\theta h\gamma v_2}{1-v_1^2}\left(2(1+2\gamma^2v_2^2)-\gamma^2v_2^2+v_1^2\right)\\
		    &=\rho\gamma^3v_2\left(h^2+\theta h(v_1^2+v_2^2)+\theta^2(1+e'(\theta))\right)=\left(\frac{\partial\mathbf{U}}{\partial\mathbf{W}}\right)_{3,4},\\
		    {\bf r}_4 {\bf r}_4^\top &= h^2\gamma^2(1-v_1^2)^2\tilde{t}_1^2d_1^{(1)}+\Delta_{\theta}^2d_1^{(2)}+4h^2\gamma^4v_2^2d_1^{(3)}+h^2\gamma^2(1-v_1^2)^2\tilde{t}_4^2d_1^{(4)}\\
		    &=h^2\gamma^2(1-v_1^2)^2\left(\frac{M_1}{2}(\tilde{t}_1^2+\tilde{t}_4^2)+\frac{N_1}{2}(\tilde{t}_4^2-\tilde{t}_1^2)\right)+\frac{\rho\gamma^3\Delta_{\theta}^2}{1+e'(\theta)}+\frac{4\rho\theta h\gamma^3v_2^2}{1-v_1^2}\\
		    &\overset{\eqref{tildet1t4sq}}{=}h^2\gamma^2\left(\rho\gamma\left(\frac{e'(\theta)}{1+e'(\theta)}-\frac{\theta v_2^2}{h(1-v_1^2)}\right)\left(1+\frac{c_s^2v_1^2(1-v_1^2-v_2^2)}{1-v_1^2-v_2^2c_s^2}\right)+\frac{2\rho\theta v_1^2}{h\gamma(1-v_1^2)}\right)\\
		    &\quad +\frac{\rho\gamma^3\Delta_{\theta}^2}{1+e'(\theta)}+\frac{4\rho\theta h\gamma^3v_2^2}{1-v_1^2}\\
		    &=\rho\gamma^3\left(\frac{h^2e'(\theta)}{1+e'(\theta)}+\frac{\Delta_{\theta}^2}{1+e'(\theta)}\right)+\frac{\rho\theta h\gamma}{1-v_1^2}\left(2v_1^2+3\gamma^2v_2^2\right)+\frac{\rho\gamma v_1^2}{1-v_1^2-v_2^2c_s^2}\left(\frac{h^2e'(\theta)c_s^2}{1+e'(\theta)}-\frac{\theta hv_2^2c_s^2}{1-v_1^2}\right)\\
		    &\overset{\eqref{cs_theta}}{=}\rho\gamma^3(h^2-2\theta h+\theta^2(1+e'(\theta))+\frac{3\rho\gamma h \theta}{1-v_1^2}(v_1^2+\gamma^2v_2^2)\\
		    &=\rho\gamma^3(h^2-2\theta h+\theta^2(1+e'(\theta))+3\theta h(v_1^2+v_2^2))=\left(\frac{\partial\mathbf{U}}{\partial\mathbf{W}}\right)_{4,4}.
		\end{align*}
		Noting that both matrices $\mathbf{R}_1\mathbf{R}_1^\top$ and $\frac{\partial\mathbf{U}}{\partial\mathbf{W}}$ in \eqref{U_W:2D} are symmetric, we have
		\begin{align*}
		    {\bf r}_2 {\bf r}_1^\top &= \left(\frac{\partial\mathbf{U}}{\partial\mathbf{W}}\right)_{2,1},\quad
		    {\bf r}_3 {\bf r}_1^\top = \left(\frac{\partial\mathbf{U}}{\partial\mathbf{W}}\right)_{3,1},\quad
		    {\bf r}_3 {\bf r}_2^\top = \left(\frac{\partial\mathbf{U}}{\partial\mathbf{W}}\right)_{3,2},\\
		    {\bf r}_4 {\bf r}_1^\top &= \left(\frac{\partial\mathbf{U}}{\partial\mathbf{W}}\right)_{4,1},\quad
		    {\bf r}_4 {\bf r}_2^\top = \left(\frac{\partial\mathbf{U}}{\partial\mathbf{W}}\right)_{4,2},\quad
		    {\bf r}_4 {\bf r}_3^\top = \left(\frac{\partial\mathbf{U}}{\partial\mathbf{W}}\right)_{4,3}.
		\end{align*}
		Hence, we have verified that $\frac{\partial \mathbf{U}}{\partial \mathbf{W}}=\mathbf{R}_1\mathbf{R}_1^\top$.  Next, in order to verify $\frac{\partial \mathbf{U}}{\partial \mathbf{W}}=\mathbf{R}_2\mathbf{R}_2^\top$, we consider the rotation matrix  
		$$T=
		\begin{pmatrix}
		    1 & 0 & 0 & 0\\
		    0 & 0 & 1 & 0\\
		    0 & 1 & 0 & 0\\
		    0 & 0 & 0 & 1
		\end{pmatrix}.$$
		We observe that 
		\begin{align*}
		    \tilde{\mathbf{R}}_2(\mathbf{V}) &= T\tilde{\mathbf{R}}_1(T\mathbf{V})T,\qquad
		    \tilde{\mathbf{D}}_2(\mathbf{V}) = T\tilde{\mathbf{D}}_1(T\mathbf{V})T,\qquad
		    \mathbf{R}_1(T\mathbf{V})\mathbf{R}_1(T\mathbf{V})^\top=T\frac{\partial\mathbf{U}}{\partial\mathbf{W}}T,\\
		    \mathbf{\Lambda}_2(\mathbf{V})&=\mathbf{\Lambda}_1(T\mathbf{V}), \qquad
		    \mathbf{F}_2(\mathbf{V})=T\mathbf{F}_1(\mathbf{T\mathbf{V}}), \qquad 
		    \frac{\partial\mathbf{F}_2(\mathbf{U})}{\partial\mathbf{U}}=T\frac{\partial\mathbf{F}_1(T\mathbf{U})}{\partial\mathbf{U}}=T\frac{\partial\mathbf{F}_1(T\mathbf{U})}{\partial T\mathbf{U}}T.
		\end{align*}
		Then we have
		\begin{align*}
		    \mathbf{R}_2(\mathbf{V})\mathbf{\Lambda}_2(\mathbf{V})\mathbf{R}_2^{-1}(\mathbf{V})&=\tilde{\mathbf{R}}_2(\mathbf{V})\sqrt{\tilde{\mathbf{D}}_2(\mathbf{V})}\mathbf{\Lambda}_2(\mathbf{V})\left(\sqrt{\tilde{\mathbf{D}}_2(\mathbf{V})}\right)^{-1}\tilde{\mathbf{R}}_2^{-1}(\mathbf{V})\\
		    &=T\tilde{\mathbf{R}}_1(T\mathbf{V})T\mathbf{\Lambda}_1(T\mathbf{V})T\tilde{\mathbf{R}}_1^{-1}(T\mathbf{V})T\\
		    &=T\tilde{\mathbf{R}}_1(T\mathbf{V})\sqrt{\tilde{\mathbf{D}}_1(T\mathbf{V})}\mathbf{\Lambda}_1(T\mathbf{V})\left(\sqrt{\tilde{\mathbf{D}}_1(T\mathbf{V})}\right)^{-1}\tilde{\mathbf{R}}_1^{-1}(T\mathbf{V})T\\
		    &=T\mathbf{R}_1(T\mathbf{V})\mathbf{\Lambda}_1(T\mathbf{V})\mathbf{R}_1^{-1}(T\mathbf{V})T
		    =T\frac{\partial\mathbf{F}_1(T\mathbf{U})}{\partial T\mathbf{U}}T=\frac{\partial\mathbf{F}_2(\mathbf{U})}{\partial\mathbf{U}},
		\end{align*}
		and
		\begin{align*}
		    \mathbf{R}_2(\mathbf{V})\mathbf{R}_2(\mathbf{V})^\top &= 
		    \tilde{\mathbf{R}}_2(\mathbf{V})\tilde{\mathbf{D}}_2(\mathbf{V})\tilde{\mathbf{R}}_2(\mathbf{V})^\top
		    =T\tilde{\mathbf{R}}_1(T\mathbf{V})\tilde{\mathbf{D}}_1(T\mathbf{V})\tilde{\mathbf{R}}_1(T\mathbf{V})^\top T\\
		    &=T\mathbf{R}_1(T\mathbf{V})\mathbf{R}_1(T\mathbf{V})^\top T
		    =T^2\frac{\partial\mathbf{U}}{\partial\mathbf{W}}T^2=\frac{\partial\mathbf{U}}{\partial\mathbf{W}}.
		\end{align*}
		The proof is completed.

\end{proof}

	Since the dissipation matrices \eqref{2DxDissMat} and \eqref{2DyDissMat} are defined at the cell interfaces, we should estimate the quantities in the dissipation matrices by some ``averged'' states. Similar to the 1D case in Theorem \ref{thm:enthalpyEvaluation}, we evaluate $h_{i+\frac{1}{2}}$ appropriately to obtain an accurate resolution of stationary contact discontinuities. The averages of the other quantities are consistent with the choice in the 1D case.

\begin{remark} \label{rmk:RRK}
	Up to now, we have achieved semi-discrete EC and ES schemes for the RHD equations, which can be written into an ODE system  
	    \begin{equation}\label{eq:ODEsemi}
		\frac{d}{dt}\mathbf{\mathcal{U}} = \mathbf{\mathcal{L}}(\mathbf{\mathcal{U}}), 
	\end{equation}
	where $\mathbf{\mathcal{U}}$ and $\mathbf{\mathcal{L}}(\mathbf{\mathcal{U}})$ in the 1D case read 
	\begin{equation*}
		\mathbf{\mathcal{U}} = \left(\mathbf{U}_1,\ldots,\mathbf{U}_N\right)^\top,\qquad
		\mathbf{\mathcal{L}}(\mathbf{\mathcal{U}}) = \left(\frac{\bar{\mathbf{F}}_{\frac{1}{2}}-\bar{\mathbf{F}}_{\frac{3}{2}}}{\Delta x},\ldots,\frac{\bar{\mathbf{F}}_{i-\frac{1}{2}}-\bar{\mathbf{F}}_{i+\frac{1}{2}}}{\Delta x},\ldots,\frac{\bar{\mathbf{F}}_{N-\frac{1}{2}}-\bar{\mathbf{F}}_{N+\frac{1}{2}}}{\Delta x}\right)^\top,
	\end{equation*}
	where $\bar{\mathbf{F}} = \widetilde{\mathbf{F}}$ for EC schemes and $\bar{\mathbf{F}} = \widehat{\mathbf{F}}$ for ES schemes.
	The semi-discrete system \eqref{eq:ODEsemi} can be further discretized in time by using some Runge--Kutta (RK) methods, for example, the classic third-order strong-stability-preserving RK (SSP-RK3) method:
\begin{equation}\label{eq:SSP-RK3}
	    \begin{aligned}
		\mathbf{\mathcal{U}}^{(1)} &= \mathbf{\mathcal{U}}^n + \Delta t \mathcal{L}(\mathbf{\mathcal{U}}^n),\\
		\mathbf{\mathcal{U}}^{(2)} &= \frac{3}{4}\mathbf{\mathcal{U}}^n + \frac{1}{4} \left(\mathbf{\mathcal{U}}^{(1)}+\Delta t\mathcal{L}(\mathbf{\mathcal{U}}^{(1)})\right),\\
		\mathbf{\mathcal{U}}^{n+1} &= \frac{1}{3}\mathbf{\mathcal{U}}^n + \frac{2}{3} \left(\mathbf{\mathcal{U}}^{(2)}+\Delta t\mathcal{L}(\mathbf{\mathcal{U}}^{(2)})\right).
	\end{aligned}
\end{equation}	
	The semi-discrete ES schemes coupled with SSP-RK3 method work well for many benchmark problems, but the rigorous analysis of the fully discrete ES property is yet unavailable in theory. 
	 To obtain the fully discrete, provably EC/ES schemes, one can employ the relaxation RK (RRK) method developed in \cite{ketcheson2019relaxation, ranocha2020relaxation}. For example, we can use the third order RRK (RRK3) method \cite{ranocha2020relaxation}: 
\begin{equation}\label{eq:RRK3}
	    \begin{aligned}
        \mathbf{\mathcal{U}}^{(1)} &= \mathbf{\mathcal{U}}^n + \Delta t \mathcal{L}(\mathbf{\mathcal{U}}^n),\\
        \mathbf{\mathcal{U}}^{(2)} &= \frac{3}{4}\mathbf{\mathcal{U}}^n + \frac{1}{4} \left(\mathbf{\mathcal{U}}^{(1)}+\Delta t\mathcal{L}(\mathbf{\mathcal{U}}^{(1)})\right),\\
        \mathbf{\mathcal{U}}^{n+1} &= \mathbf{\mathcal{U}}^n + \frac{1}{6}\gamma_n\Delta t \left(\mathcal{L}(\mathbf{\mathcal{U}}^n)+\mathcal{L}(\mathbf{\mathcal{U}}^{(1)})+4\mathcal{L}(\mathbf{\mathcal{U}}^{(2)})\right),
    \end{aligned}
\end{equation}
    where $\gamma_n$ is the relaxation parameter. 
    Define the total entropy $\mathcal{E}:=\sum\limits_{i=1}^N\eta(\mathbf{U}_i)\Delta x$, then the relaxation parameter $\gamma_n$ is computed in each time step by solving the following scalar algebraic equation
    \begin{equation*}
        r(\gamma) = \mathcal{E}(\mathbf{\mathcal{U}}^n+\gamma \mathbf{d}^n)-\mathcal{E}(\mathbf{\mathcal{U}}^n) - \gamma \epsilon=0,
    \end{equation*}
    where
    \begin{align*}
        \mathbf{d}^n &= \frac{\Delta t}{6}\left(\mathcal{L}(\mathbf{\mathcal{U}}^n)+\mathcal{L}(\mathbf{\mathcal{U}}^{(1)})+4\mathcal{L}(\mathbf{\mathcal{U}}^{(2)})\right),\\
        \epsilon &= \frac{\Delta t}{6}\sum\limits_{i=1}^N\left(\mathbf{W}(\mathbf{U}_i^n)^\top \mathcal{L}_i(\mathbf{\mathcal{U}}^n)+\mathbf{W}(\mathbf{U}_i^{(1)})^\top\mathcal{L}_i(\mathbf{\mathcal{U}}^{(1)})+4\mathbf{W}(\mathbf{U}_i^{(2)})^\top \mathcal{L}_i(\mathbf{\mathcal{U}}^{(2)}) \right) \Delta x
    \end{align*}
	with $\mathcal{L}_i(\mathbf{\mathcal{U}}):= \frac{1}{\Delta x} ( \bar{\mathbf{F}}_{i-\frac{1}{2}}-\bar{\mathbf{F}}_{i+\frac{1}{2}} )$. 
\end{remark}


\section{Numerical experiments}
\label{section:5}

In this section, we present a series of numerical experiments to demonstrate the accuracy and effectiveness of our high-order accurate EC and ES schemes for 1D and 2D RHD with various special EOSs. Specifically, we investigate the sixth-order and fourth-order accurate EC schemes, referred to as EC6 and EC4 respectively, as well as the fifth-order accurate ES scheme with WENO-based numerical flux (abbreviated as ES5), and the fourth-order accurate ES scheme with ENO-based numerical flux (abbreviated as ES4). To obtain the fully discrete schemes, we use either RRK3 \eqref{eq:RRK3} or SSP-RK3 \eqref{eq:SSP-RK3} for time discretization. To illustrate the importance of ES property, we will also present a comparison between our EC/ES schemes and a non-EC, non-ES scheme.
Unless otherwise specified, the CFL number for all tests is set as 0.4, and the Rusanov-type dissipation term \eqref{Rusanov-type} is used. The choice of EOS will be specified in each test case.

\subsection{One-dimensional examples}

\begin{example}[Accuracy test]\label{Ex5.1.1}
In this example, we evaluate the accuracy of our high-order EC and ES schemes for 1D RHD using a smooth problem with the initial data provided by
$$
\textbf{V}(x,0) = \left(1+0.2{\rm sin}(x),0.2,1\right)^\top.
$$
The periodic boundary conditions and TM-EOS \eqref{hEOS3} are employed. The exact solution of this problem is given by 
\begin{equation*}
    \textbf{V}(x,t) = \left(1+0.2{\rm sin}\left(x-0.2t\right),0.2,1\right)^\top, 
\end{equation*}
which describes a sine wave propagating in the domain $[0,2\pi]$.

To examine the spatial accuracy, we set the mesh size as $\Delta x=\frac{2\pi}{N}$ with a varying number of uniformly distributed grids $N \in \{10,20,40,80,160\}$. 
We test both RRK3 and SSP-RK3 methods for time discretization.
The time step-size is chosen to match the spatial accuracy, with $\Delta t = 0.4\Delta x^2$ for EC6, $\Delta t = 0.4\Delta x^\frac{5}{3}$ for ES5, and $\Delta t = 0.4\Delta x^\frac{4}{3}$ for EC4 and ES4, respectively. 
The quadruple precision is used for implementation.  
The $l^1$ and $l^2$ errors at $t=1.5$ in the rest-mass density $\rho$ and the corresponding rates are shown in Table \ref{table:1DEC6ES5_EOS} for EC6 and ES5, and in Table \ref{table:1DEC4ES4_EOS_3} for EC4 and ES4, respectively. 
As shown in the tables, the schemes achieve the expected convergence orders.  
We also verify the time evolution of discrete total entropy $\sum_i\eta(\mathbf{U}_i(t))\Delta x$ up to $t=1.5$. As shown in Figures \ref{Fig:evolutionDiscreteEntropy_EC6ES5} and \ref{Fig:evolutionDiscreteEntropy_EC4ES4}, the discrete total entropy decreases for ES5 and ES4 schemes, owing to their numerical dissipation mechanism, but this effect reduces as the number of grids increases. For EC6 and EC4 schemes, the discrete total entropy remains nearly unchanged with time, as expected.

\begin{table}[htb]
	\begin{center}
		\caption{$l^1$ and $l^2$ errors in $\rho$ and the corresponding convergence rates for EC6 and ES5 at different grid resolutions. }\label{table:1DEC6ES5_EOS}
		\begin{tabular}{c|c||c|c|c|c||c|c|c|c} 
			\hline
			\multirow{2}{*}{} & \multirow{2}{*}{$N$} & \multicolumn{4}{c||}{EC6} & \multicolumn{4}{c}{ES5}\\
			\cline{3-10}
			& & $l^1$ error & order & $l^2$ error & order  
			& $l^1$ error & order & $l^2$ error & order   \\
			\hline
			\multirow{5}{*}{SSP-RK3} & 10 & 1.7104e-04 & - & 9.5550e-05 & - & 6.0735e-03 & - & 2.6810e-03 & -  \\
			& 20 & 3.4854e-06 & 5.6169 & 2.0375e-06 & 5.5514 & 2.9496e-04 & 4.3639 & 1.4836e-04 & 4.1757   \\
			& 40 & 5.8181e-08 & 5.9046 & 3.4831e-08 & 5.8703 & 1.0087e-05 & 4.8699 & 5.4064e-06 & 4.7782   \\
			& 80 & 9.2642e-10 & 5.9727 & 5.5718e-10 & 5.9661 & 3.5354e-07 & 4.8345 & 1.9520e-07 & 4.7916   \\
			& 160 & 1.4706e-11 & 5.9772 & 8.7673e-12 & 5.9899 & 1.1270e-08 & 4.9713 & 6.1611e-09 & 4.9856     \\
			\hline
			\multirow{5}{*}{RRK3} & 10 & 1.7104e-04 & - & 9.5550e-05 & - & 6.0732e-03 & - & 2.6809e-03 & -  \\
			& 20 & 3.4849e-06 & 5.6170 & 2.0375e-06 & 5.5514 & 2.9494e-04 & 4.3639 & 1.4835e-04 & 4.1757   \\
			& 40 & 5.8177e-08 & 5.9045 & 3.4831e-08 & 5.8703 & 1.0087e-05 & 4.8699 & 5.4063e-06 & 4.7782   \\
			& 80 & 9.2631e-10 & 5.9728 & 5.5718e-10 & 5.9661 & 3.5353e-07 & 4.8345 & 1.9520e-07 & 4.7916   \\
			& 160 & 1.4537e-11 & 5.9937 & 8.7578e-12 & 5.9914 & 1.1270e-08 & 4.9713 & 6.1610e-09 & 4.9857   \\
			\hline
		\end{tabular}
	\end{center}
\end{table}

\begin{table}[!htb]
	\begin{center}
		\caption{Same as Table \ref{table:1DEC6ES5_EOS} except for EC4 and ES4.}\label{table:1DEC4ES4_EOS_3}
		\begin{tabular}{c|c||c|c|c|c||c|c|c|c}
			\hline
			\multirow{2}{*}{} & \multirow{2}{*}{$N$} & \multicolumn{4}{c||}{EC4} & \multicolumn{4}{c}{ES4}\\
			\cline{3-10}
			& & $l^1$ error & order & $l^2$ error & order  
			& $l^1$ error & order & $l^2$ error & order   \\
			\hline
			\multirow{5}{*}{SSP-RK3}
			& 10 & 1.2361e-03 & - & 5.7169e-04 & - & 6.4579e-03 & - & 3.3222e-03 & -   \\
			& 20 & 7.9981e-05 & 3.9500 & 3.8034e-05 & 3.9099 & 4.9999e-04 & 3.6911 & 3.2921e-04 & 3.3350   \\
			& 40 & 5.0424e-06 & 3.9875 & 2.4168e-06 & 3.9761 & 4.6449e-05 & 3.4282 & 3.2099e-05 & 3.3584  \\
			& 80 & 3.1588e-07 & 3.9967 & 1.5169e-07 & 3.9939 & 3.3745e-06 & 3.7829 & 2.7456e-06 & 3.5473   \\
			& 160 & 1.9754e-08 & 3.9991 & 9.4904e-09 & 3.9985 & 2.3313e-07 & 3.8555 & 2.3179e-07 & 3.5663     \\
			\hline
			\multirow{5}{*}{RRK3}
			& 10 & 1.2362e-03 & - & 5.7169e-04 & - & 6.4540e-03 & - & 3.3207e-03 & -   \\
			& 20 & 7.9983e-05 & 3.9500 & 3.8034e-05 & 3.9099 & 4.9997e-04 & 3.6903 & 3.2920e-04 & 3.3345   \\
			& 40 & 5.0425e-06 & 3.9875 & 2.4168e-06 & 3.9761 & 4.6448e-05 & 3.4281 & 3.2099e-05 & 3.3584  \\
			& 80 & 3.1587e-07 & 3.9967 & 1.5169e-07 & 3.9940 & 3.3745e-06 & 3.7829 & 2.7456e-06 & 3.5473   \\
			& 160 & 1.9754e-08 & 3.9991 & 9.4903e-09 & 3.9985 & 2.3313e-07 & 3.8554 & 2.3179e-07 & 3.5663     \\
			\hline
		\end{tabular}
	\end{center}
\end{table}

\begin{figure}[!htb]
	\centering
	\begin{subfigure}[t]{.48\linewidth}
		\centering
		\includegraphics[width=1\textwidth]{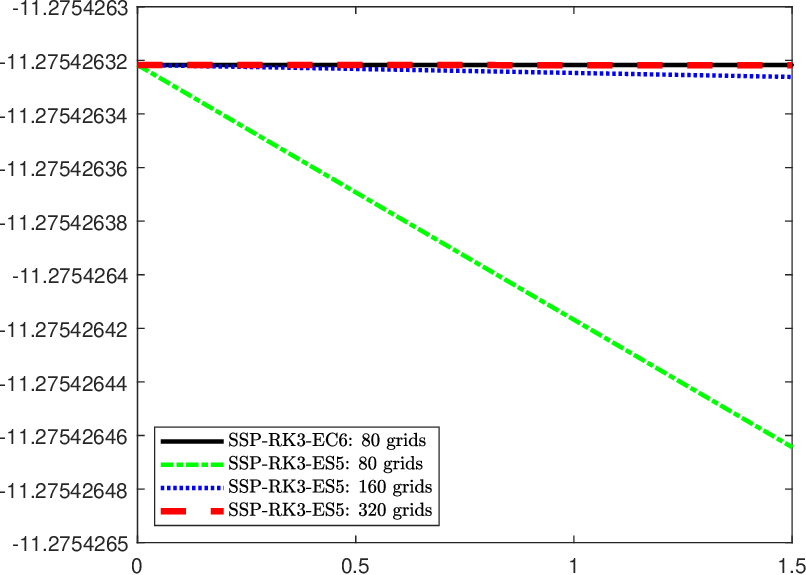}
		\caption{ SSP-RK3.}
	\end{subfigure}
	\begin{subfigure}[t]{.48\linewidth}
		\centering
		\includegraphics[width=1\textwidth]{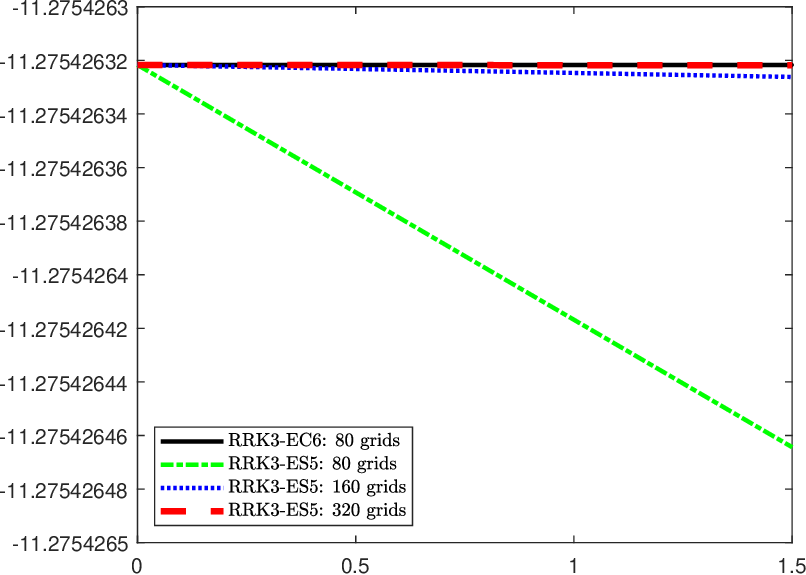}
		\caption{ RRK3.}
	\end{subfigure}
	\caption{Example \ref{Ex5.1.1}: Evolution of discrete total entropy, EC6 and ES5.}\label{Fig:evolutionDiscreteEntropy_EC6ES5}
\end{figure}

\begin{figure}[!htb]
	\centering
	\begin{subfigure}[t]{.48\linewidth}
		\centering
		\includegraphics[width=1\textwidth]{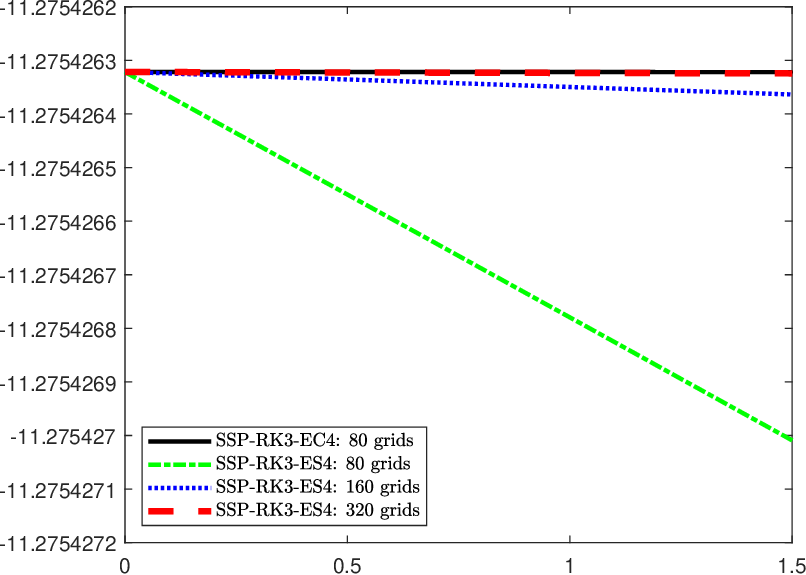}
		\caption{ SSP-RK3.}
	\end{subfigure}
	\begin{subfigure}[t]{.48\linewidth}
		\centering
		\includegraphics[width=1\textwidth]{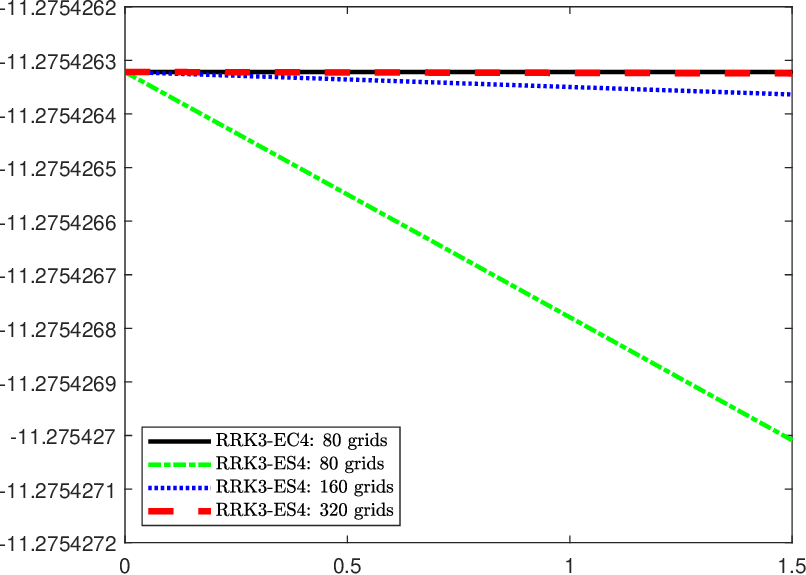}
		\caption{RRK3.}
	\end{subfigure}
	\caption{Example \ref{Ex5.1.1}: Evolution of discrete total entropy, EC4 and ES4.}\label{Fig:evolutionDiscreteEntropy_EC4ES4}
\end{figure}

\end{example}

\begin{example}[Relativistic isentropic problem]\label{REFEREE}
In this example. we investigate a truly isentropic problem \cite{radice2012thc} to study the temporal evolution of discrete total entropy. 
The initial conditions for the rest-mass density and pressure are given by
\begin{equation*}
\rho(x,0)=
\begin{cases}
    1+{\rm exp}\left[-1/(1-x^2/L^2)\right],&  |x| < L,\\
    1,& \text{otherwise},
\end{cases} \qquad p = K \rho^\Gamma. 
\end{equation*}
The initial velocity $v(x,0) = 0$ for $|x|\geq L$, while it is determined  for $|x|< L$ by  enforcing the following Riemann invariant constant  
\begin{equation*}
    J_- = \frac{1}{2}\ln\left(\frac{1+v}{1-v}\right)-\frac{1}{\sqrt{\Gamma-1}}\ln\left(\frac{\sqrt{\Gamma-1}+c_s}{\sqrt{\Gamma-1}-c_s}\right).
\end{equation*}

The computational domain is $[-0.4,2]$ with periodic boundary conditions. 
The parameters are consistent with those in \cite{radice2012thc}: $L = 0.3,\ \Gamma = \frac{5}{3},\ K=100$, and the CFL number is set to be $0.2$. 
Figure \ref{Fig:evolutionDiscreteEntropy_RF} presents the temporal evolution of the discrete total entropy, $\sum_i\eta(\mathbf{U}_i(t))\Delta x$, up to $t=0.8$ for the EC6 and ES5 schemes on different uniform mesh grids. 
One can observe that the discrete total entropy remains nearly constant for EC6 schemes, while for ES5 schemes, it decays slightly over time due to the inherent dissipation mechanism, as expected.

Furthermore, we conduct a comparative analysis between the ES/EC schemes and a fifth-order non-EC, non-ES scheme (termed non-ES5), which is constructed by adding a time-dependent sinusoidal term to the dissipation item into the EC flux: 
\begin{equation} \label{nonESFluxRF}
\widehat{\mathbf F}_{i+\frac{1}{2}}=\widetilde{\mathbf F}_{i+\frac{1}{2}}-{\frac{3}{5}\sin(50t)}\ {\mathbf D}_{i+\frac{1}{2}}\ [\![{\mathbf W}]\!]_{i+\frac{1}{2}}.
\end{equation}
For time discretization, we use the SSP-RK3 method in this comparison. The numerical results at $t=0.8$ computed with $200$ uniform grids are shown in Figure \ref{Fig:density_nonES5_RK3RF}. 
One can see that the non-ES5 solution exhibits spurious oscillations and is unable to accurately resolve the structure of the solution in comparison with EC6 and ES5.  The evolution of discrete total entropy is presented in Figure \ref{Fig:evolutionDiscreteEntropy_nonES5_RK3RF}, where we observe that the discrete total entropy generated by non-ES5 does not always diminish over time. This indicates that the non-ES5 scheme fails to satisfy the discrete entropy inequality \eqref{3.2}.

\end{example}

\begin{figure}[!htb]
	\centering
	\begin{subfigure}[t]{.48\linewidth}
		\centering
		\includegraphics[width=1\textwidth]{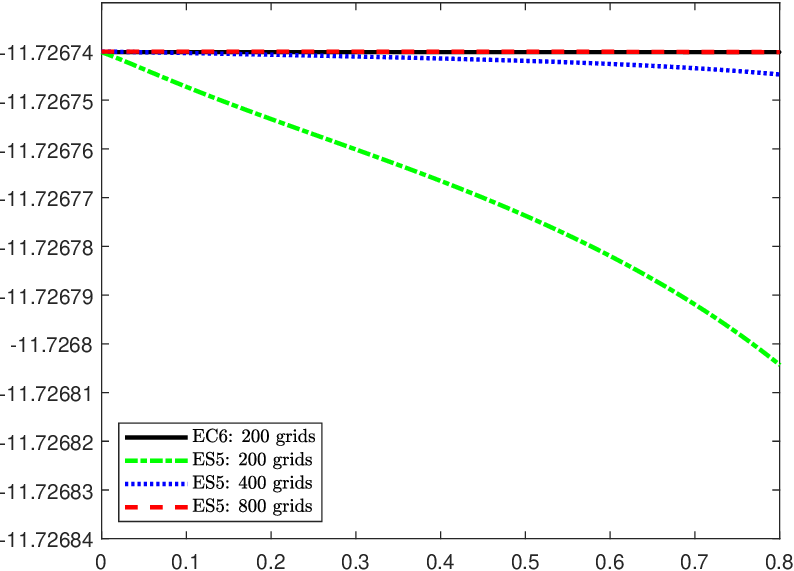}
		\caption{SSP-RK3}\label{Fig:evolutionDiscreteEntropy_RK3RF}
	\end{subfigure}
	\begin{subfigure}[t]{.48\linewidth}
		\centering
		\includegraphics[width=1\textwidth]{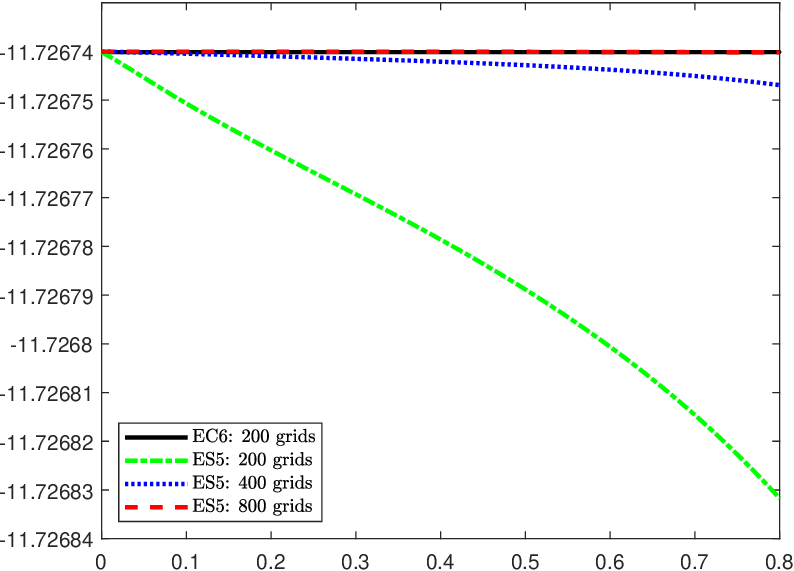}
		\caption{RRK3}\label{Fig:evolutionDiscreteEntropy_RRK3RF}
	\end{subfigure}
	\caption{Example \ref{REFEREE}: Evolution of discrete total entropy for EC and ES schemes with SSP-RK3 and RRK3 time discretization.}\label{Fig:evolutionDiscreteEntropy_RF}
\end{figure}

\begin{figure}[!htb]
	\centering
	\begin{subfigure}[t]{.48\linewidth}
		\centering
		\includegraphics[width=0.9\textwidth]{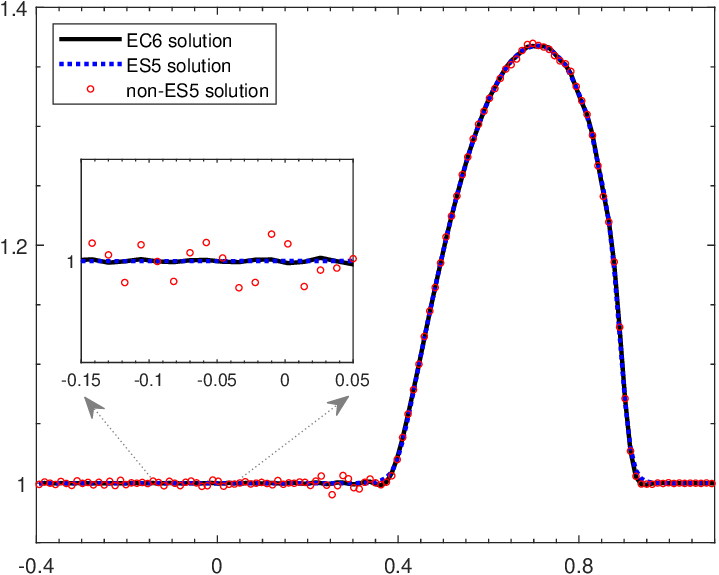}
		\caption{Numerical results of $\rho$ at $t = 0.8$.}\label{Fig:density_nonES5_RK3RF}
	\end{subfigure}
	\begin{subfigure}[t]{.48\linewidth}
		\centering
		\includegraphics[width=1\textwidth]{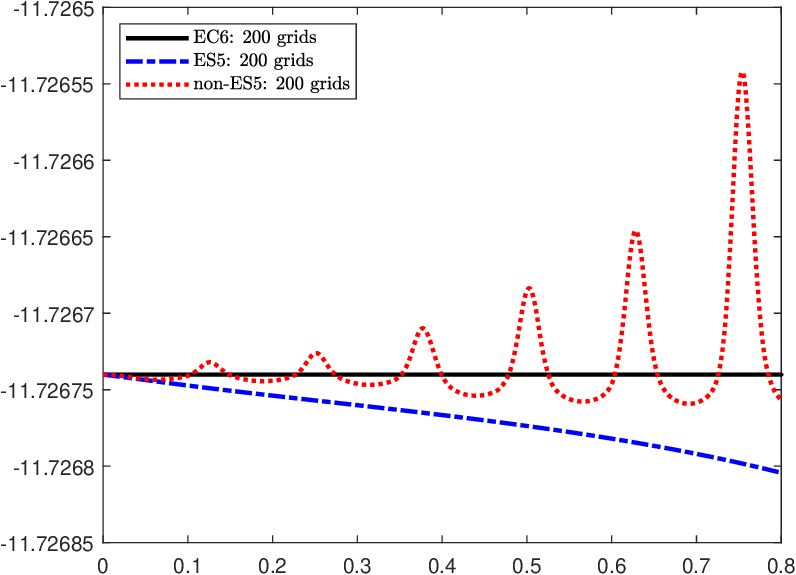}
		\caption{Evolution of discrete total entropy.} \label{Fig:evolutionDiscreteEntropy_nonES5_RK3RF}
	\end{subfigure}
	\caption{Example \ref{REFEREE}: Comparison between EC6, ES5, and  non-ES5.}\label{Fig:density_ESorNonES_RF}
\end{figure}

In the following, we test a density perturbation problem, a blast wave interaction problem, and four Riemann problems to validate the capability of our high-order ES schemes in resolving discontinuous solutions. 
Due to the complexity of obtaining the exact solutions for these problems with various EOSs, we  utilize the first-order local Lax--Friedrichs scheme with 100,000 uniform grids to produce reference solutions. The reference solutions are depicted with solid lines, while the numerical solutions with SSP-RK3 time discretization are indicated by circle markers ``${\color{blue}\circ}$''. For the examples using RRK3, the numerical results are denoted by square markers ``${\color{red}\square}$''.

\begin{example}[Density perturbation problem]\label{Ex5.1.6}
The initial conditions of this example are given by
\begin{equation*}
    \textbf{V}(x,0)=
    \begin{cases}
    (5,0,50)^\top,&  0\leq x<0.5,\\
    (2+0.3\sin(50x),0,5)^\top,& 0.5\leq x \leq 1,
    \end{cases}
\end{equation*}
which introduce a sine-type perturbation to the rest-mass density \cite{del2002efficient}. This problem models the interaction between a shock and a sine wave. We adopt the RC-EOS \eqref{hEOS1} for this test. The outflow boundary conditions are imposed on both the left and right boundaries of the domain $[0,1]$, by setting the data for all left (and similarly, right) ghost points to match the values of the nearest computational points. 
We numerically simulate this problem by using ES5 on 400 uniform grids up to time $t=0.376$. The computational results are shown in Figure \ref{Fig:1D_densityperturb_EOS_RRK}. We observe that ES5 resolves the small perturbation waves with high fidelity. To verify the ES property of our scheme, we compute the discrete total entropy $\sum_i\eta(\mathbf{U}_i(t))\Delta x$, which should decrease over time. The resulting plot is shown in Figure \ref{Fig:evolutionDiscreteEntropy_density}, and we observe the expected decay of the discrete total entropy, which confirms the ES property of the scheme.
\end{example}

\begin{figure}[!htb]
	\centering
	\begin{subfigure}[t]{.3\linewidth}
		\centering
		\includegraphics[width=0.99\textwidth]{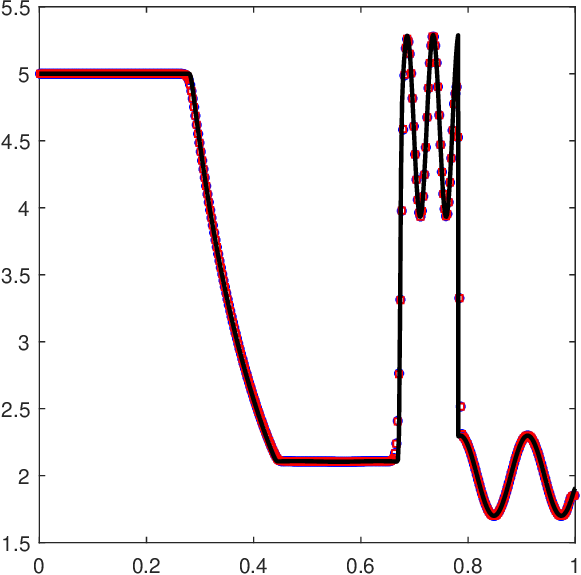}
		\caption{$\rho$}
	\end{subfigure}
	\begin{subfigure}[t]{.3\linewidth}
		\centering
		\includegraphics[width=1\textwidth]{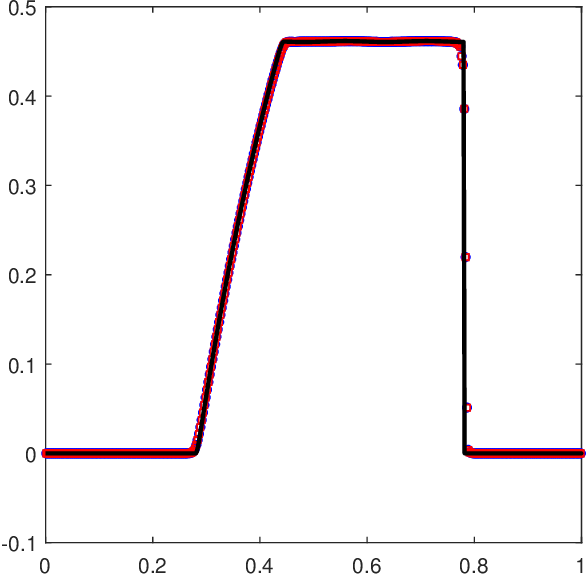}
		\caption{$v_1$}
	\end{subfigure}
	\begin{subfigure}[t]{.3\linewidth}
		\centering
		\includegraphics[width=0.98\textwidth]{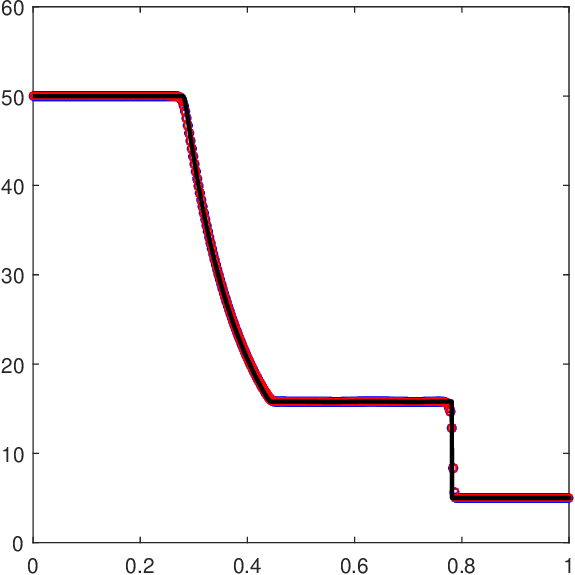}
		\caption{\color{blue}$p$}
	\end{subfigure}
	\caption{Example \ref{Ex5.1.6}: Numerical results obtained by ES5 with SSP-RK3 (circle markers ``${\color{blue}\circ}$'') and RRK3 (square markers ``${\color{red}\square}$'')  at $t=0.376$. RC-EOS \eqref{hEOS1} is used.}\label{Fig:1D_densityperturb_EOS_RRK}
\end{figure}

\begin{figure}[!htb]
	\centering
	\begin{subfigure}[t]{.48\linewidth}
		\centering
		\includegraphics[width=0.95\textwidth]{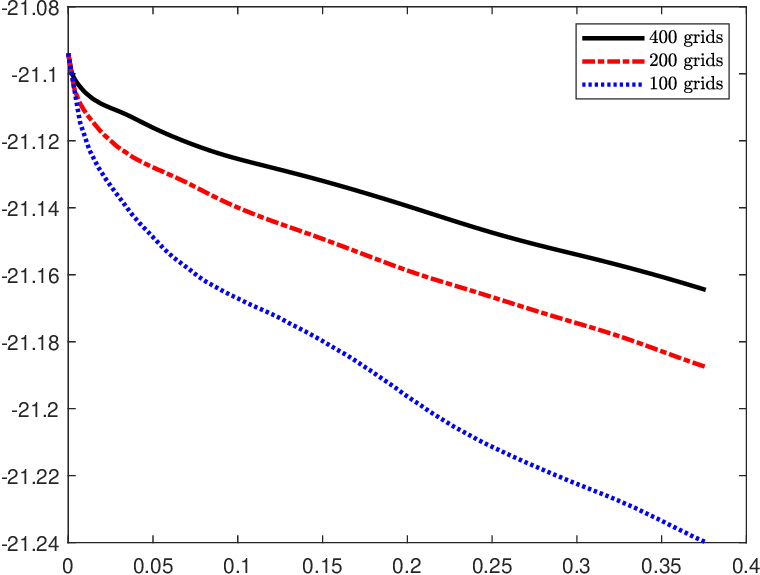}
	\end{subfigure}
	\begin{subfigure}[t]{.48\linewidth}
		\centering
		\includegraphics[width=0.95\textwidth]{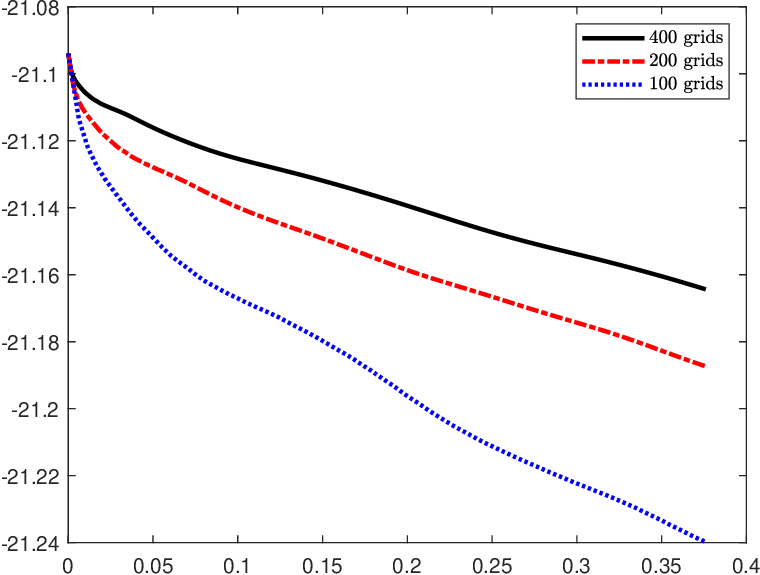}
	\end{subfigure}
	\caption{Evolution of discrete total entropy for Example \ref{Ex5.1.3}: SSP-RK3 (left) and RRK3 (right).}\label{Fig:evolutionDiscreteEntropy_density}
\end{figure}

\begin{example}[Blast wave interaction]\label{Ex5.1.7}
The final 1D example simulates the interaction between two strong relativistic blast waves in the domain $[0,1]$. 
The initial conditions are defined as follows: 
\begin{equation*}
    \textbf{V}(x,0)=
    \begin{cases}
    (1,0,10^3)^\top,& 0\leq x<0.1,\\
    (1,0,10^{-2})^\top,& 0.1\leq x < 0.9,\\
    (1,0,10^2)^\top,& 0.9\leq x \leq 1,
    \end{cases}
\end{equation*} 
We employ TM-EOS \eqref{hEOS3} as our EOS, and apply outflow boundary conditions at both left and right boundaries. At $t=0.43$, the solutions form a complex wave structure, which encompasses three contact discontinuities and two shock waves within the interval $[0.5,0.53]$. We compute the reference solution via a first-order local Lax--Friedrichs scheme, using an ultra-fine mesh of $200,000$ uniform cells. The numerical solutions obtained by ES5 on 4000 uniform grids are presented in Figure \ref{Fig:1D_blast_EOS_RRK}, where we observe good agreement between the numerical and reference solutions. 
Furthermore, we compute the discrete total entropy, which is shown in Figure \ref{Fig:evolutionDiscreteEntropy_blast}. We can observe that the discrete total entropy is decreasing over time, indicating that the fully discrete scheme is ES.
\end{example}

\begin{figure}[!htb]
	\centering
	\begin{subfigure}[t]{.3\linewidth}
		\centering
		\includegraphics[width=1\textwidth]{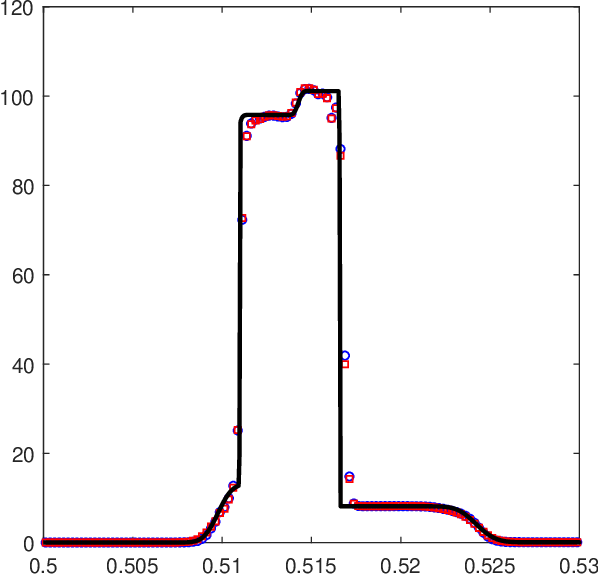}
		\caption{\color{blue}$\rho$}
		\label{Fig:1D_blast_EOS_density}
	\end{subfigure}
	\begin{subfigure}[t]{.3\linewidth}
		\centering
		\includegraphics[width=1\textwidth]{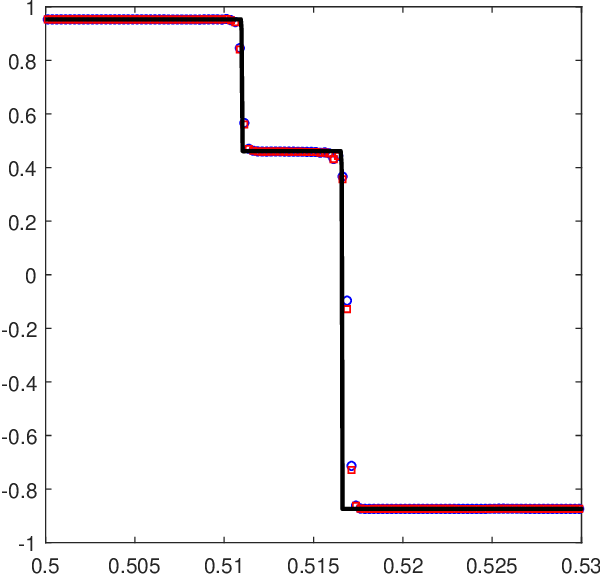}
		\caption{\color{blue}$v_1$}
		\label{Fig:1D_blast_EOS_velocity}
	\end{subfigure}
	\begin{subfigure}[t]{.3\linewidth}
		\centering
		\includegraphics[width=1\textwidth]{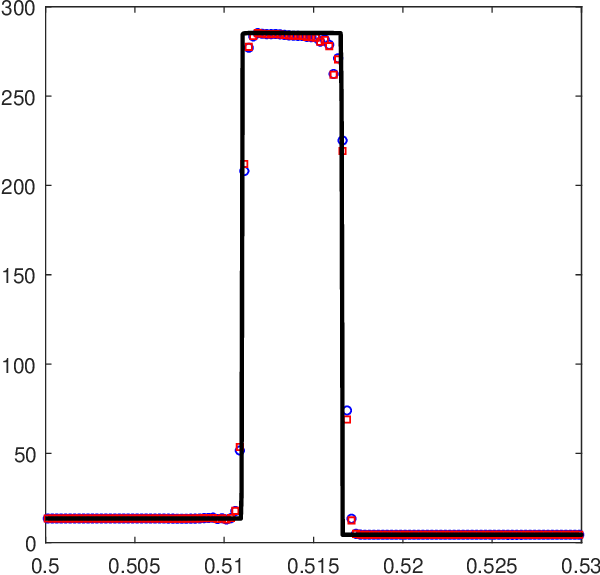}
		\caption{\color{blue}$p$}
		\label{Fig:1D_blast_EOS_pressure}
	\end{subfigure}
	\caption{Example \ref{Ex5.1.7}: Numerical results obtained by ES5 with SSP-RK3 (circle markers ``${\color{blue}\circ}$'') and RRK3 (square markers ``${\color{red}\square}$'') at $t=0.43$. TM-EOS \eqref{hEOS3} is used.}\label{Fig:1D_blast_EOS_RRK}
\end{figure}

\begin{figure}[!htb]
	\centering
	\begin{subfigure}[t]{.48\linewidth}
		\centering
		\includegraphics[width=0.95\textwidth]{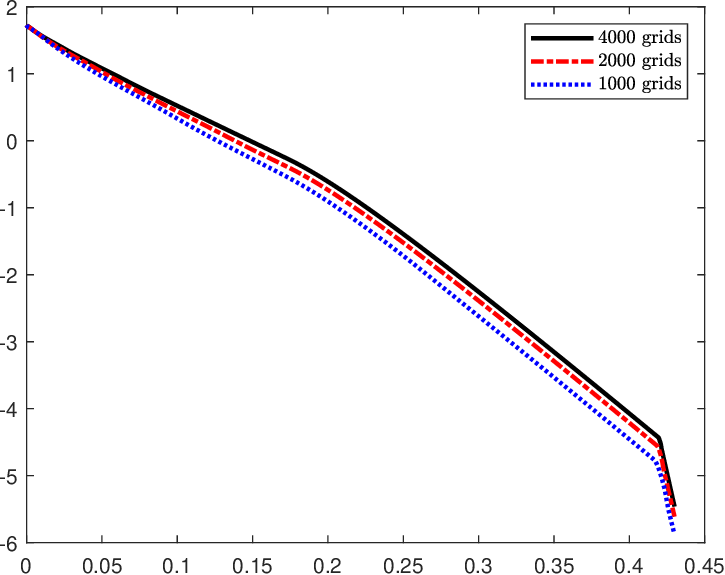}
	\end{subfigure}
	\begin{subfigure}[t]{.48\linewidth}
		\centering
		\includegraphics[width=0.95\textwidth]{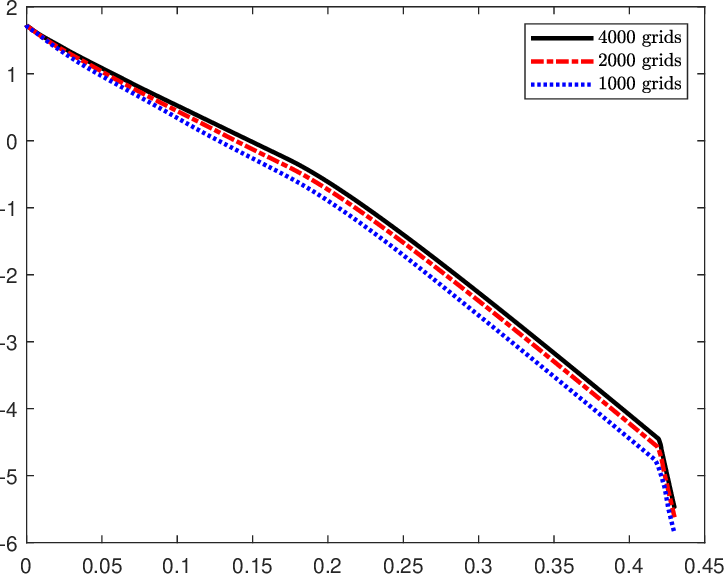}
	\end{subfigure}
	\caption{Evolution of discrete total entropy for Example \ref{Ex5.1.3}: SSP-RK3 (left) and RRK3 (right).}\label{Fig:evolutionDiscreteEntropy_blast}
\end{figure}

\begin{example}[1D Riemann problem I]\label{Ex5.1.2}
The initial conditions are taken as 
\begin{equation*}
\textbf{V}(x,0)=
\begin{cases}
    (10,0,\frac{40}{3})^\top,&  0\leq x<0.5,\\
    (1,0,10^{-6})^\top,& 0.5\leq x \leq 1.
\end{cases}
\end{equation*}
The computational domain $[0,1]$ is divided into $400$ uniform cells, and we use the outflow boundary conditions. For this example, we use the RC-EOS \eqref{hEOS1} and set the CFL number as $0.1$.  The exact solution contains a left-moving rarefaction wave, a contact discontinuity, and a right-moving shock over time.   
The numerical solutions at $t=0.4$ obtained by ES5 are shown in Figure \ref{Fig:1D_RP1_EOS_RRK}. We can see that the computed solutions agree with the reference solution, and the wave structures are well captured by ES5. To verify the ES property, we examine the evolution of discrete total entropy, as shown in Figure \ref{Fig:evolutionDiscreteEntropy_RP1}. We observe that the total numerical entropy $\sum_i\eta(\mathbf{U}_i(t))\Delta x$ decreases over time, as expected. 
\end{example}

\begin{figure}[!htb]
	\centering
	\begin{subfigure}[t]{.3\linewidth}
		\centering
		\includegraphics[width=0.975\textwidth]{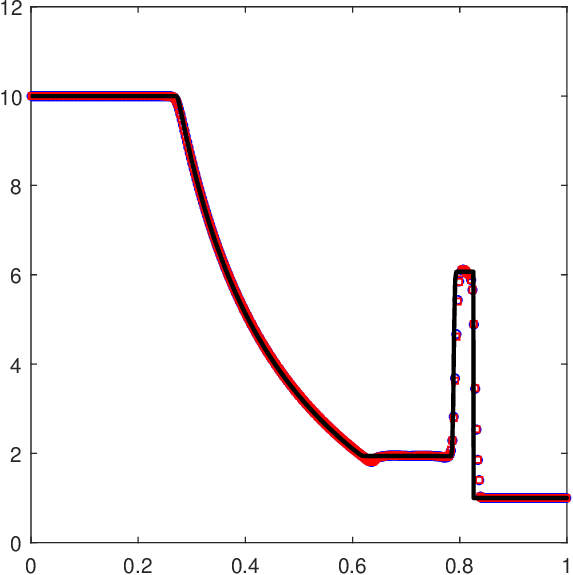}
		\caption{$\rho$}
		\label{Fig:1D_RP1_EOS_density_RRK}
	\end{subfigure}
	\begin{subfigure}[t]{.3\linewidth}
		\centering
		\includegraphics[width=1\textwidth]{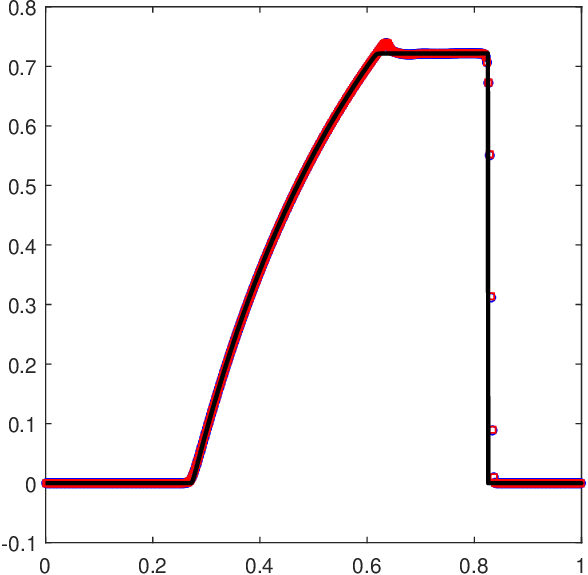}
		\caption{$v_1$}
		\label{Fig:1D_RP1_EOS_velocity_RRK}
	\end{subfigure}
	\begin{subfigure}[t]{.3\linewidth}
		\centering
		\includegraphics[width=0.975\textwidth]{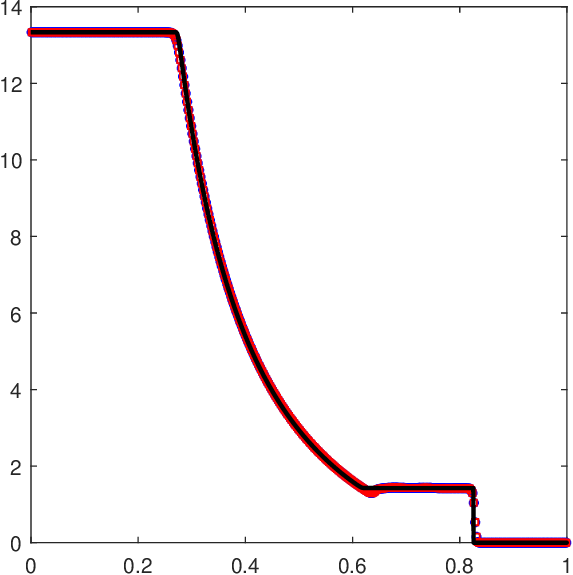}
		\caption{$p$}
		\label{Fig:1D_RP1_EOS_pressure_RRK}
	\end{subfigure}
	\caption{Example \ref{Ex5.1.2}: Numerical results obtained by ES5 with SSP-RK3 (circle markers ``${\color{blue}\circ}$'') and RRK3 (square markers ``${\color{red}\square}$'') at $t=0.4$, RC-EOS \eqref{hEOS1}.}\label{Fig:1D_RP1_EOS_RRK}
\end{figure}

\begin{figure}[!htb]
	\centering
	\begin{subfigure}[t]{.48\linewidth}
		\centering
		\includegraphics[width=0.95\textwidth]{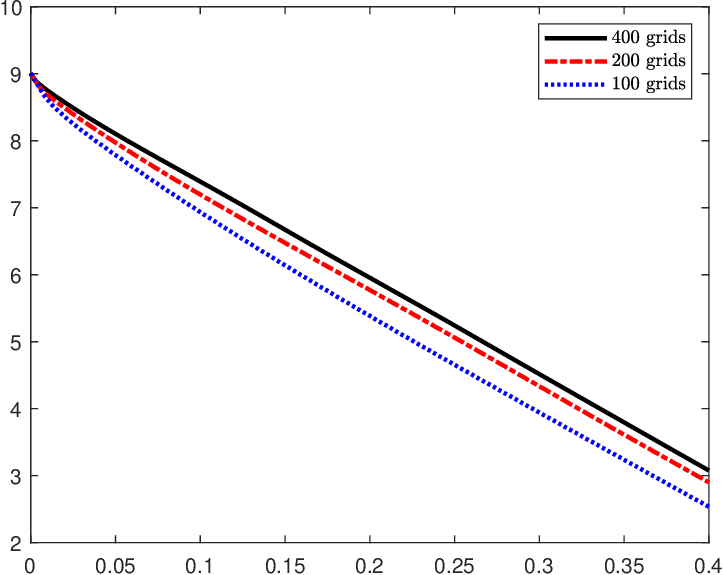}
	\end{subfigure}
	\begin{subfigure}[t]{.48\linewidth}
		\centering
		\includegraphics[width=0.95\textwidth]{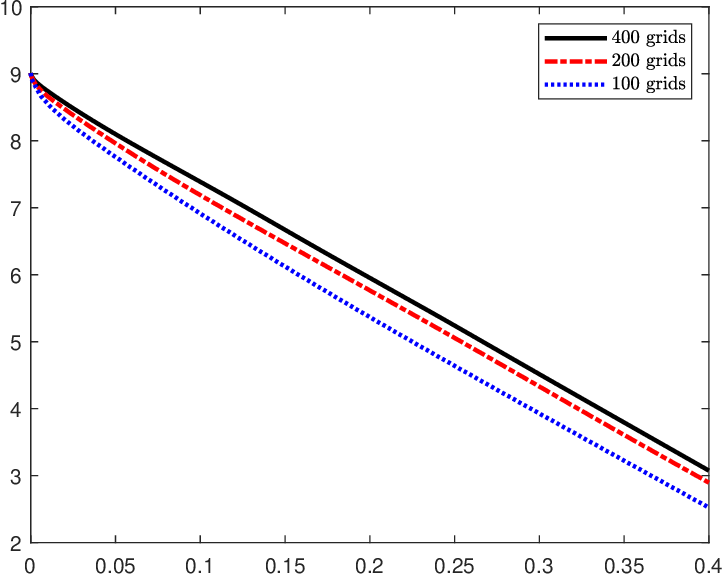}
	\end{subfigure}
	\caption{Evolution of discrete total entropy for Example \ref{Ex5.1.2}: SSP-RK3 (left) and RRK3 (right).}\label{Fig:evolutionDiscreteEntropy_RP1}
\end{figure}

\begin{example}[1D Riemann problem II]\label{Ex5.1.3}
The wave structure of this problem is akin to that of Example \ref{Ex5.1.2}. 
However, the region between the contact discontinuity and the right-moving shock is significantly narrow, which poses a challenge to the simulation. 
The initial conditions of the problem are as follows:
\begin{equation*}
    \textbf{V}(x,0)=
    \begin{cases}
    (1,0,10^{3})^\top,&  0\leq x<0.5,\\
    (1,0,10^{-2})^\top,&  0.5\leq x \leq 1.
    \end{cases}
\end{equation*}
We adopt the TM-EOS \eqref{hEOS3} for this test case and use the outflow boundary conditions. Figure \ref{Fig:1D_RP2_EOS_RRK} shows the numerical solutions of the rest-mass density $\rho$, the velocity $v_1$, and the pressure $p$ obtained by ES5 with 400 uniform grids at $t=0.4$.  The results demonstrate that our ES5 method can effectively resolve wave structures without producing any significant oscillations. The evolution of the discrete total entropy is depicted in Figure \ref{Fig:evolutionDiscreteEntropy_RP2}, which reveals a dissipated entropy. This verifies the entropy stability of the fully discrete scheme.

\end{example}

\begin{figure}[!htb]
	\centering
	\begin{subfigure}[t]{.3\linewidth}
		\centering
		\includegraphics[width=0.96\textwidth]{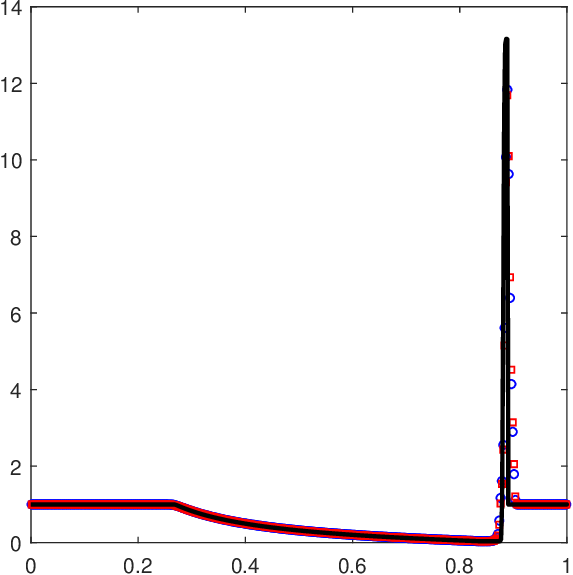}
		\caption{$\rho$}
		\label{Fig:1D_RP2_EOS_density_RRK}
	\end{subfigure}
	\begin{subfigure}[t]{.3\linewidth}
		\centering
		\includegraphics[width=0.97\textwidth]{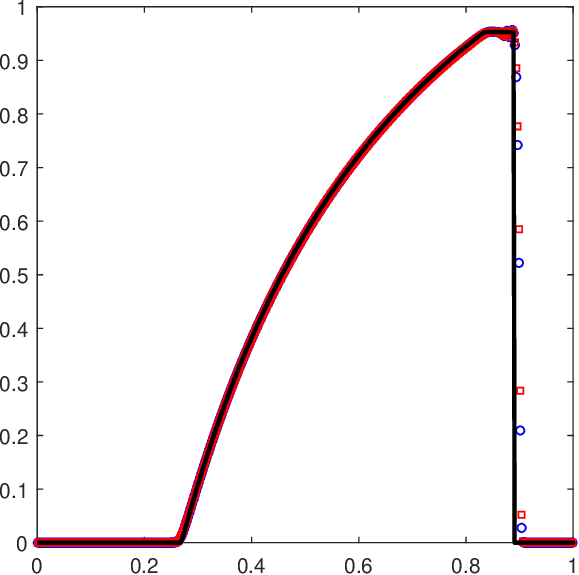}
		\caption{$v_1$}
		\label{Fig:1D_RP2_EOS_velocity_RRK}
	\end{subfigure}
	\begin{subfigure}[t]{.3\linewidth}
		\centering
		\includegraphics[width=1\textwidth]{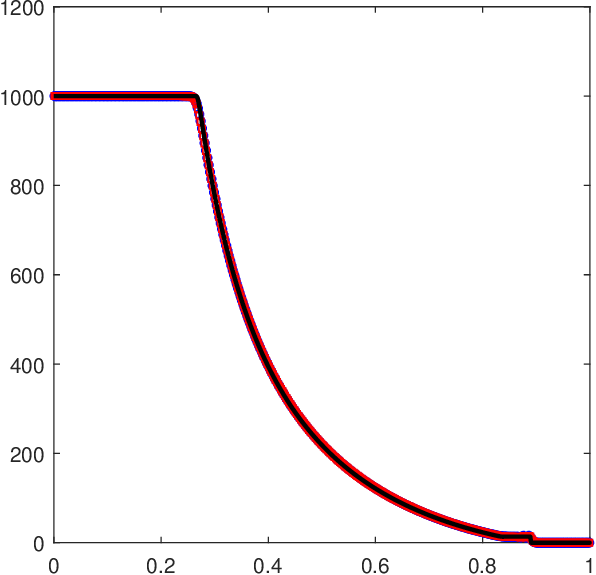}
		\caption{$p$}
		\label{Fig:1D_RP2_EOS_pressure_RRK}
	\end{subfigure}
	\caption{Example \ref{Ex5.1.3}: Numerical results obtained by ES5 with SSP-RK3 (circle markers ``${\color{blue}\circ}$'') and RRK3 (square markers ``${\color{red}\square}$'') at $t=0.4$. TM-EOS \eqref{hEOS3} is used.}\label{Fig:1D_RP2_EOS_RRK}
\end{figure}

\begin{figure}[!htb]
	\centering
	\begin{subfigure}[t]{.48\linewidth}
		\centering
		\includegraphics[width=0.95\textwidth]{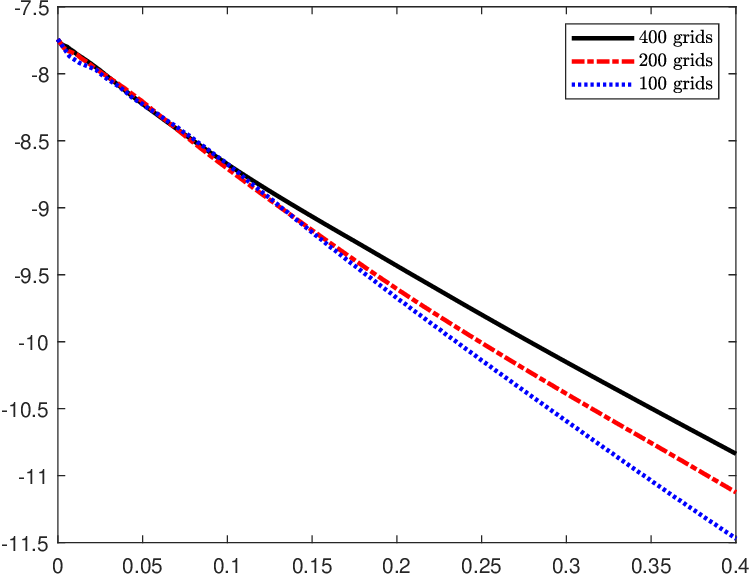}
	\end{subfigure}
	\begin{subfigure}[t]{.48\linewidth}
		\centering
		\includegraphics[width=0.95\textwidth]{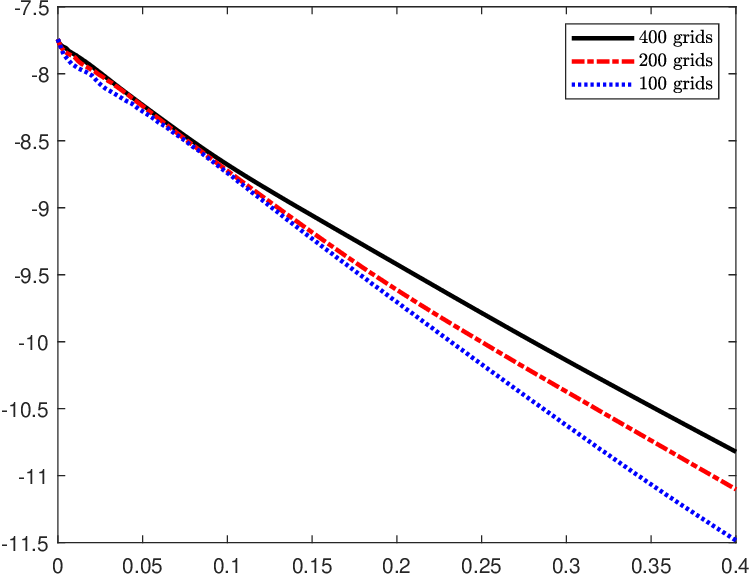}
	\end{subfigure}
	\caption{Evolution of discrete total entropy for Example \ref{Ex5.1.3}: SSP-RK3 (left) and RRK3 (right).}\label{Fig:evolutionDiscreteEntropy_RP2}
\end{figure}

	As observed from the above examples, the numerical results with SSP-RK3 and RRK3 time discretization are very close for both smooth and discontinuous problems. 
	 To save space, in the following, we will only present the numerical results with the classic SSP-RK3 time discretization.

\begin{example}[1D Riemann problem III]\label{Ex5.1.4}
The initial conditions of this Riemann problem are
\begin{equation*}
    \textbf{V}(x,0)=
    \begin{cases}
    (1,0.9,1)^\top,&  0\leq x<0.5,\\
    (1,0,10)^\top,&  0.5\leq x \leq 1.
    \end{cases}
\end{equation*}
This test models a contact discontinuity and two shock waves that propagate in opposite directions. The spatial domain $[0,1]$ is divided into 400 uniform cells, with outflow boundary conditions. The RC-EOS \eqref{hEOS1} is used for this example, and the outflow boundary conditions are employed.  
We compare the ES5 scheme and a non-EC, non-ES scheme (termed non-ES5), which is constructed by augmenting the dissipative term in the numerical flux \eqref{3.1} with a sinusoidal term: 
\begin{equation} \label{nonESFluxRP3}
\widehat{\mathbf F}_{i+\frac{1}{2}}=\widetilde{\mathbf F}_{i+\frac{1}{2}}-{\frac{6}{5}\sin(7.6t+0.1)}\ \mathbf{D}_{i+\frac{1}{2}}\ [\![\mathbf{W}]\!]_{i+\frac{1}{2}}.
\end{equation}
For comparison purpose, the computational settings of non-ES5 are the same as ES5. 
The numerical solutions obtained by using ES5 and non-ES5 at $t=0.4$ are shown in Figure \ref{Fig:1D_RP3_EOS_nonES}, where the ES5 solution is indicated by circle markers ``${\color{blue}\circ}$'' and the non-ES solution is denoted by square markers ``${\color{red}\square}$''. 
One can see that 
the ES5 solution is in good agreement with the reference one, while the non-ES5 solution exhibits significant nonphysical oscillations.

\end{example}

\begin{figure}[!htb]
	\centering
	\begin{subfigure}[t]{.3\linewidth}
		\centering
		\includegraphics[width=0.985\textwidth]{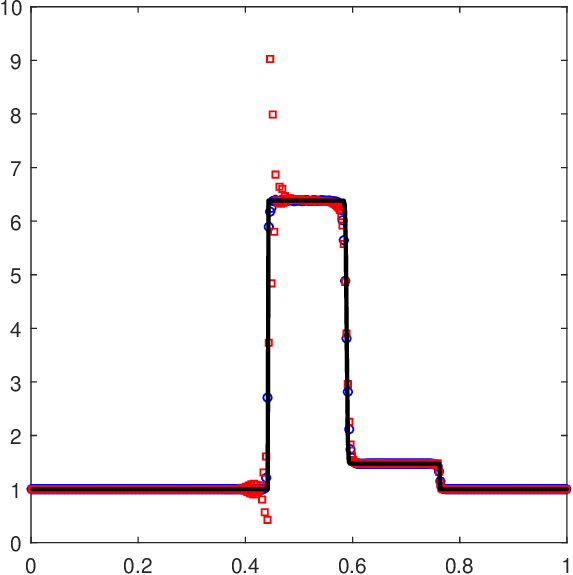}
		\caption{$\rho$}
	\end{subfigure}
	\begin{subfigure}[t]{.3\linewidth}
		\centering
		\includegraphics[width=1\textwidth]{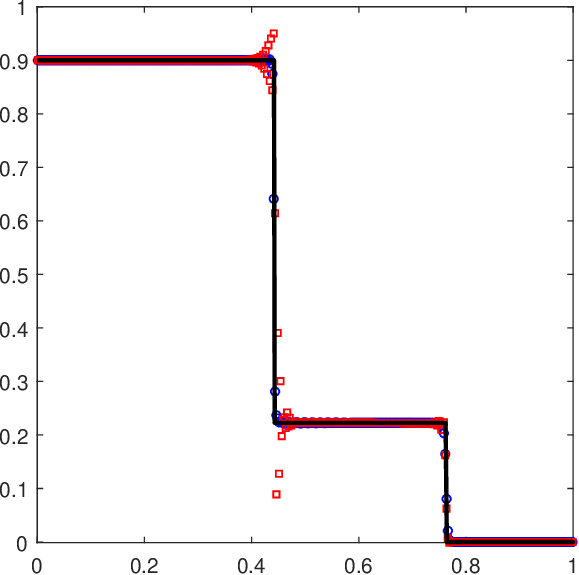}
		\caption{$v_1$}
	\end{subfigure}
	\begin{subfigure}[t]{.3\linewidth}
		\centering
		\includegraphics[width=0.985\textwidth]{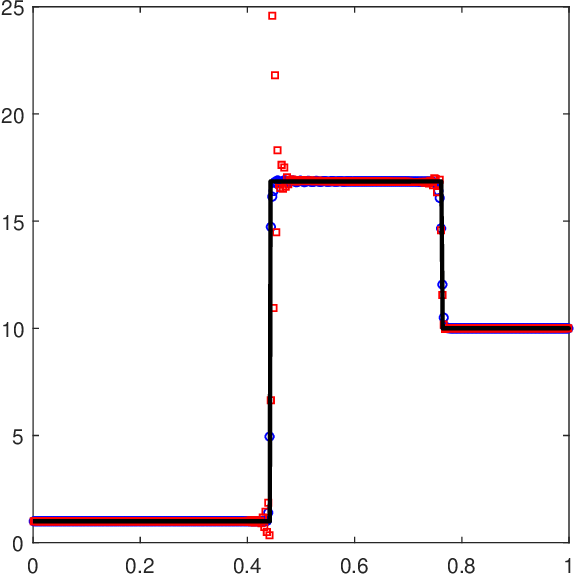}
		\caption{$p$}
	\end{subfigure}
	\caption{Example \ref{Ex5.1.4}: Numerical results obtained by ES5 (symbols ``${\color{blue}\circ}$'') and non-ES5 (symbols ``${\color{red}\square}$'') schemes at $t=0.4$. RC-EOS \eqref{hEOS1} is used.}\label{Fig:1D_RP3_EOS_nonES}
\end{figure}

\begin{example}[1D Riemann problem IV]\label{Ex5.1.5}
For this test, we consider the following initial data
\begin{equation*}
    \textbf{V}(x,0)=
    \begin{cases}
    (1,-0.7,20)^\top,& 0\leq x<0.5,\\
    (1,0.7,20)^\top,& 0.5\leq x \leq 1.
    \end{cases}
\end{equation*}
We choose the IP-EOS \eqref{hEOS2} for this problem. 
The solution consists of a contact discontinuity, and two rarefaction waves moving left and right, respectively. 
Figure \ref{Fig:1D_RP4_EOS} gives the numerical solutions at $t=0.4$ obtained by using ES5 on 400 uniform grids. The wave patterns of the numerical solutions are consistent with the reference ones, but there exists a undershoot for the rest-mass density $\rho$ at $x=0.5$. This phenomenon was also observed in the results obtained using the ID-EOS \eqref{ID-EOS} as reported in \cite{duan2019high}.

\end{example}

\begin{figure}[!htb]
	\centering
	\begin{subfigure}[t]{.3\linewidth}
		\centering
		\includegraphics[width=0.985\textwidth]{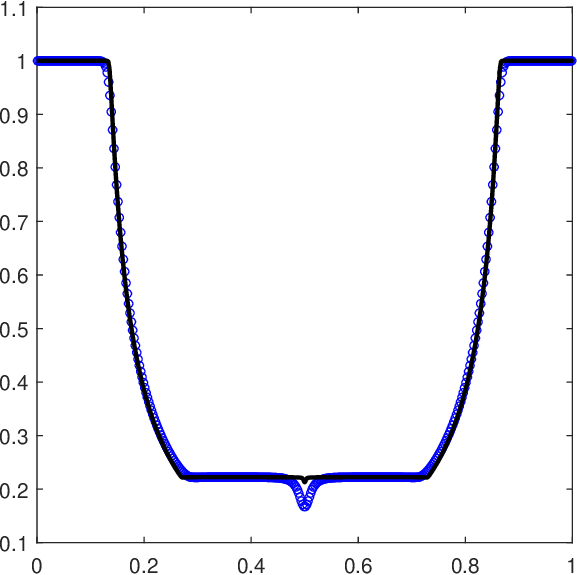}
		\caption{$\rho$}
		\label{Fig:1D_RP4_EOS_density}
	\end{subfigure}
	\begin{subfigure}[t]{.3\linewidth}
		\centering
		\includegraphics[width=1\textwidth]{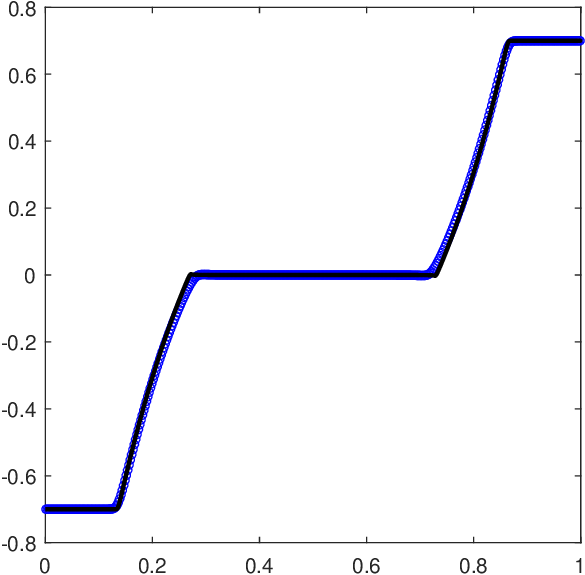}
		\caption{$v_1$}
		\label{Fig:1D_RP4_EOS_velocity}
	\end{subfigure}
	\begin{subfigure}[t]{.3\linewidth}
		\centering
		\includegraphics[width=0.985\textwidth]{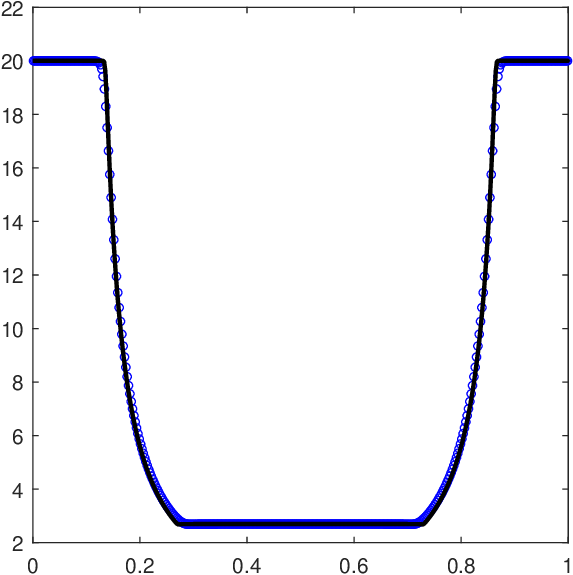}
		\caption{$p$}
		\label{Fig:1D_RP4_EOS_pressure}
	\end{subfigure}
	\caption{Example \ref{Ex5.1.5}: Numerical results obtained by ES5 at $t=0.4$. IP-EOS \eqref{hEOS2} is used.}\label{Fig:1D_RP4_EOS}
\end{figure}

\subsection{Two-dimensional examples}
\begin{example}[Accuracy test]\label{Ex5.2.1}
This example investigates a 2D smooth problem with periodic boundary conditions in the domain $[0,2\pi]^2$. The initial conditions are given by 
$$
\textbf{V}(x,y,0) = \left(1+0.2{\rm sin}(x+y),0.2,0.2,1\right)^\top.
$$
We consider the three EOSs \eqref{hEOS1}, \eqref{hEOS2}, and \eqref{hEOS3}, repetitively. 
The exact solution is  
\begin{equation*}
    \textbf{V}(x,y,t) = \left(1+0.2{\rm sin}\left(x+y-0.4t\right),0.2,0.2,1\right)^\top.
\end{equation*}
To investigate the spatial accuracy, we set the mesh size as $\Delta x=\Delta y = \frac{2\pi}{N}$ with  varying  $N \in \{10,20,40,80,160\}$.
The time step-size is chosen to match the spatial accuracy: $\Delta t = 0.4\Delta x^2$ for EC6 and $\Delta t = 0.4\Delta x^\frac{5}{3}$ for ES5. We report the $l^1$ and $l^2$ errors at $t = 0.1$ for the rest-mass density and corresponding orders in Tables \ref{table:2DEC6ES5_EOS_1}, \ref{table:2DEC6ES5_EOS_2}, and \ref{table:2DEC6ES5_EOS_3} for RC-EOS \eqref{hEOS1}, IP-EOS \eqref{hEOS2}, and TM-EOS \eqref{hEOS3}, respectively. We observe the expected convergence rates for ES5 and EC6. 
 In addition, we analyze the evolution of the discrete total entropy $\sum_{i,j}\eta(\mathbf{U}_{i,j}(t))\Delta x\Delta y$, as shown in Figure \ref{Fig:2DevolutionDiscreteEntropy}. The results indicate that the total numerical entropy decreases over time for ES5, while it almost remains constant for EC6, as expected.
\end{example}

\begin{table}[!htb]
	\begin{center}
		\caption{Numerical errors and convergence rates for EC6 and ES5 at different grid resolutions. RC-EOS \eqref{hEOS1} is used.}\label{table:2DEC6ES5_EOS_1}
		\begin{tabular}{c||c|c|c|c||c|c|c|c} 
			\hline
			\multirow{2}{*}{N} & \multicolumn{4}{c||}{EC6} & \multicolumn{4}{c}{ES5}\\
			\cline{2-9}
			& $l^1$ error & order & $l^2$ error & order  
			& $l^1$ error & order & $l^2$ error & order   \\
			\hline
			$10\times 10$ & 1.3460e-04 & - & 3.0063e-05 & - & 5.9926e-03 & - & 1.1046e-03 & -   \\
			$20\times 20$ & 2.6420e-06 & 5.6709 & 6.2944e-07 & 5.5778 & 2.4259e-04 & 4.6266 & 4.9739e-05 & 4.4730    \\
			$40\times 40$ & 4.4528e-08 & 5.8908 & 1.0725e-08 & 5.8750 & 7.7940e-06 & 4.9600 & 1.8410e-06 & 4.7558    \\
			$80\times 80$ & 7.0904e-10 & 5.9727 & 1.7143e-10 & 5.9672 & 3.3540e-07 & 4.5384 & 7.4067e-08 & 4.6355   \\
			$160\times 160$ & 1.1383e-11 & 5.9609 & 2.7142e-12 & 5.9810 & 1.1340e-08 & 4.8863 & 2.4839e-09 & 4.8982    \\
			\hline
		\end{tabular}
	\end{center}
\end{table}

\begin{table}[!htb]
	\begin{center}
		\caption{Numerical errors and convergence rates for EC6 and ES5 at different grid resolutions. IP-EOS \eqref{hEOS2} is used.}\label{table:2DEC6ES5_EOS_2}
		\begin{tabular}{c||c|c|c|c||c|c|c|c} 
			\hline
			\multirow{2}{*}{N} & \multicolumn{4}{c||}{EC6} & \multicolumn{4}{c}{ES5}\\
			\cline{2-9}
			& $l^1$ error & order & $l^2$ error & order  
			& $l^1$ error & order & $l^2$ error & order   \\
			\hline
			$10\times 10$ & 1.3452e-04 & - & 3.0053e-05 & - & 6.3843e-03 & - & 1.1773e-03 & -   \\
			$20\times 20$ & 2.6397e-06 & 5.6713 & 6.2929e-07 & 5.5777 & 2.6453e-04 & 4.5930 & 5.3620e-05 & 4.4565    \\
			$40\times 40$ & 4.4492e-08 & 5.8907 & 1.0723e-08 & 5.8750 & 8.3460e-06 & 4.9862 & 1.9423e-06 & 4.7869    \\
			$80\times 80$ & 7.0851e-10 & 5.9726 & 1.7140e-10 & 5.9672 & 3.5366e-07 & 4.5607 & 7.8162e-08 & 4.6352   \\
			$160\times 160$ & 1.1415e-11 & 5.9558 & 2.7108e-12 & 5.9825 & 1.1903e-08 & 4.8929 & 2.6101e-09 & 4.9043     \\
			\hline
		\end{tabular}
	\end{center}
\end{table}

\begin{table}[!htb]
	\begin{center}
		\caption{Numerical errors and convergence rates for EC6 and ES5 at different grid resolutions. TM-EOS \eqref{hEOS3} is used.}\label{table:2DEC6ES5_EOS_3}
		\begin{tabular}{c||c|c|c|c||c|c|c|c}
			\hline
			\multirow{2}{*}{N} & \multicolumn{4}{c||}{EC6} & \multicolumn{4}{c}{ES5}\\
			\cline{2-9}
			& $l^1$ error & order & $l^2$ error & order  
			& $l^1$ error & order & $l^2$ error & order   \\
			\hline
			$10\times 10$ & 1.3449e-04 & - & 3.0044e-05 & - & 6.1645e-03 & - & 1.1366e-03 & -   \\
			$20\times 20$ & 2.6391e-06 & 5.6714 & 6.2907e-07 & 5.0777 & 2.6220e-04 & 4.5552 & 5.2988e-05 & 4.4229    \\
			$40\times 40$ & 4.4477e-08 & 5.8908 & 1.0719e-08 & 5.8750 & 8.2779e-06 & 4.9853 & 1.9143e-06 & 4.7908    \\
			$80\times 80$ & 7.0826e-10 & 5.9726 & 1.7133e-10 & 5.9672 & 3.4381e-07 & 4.5896 & 7.6007e-08 & 4.6545   \\
			$160\times 160$ & 1.1414e-11 & 5.9554 & 2.7108e-12 & 5.9820 & 1.1621e-08 & 4.8869 & 2.5518e-09 & 4.8965     \\
			\hline
		\end{tabular}
	\end{center}
\end{table}

\begin{figure}[!htb]
	\centering
	\begin{subfigure}[t]{.32\linewidth}
		\centering
		\includegraphics[width=1\textwidth]{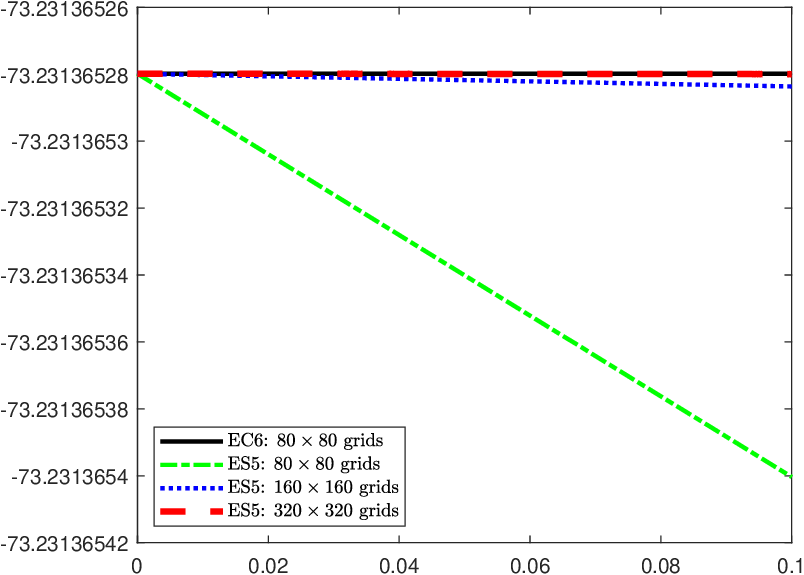}
		\caption{ RC-EOS \eqref{hEOS1}.}\label{Fig:2DevolutionDiscreteEntropy_RC}
	\end{subfigure}
	\begin{subfigure}[t]{.32\linewidth}
		\centering
		\includegraphics[width=1\textwidth]{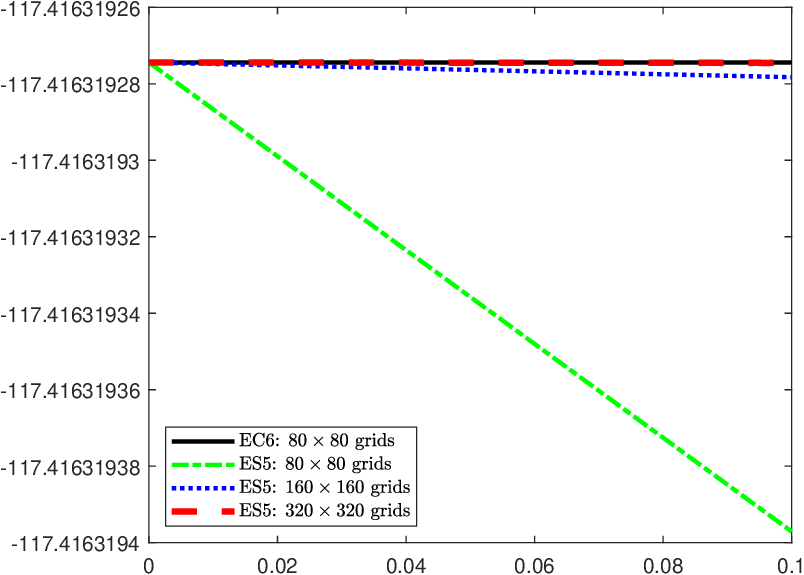}
		\caption{ IP-EOS \eqref{hEOS2}.}\label{Fig:2DevolutionDiscreteEntropy_IP}
	\end{subfigure}
	\begin{subfigure}[t]{.32\linewidth}
		\centering
		\includegraphics[width=1\textwidth]{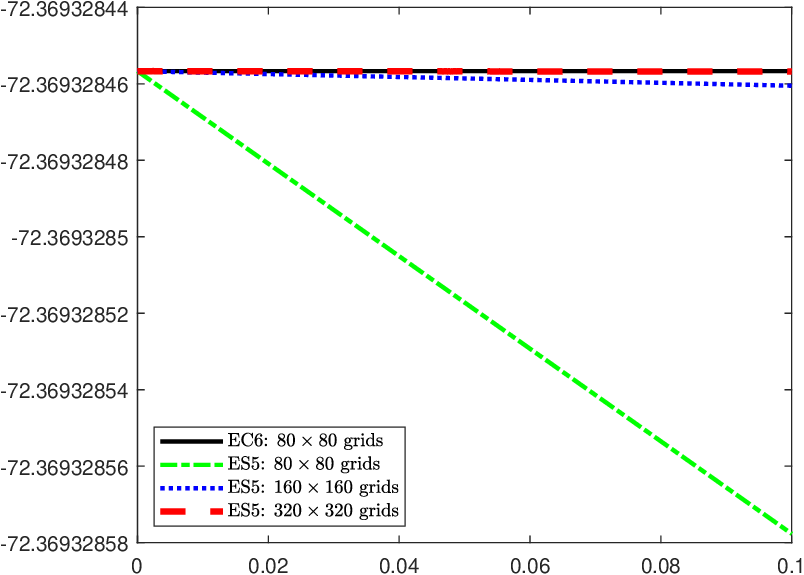}
		\caption{ TM-EOS \eqref{hEOS3}.}\label{Fig:2DevolutionDiscreteEntropy_TM}
	\end{subfigure}
	\caption{Example \ref{Ex5.2.1}: Evolution of discrete total entropy with different EOSs.}\label{Fig:2DevolutionDiscreteEntropy}
\end{figure}

In the following, we solve three 2D Riemann problems of RHD, which were proposed and studied in \cite{he2012adaptive1} with ID-EOS \eqref{ID-EOS}. We perform the same tests but with different EOSs.

\begin{example}[2D Riemann problem I]\label{Ex5.2.2}
The initial conditions of the first 2D Riemann problem are given by 
\begin{equation*}
  \textbf{V}(x,y,0) = 
    \begin{cases}
    (0.5, 0.5,-0.5, 5)^\top & x>0.5,\ y>0.5,\\
    (1, 0.5, 0.5, 5)^\top & x<0.5,\ y>0.5,\\
    (3, -0.5,0.5, 5)^\top & x<0.5,\ y<0.5,\\
    (1.5, -0.5, -0.5, 5)^\top &\text{otherwise}. 
    \end{cases}
\end{equation*}
The domain is taken as the unit square $[0,1]^2$ with outflow boundary conditions. 
This problem describes the interaction between four contact discontinuities, resulting in a spiral over time.

We adopt the IP-EOS \eqref{hEOS2} as our EOS and employ a uniform mesh consisting of $400\times400$ grids. The numerical solutions obtained by ES5 at $t = 0.4$ are presented in Figure \ref{Fig:2D_RP1_EOS} as contours for the rest-mass density and pressure logarithms. The results demonstrate that our ES5 can effectively resolve complex 2D relativistic waves.
\end{example}

\begin{figure}[!htb]
	\centering
	\begin{subfigure}[t]{.4\linewidth}
		\centering
		\includegraphics[width=0.95\textwidth]{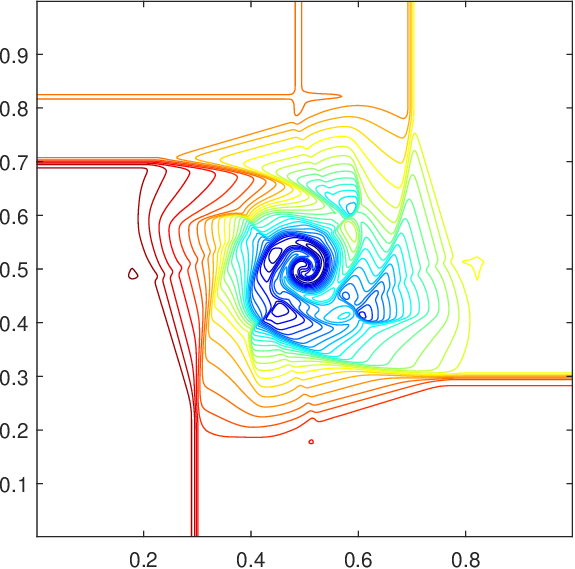}
		\caption{$\ln \rho$}
		\label{Fig:2D_RP1_EOS_density}
	\end{subfigure}
	\begin{subfigure}[t]{.4\linewidth}
		\centering
		\includegraphics[width=0.95\textwidth]{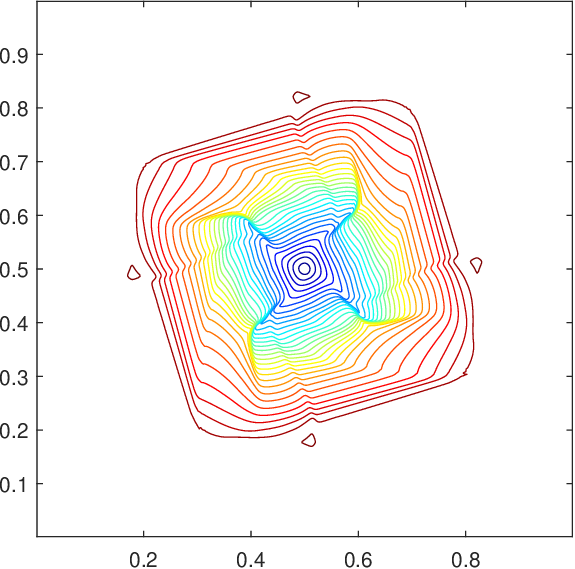}
		\caption{$\ln p$}
		\label{Fig:2D_RP1_EOS_pressure}
	\end{subfigure}
	\caption{Example \ref{Ex5.2.2}: Numerical results obtained by ES5 at $t=0.4$. 30 equally spaced contour lines are displayed.}\label{Fig:2D_RP1_EOS}
\end{figure}

\begin{example}[2D Riemann problem II]\label{Ex5.2.3}
The initial data of the second Riemann problem are given by
\begin{equation*}
  \textbf{V}(x,y,0) = 
    \begin{cases}
    (1, 0, 0, 1)^\top & x>0.5,\ y>0.5,\\
    (0.5771, -0.3529, 0, 0.4)^\top & x<0.5,\ y>0.5,\\
    (1, -0.3529, -0.3529, 1)^\top & x<0.5,\ y<0.5,\\
    (0.5771, 0, -0.3529, 0.4)^\top &\text{otherwise},
    \end{cases}
\end{equation*}
which is about the interaction between four rarefaction waves. 

In this example, TM-EOS \eqref{hEOS3} is adopted as our EOS, and  
 the spatial domain $[0,1]^2$ is divided into $400\times400$ uniform cells. The outflow boundary conditions are specified. Figure \ref{Fig:2D_RP2_EOS} presents the contours of the rest-mass density and pressure logarithms at $t=0.4$ computed by ES5. Our numerical scheme successfully captures the formation of two shock waves resulting from the interaction of four rarefaction waves. 
\end{example}

\begin{figure}[!htb]
	\centering
	\begin{subfigure}[t]{.4\linewidth}
		\centering
		\includegraphics[width=0.95\textwidth]{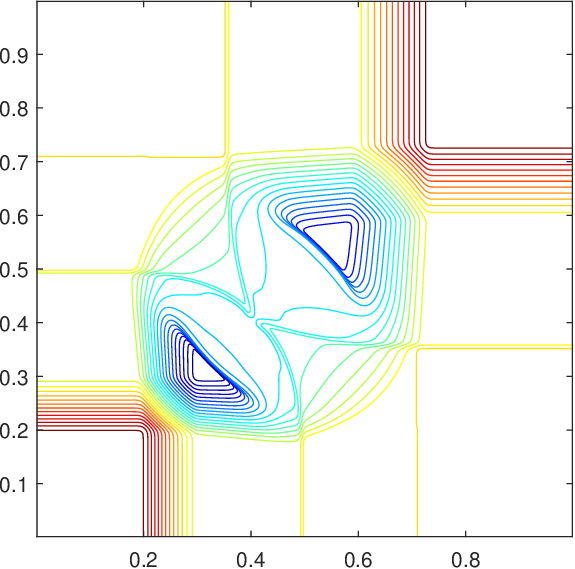}
		\caption{$\ln \rho$}
		\label{Fig:2D_RP2_EOS_density}
	\end{subfigure}
	\begin{subfigure}[t]{.4\linewidth}
		\centering
		\includegraphics[width=0.95\textwidth]{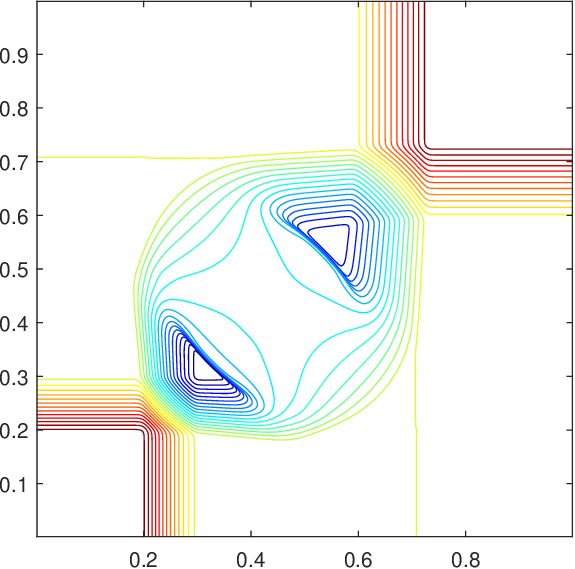}
		\caption{$\ln p$}
		\label{Fig:2D_RP2_EOS_pressure}
	\end{subfigure}
	\caption{Example \ref{Ex5.2.3}: Numerical results obtained by ES5 at $t=0.4$. 30 equally spaced contour lines are displayed.}\label{Fig:2D_RP2_EOS}
\end{figure}

\begin{example}[2D Riemann problem III]\label{Ex5.2.4}
The initial data for the third Riemann problem are given by
\begin{equation*}
  \textbf{V}(x,y,0) = 
    \begin{cases}
    (0.035145216124503, 0, 0, 0.162931056509027)^\top & x>0.5,\ y>0.5,\\
    (0.1, 0.7, 0, 1)^\top & x<0.5,\ y>0.5,\\
    (0.5, 0, 0, 1)^\top & x<0.5,\ y<0.5,\\
    (0.1, 0, 0.7, 1)^\top &\text{otherwise},
    \end{cases}
\end{equation*}
which describe two contact discontinuities and two shock waves, located at two corners. As time progresses, these waves interact with each other and coalesce into a ``mushroom cloud'' at the center of the domain $[0,1]^2$.  
For this problem, we adopt the RC-EOS \eqref{hEOS1} as our EOS and divide the computational domain $[0,1]^2$ into $400\times400$ uniform cells. Additionally, we specify outflow boundary conditions.  
The contours of the rest-mass density and pressure logarithms at $t=0.4$ are displayed in Figure \ref{Fig:2D_RP3_EOS}, which shows the high resolution of the complex wave patterns.
\end{example}

\begin{figure}[!htb]
	\centering
	\begin{subfigure}[t]{.4\linewidth}
		\centering
		\includegraphics[width=0.95\textwidth]{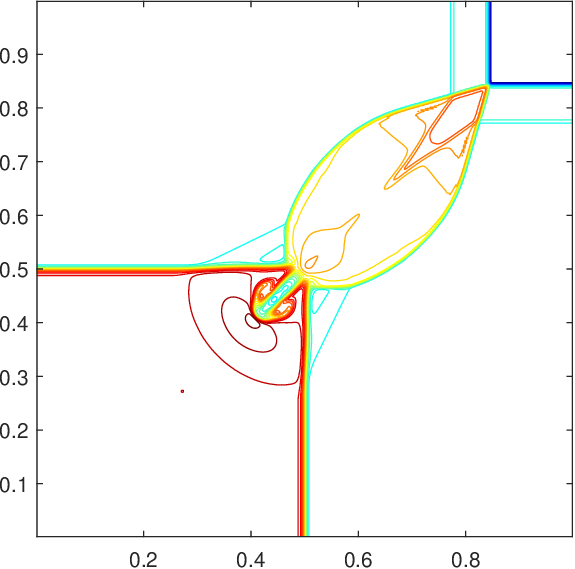}
		\caption{$\ln \rho$}
		\label{Fig:2D_RP3_EOS_density}
	\end{subfigure}
	\begin{subfigure}[t]{.4\linewidth}
		\centering
		\includegraphics[width=0.95\textwidth]{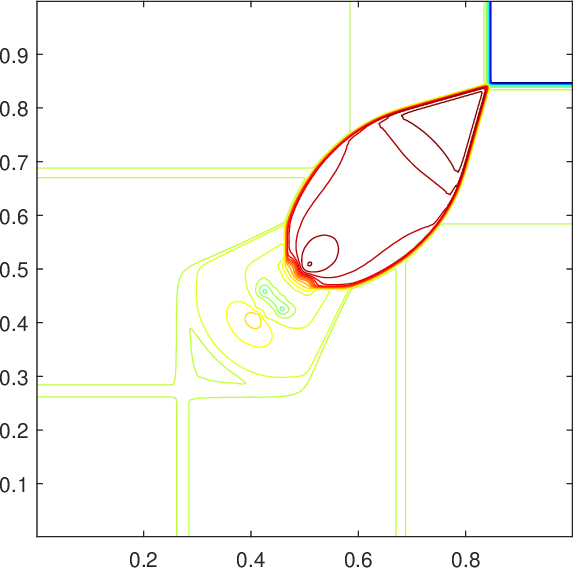}
		\caption{$\ln p$}
		\label{Fig:2D_RP3_EOS_pressure}
	\end{subfigure}
	\caption{Example \ref{Ex5.2.4}: Numerical results obtained by ES5 at $t=0.4$. 30 equally spaced contour lines are displayed.}\label{Fig:2D_RP3_EOS}
\end{figure}

\begin{figure}[!htb]
	\centering
	\includegraphics[width=0.6\textwidth]{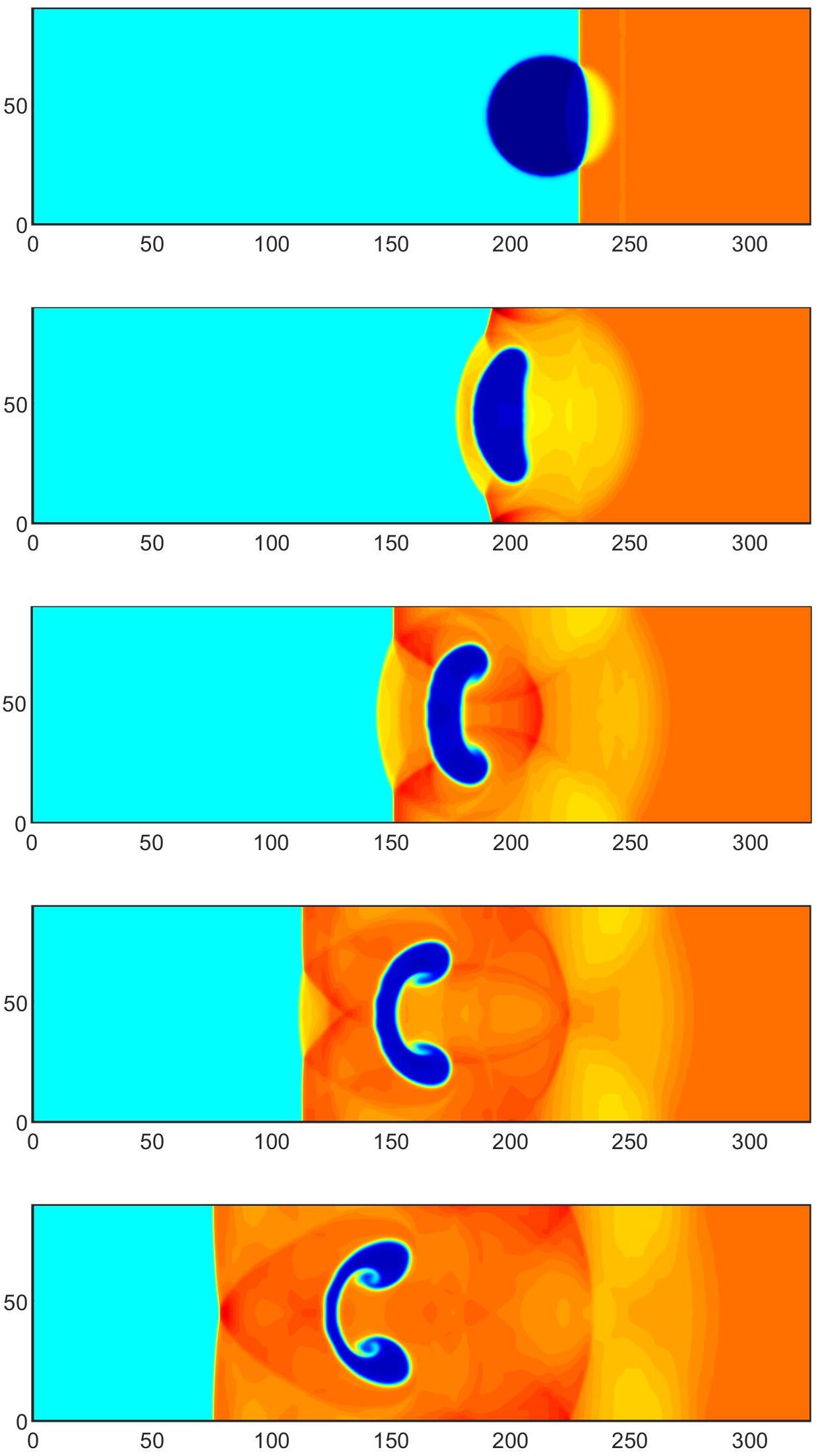}
	\caption{The first problem of Example \ref{Ex5.2.5}: the schlieren images of $\rho$ at $t=90,180,270,360,450$ (from top to bottom). RC-EOS \eqref{hEOS1} is used.}\label{Fig:2D_shockbb1_EOS}
\end{figure}

\begin{example}[Shock-bubble interaction problems]\label{Ex5.2.5}
The last example simulates two shock-bubble interaction problems within the computational domain $[0,325]\times[0,90]$, with reflective boundary conditions at $y=0$ and $y=90$, inflow boundary conditions at $x=325$, and outflow boundary conditions at $x=0$. 
The setups are similar to those in \cite{he2012adaptive1}, but with a different EOS, namely, the RC-EOS \eqref{hEOS1} is used in our setups. 
The left and right states of the shock are set as follows:
\begin{equation*}
  \mathbf{V}(x,y,0) = 
    \begin{cases}
    (1, 0, 0, 0.05)^\top&  x<265,\\
    (1.941272902134272,-0.200661045980881, 0, 0.15)^\top& x>265. 
    \end{cases}
\end{equation*}
We consider two shock-bubble interaction problems, and the setups of these two problems are the same except for the state of the bubble. 
For the first problem, the state of a (light) bubble is given by 
\begin{equation*}
    \textbf{V}(x,y,0) = (0.1358, 0, 0, 0.05)^\top,\quad\quad\quad\sqrt{(x-215)^2+(y-45)^2}\leq 25. 
\end{equation*}
For the second problem, the state of a (heavy) bubble is defined as 
\begin{equation*}
    \textbf{V}(x,y,0) = (3.1538, 0, 0, 0.05)^\top,\quad\quad\quad\sqrt{(x-215)^2+(y-45)^2}\leq 25.
\end{equation*}

\begin{figure}[!htb]
	\centering
	\includegraphics[width=0.6\textwidth]{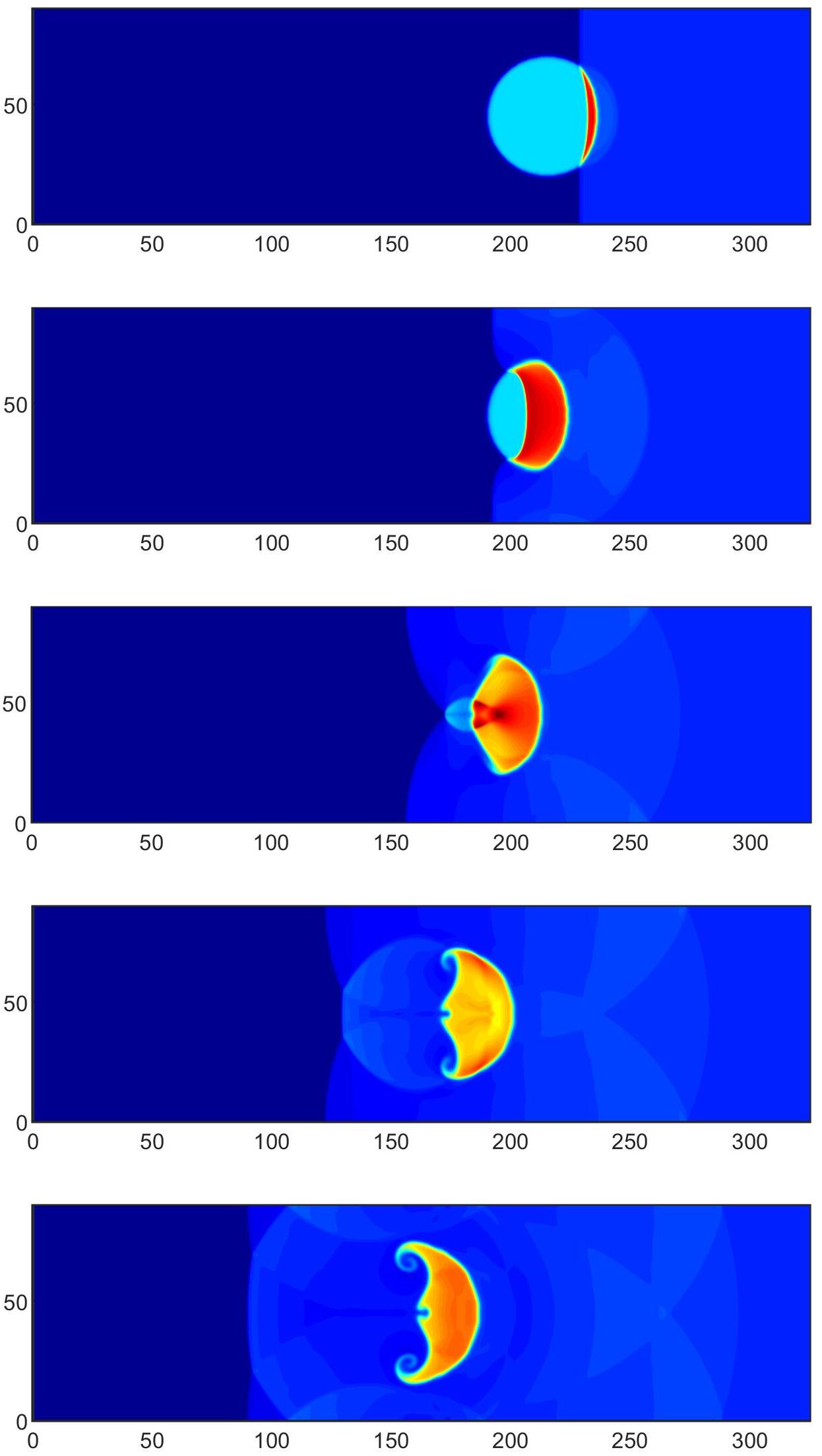}
	\caption{The second problem of Example \ref{Ex5.2.5}: the schlieren images of $\rho$ at $t=90,180,270,360,450$ (from top to bottom). RC-EOS \eqref{hEOS1} is used.}\label{Fig:2D_shockbb2_EOS}
\end{figure}

To visualize the interaction between the shock and the bubble for both problems, we present the schlieren images of the rest-mass density $\rho$ at $t=90,180,270,360,450$ in Figure \ref{Fig:2D_shockbb1_EOS} and Figure \ref{Fig:2D_shockbb2_EOS}. The numerical solutions are obtained using ES5 on $650\times180$ uniform grids. As the figures demonstrate, our scheme effectively captures the dynamics of the interaction between the left-moving shock and the bubble.
\end{example}



\section{Conclusions} \label{section:6}

The ideal EOS, originating from the non-relativistic case, is often a poor approximation for most relativistic flows. 
In this paper, we have made the first attempt to develop high-order ES finite difference schemes for RHD with  general Synge-type EOS \eqref{eq:gEOS}, which covers a wide range of more accurate EOSs. 
  We have discovered an entropy pair for the RHD equations with general Synge-type EOS. 
  We have rigorously proven that the found entropy function is strictly convex and derived the associated entropy variables.  
  However, due to the nonlinear coupling between the RHD equations, it is impossible to explicitly express primitive variables, fluxes, and entropy variables  in terms of conservative variables. As a result, it is   challenging to analyze the entropy structure of the RHD equations, study the convexity of the entropy, and construct EC numerical fluxes. 
  Based on a suitable set of parameter variables, we have constructed novel and explicit two-point EC fluxes in a unified form for general Synge-type EOS. 
  These two-point EC fluxes are used to design second-order EC schemes, and higher-order EC schemes are obtained by linearly combining the two-point EC fluxes. We have achieved arbitrarily high-order accurate ES schemes by adding dissipation terms into the EC schemes, based on ENO or WENO reconstructions. Furthermore, we have derived the general dissipation matrix for general Synge-type EOS based on the scaled eigenvectors of the RHD system. 
  The accuracy and effectiveness of the proposed schemes have been demonstrated through several numerical RHD examples with various special EOSs. 
  
  Our results are also useful for further developing ES discontinuous Galerkin or finite volume schemes for RHD with general EOS and may be helpful for exploring EC and ES schemes for the relativistic MHD equations with Synge-type EOS.

\vspace{5mm}
\noindent
{\bf Data Availability}  
Programming codes and associated data for the numerical examples in Section \ref{section:5} are available at \href{https://github.com/PeterX3AUG1/ESRHD\_gEOS\_source}{https://github.com/PeterX3AUG1/ESRHD\_gEOS\_source}.

\section*{Declarations} 
{\bf Conflict of interest} The authors declare that they have no conflict of interest.

\small
\bibliographystyle{siamplain}
\bibliography{bib}

\end{document}